\documentclass[a4,10pt,twocolumn]{scrartcl}
\usepackage[top=2.5cm, bottom=2.5cm, left=1.6cm, right=1.6cm]{geometry}

\pdfoutput=1

\input{myStandardPreamble_arXiv.tex}
\usepackage[english]{babel}
\usepackage[utf8x]{inputenc}
\usepackage{enumitem}					

\usepackage{helvet}
\usepackage[T1]{fontenc}
\usepackage{mathpazo}

\usepackage{soul}
\usepackage[hang]{footmisc}
\setfnsymbol{wiley}

\makeatletter
\renewcommand{\@biblabel}[1]{[#1]\hfill}
\makeatother

\makeatletter
\let\NAT@parse\undefined
\makeatother

\usepackage[format=plain,			
            labelsep=period,		
            font=small,			
            labelfont=bf,			
            skip=5pt			
            ]{caption}
            
\usepackage[final,kerning=false,protrusion=false]{microtype}

\let\tmp\newinsert
\let\newinsert\newbox
\usepackage{floatrow}
\let\newinsert\tmp
\newfloatcommand{capbtabbox}{table}[][\FBwidth]

\usepackage{authblk}

\newtheorem{theorem}{Theorem}
\newtheorem{lem}{Lemma}
\newtheorem{definition}{Definition}
\newtheorem{rem}{Remark}
\newtheorem{example}{Example}
\newtheorem{assumption}{Assumption}
\newtheorem{corollary}{Corollary}

\newtheorem{property}{Conjecture}

\let\oldTheorem\theorem
\renewcommand{\theorem}{\oldTheorem\normalfont}
\let\oldLemma\lem
\renewcommand{\lem}{\oldLemma\normalfont}
\let\oldCorollary\corollary
\renewcommand{\corollary}{\oldCorollary\normalfont}
\let\oldDefinition\definition
\renewcommand{\definition}{\oldDefinition\normalfont}
\let\oldRemark\rem
\renewcommand{\rem}{\oldRemark\normalfont}
\let\oldAssumption\assumption
\renewcommand{\assumption}{\oldAssumption\normalfont}
\let\oldExample\example
\renewcommand{\example}{\oldExample\normalfont}
\let\oldProperty\property
\renewcommand{\property}{\oldProperty\normalfont}

\addtokomafont{paragraph}{\itshape}

\newcommand{\boldmin}{\mathord{\begin{tikzpicture}[baseline=0ex, line width=2, scale=0.13]	\draw (0,0.7) -- (2,0.7);\end{tikzpicture}}}
\newcommand{\thinmin}{\mathord{\begin{tikzpicture}[baseline=0ex, line width=0.5, scale=0.13]	\draw (0,0.7) -- (2,0.7);\end{tikzpicture}}}
\newcommand{\thinmins}{\mathord{\begin{tikzpicture}[baseline=0ex, line width=0.5, scale=0.13]	\draw (0,0.7) -- (0.7,0.7);\end{tikzpicture}}}
\newlength\figureheight
\newlength\figurewidth

\newcommand{\corrected}[1]{{\color{black}#1}}
\newcommand{\correct}[1]{{\color{black}#1}}

\usepackage{xfrac}
\usepackage{cases}
\usepackage{algorithm} 
\usepackage{algpseudocode}
\usepackage{bm}
\usepackage{bbm}
\usepackage[nice]{nicefrac}
\usepackage{arydshln}
\usepackage{dsfont}
\usetikzlibrary{patterns} 
\usepackage{subfig} 
\usepackage{csvsimple}
\usetikzlibrary{shapes,snakes}
\usepackage{ifthen}
\usepgfplotslibrary{colormaps,fillbetween}

\colorlet{green60}{green!60!black} 
\colorlet{orange90}{orange!90!black} 

\definecolor{green}{rgb}{0,0.5,0}

\newcommand{\mapA}[3]{#1_{#2}^{\sqrt{h}}(#3)}
\newcommand{\map}[2]{#1_{#2}^{\sqrt{h}}}

\newcommand{\allone}[1]{\mathds{1}_{#1}}
\newcommand{\zeros}[1]{0}

\pgfplotsset{select coords between index/.style 2 args={
		x filter/.code={
			\ifnum\coordindex<#1\fi
			\ifnum\coordindex>#2\fi
		}
}}

\newcommand*{\GetElement}[4]{%
	\pgfplotstablegetelem{#2}{#3}\of{#1}%
	\let#4\pgfplotsretval
}

\newcommand{\areapos}{%
	\mathord{
	\begin{tikzpicture}[baseline, scale = 0.2]%
	\draw[draw = green60, fill = green60] (0,0) rectangle (1,1) {};
	\end{tikzpicture}
	}
}

\newcommand{\areaneg}{%
	\mathord{
		\begin{tikzpicture}[baseline, scale = 0.2]%
		\draw[pattern=north west lines, draw = orange90, pattern color = orange90] (0,0) rectangle (1,1) {};
		\end{tikzpicture}
	}
}

\tikzset{cross/.style={cross out, draw, minimum size=2*(#1-\pgflinewidth), inner sep=0pt, outer sep=0pt},
cross/.default={1pt}}

\newcommand{\cross}{%
	\mathord{
		\begin{tikzpicture}[baseline = 0ex, scale = 1]%
		\node[cross = 3pt, thick] at (0,0.09) {};
		\end{tikzpicture}
	}
}

\title{\LARGE \textbf
Gradient Approximation and Multi-Variable Derivative-Free Optimization based on Non-Commutative Maps
}

\setkomafont{section}{\Large}
\setkomafont{subsection}{\large}
\setkomafont{subsubsection}{\itshape}

\begin{document}

\date{}
\author[1]{Jan Feiling}
\author[2]{Mohamed-Ali Belabbas}
\author[3]{Christian Ebenbauer}
\affil[1]{Institute for Systems Theory and Automatic Control, University of Stuttgart, Germany \protect\\ \texttt{\small jan.feiling@ist.uni-stuttgart.de}}
\affil[2]{Coordinated Science Laboratory, University of Illinois at Urbana-Champaign, United States \protect\\ \texttt{\small belabbas@illinois.edu}}
\affil[3]{Chair of Intelligent Control Systems, RWTH Aachen University, Germany \protect\\ \texttt{\small christian.ebenbauer@ic.rwth-aachen.de}
 \protect\\[1em]}

\maketitle

\begin{abstract}
\textbf{Abstract.} In this work, multi-variable derivative-free optimization algorithms for unconstrained optimization problems are developed. A novel procedure for approximating the gradient of multi-variable objective functions based on non-commutative maps is introduced. The procedure is based on the construction of an exploration sequence to specify where the objective function is evaluated and the definition of so-called gradient generating functions which are composed with the objective function such that the procedure mimics a gradient descent algorithm. Various theoretical properties of the proposed class of algorithms are investigated and numerical examples are presented. 
\end{abstract}

\section{INTRODUCTION}\label{sec:intro}
%
%
A key ingredient in the solution of problems arising in machine learning, real-time decision making, and control are sophisticated optimization algorithms. Hence, improving existing optimization algorithms and developing novel algorithms is of central importance in these areas. The optimization problems therein are often very challenging, i.e., they are high-dimensional, non-convex, non-smooth, or of stochastic nature. 
In addition, in some applications the evaluation of the objective to be optimized involves noisy measurements or the mathematical description of the objective is unknown.
For this type of problems, a promising class of algorithms are derivative-free algorithms \cite{conn2009introduction}, which typically need only evaluations of the objective function for optimization. 
Due to the increasing computational power and the generic applicability, derivative-free optimization algorithms have gained renewed interest in recent years, especially in the field of machine learning and control \cite{nesterov2017random,duchi2015optimal,torczon1997convergence,spall1992multivariate, benosman2016learning,wildhagen2018characterizing,
golovin2017google,flaxman2004online,mania2018simple}.
%

%
%
In this paper, we propose a novel class of derivative-free optimization algorithms 
based on a concept introduced in \cite{jfDerFree}. The key idea is to use non-commutative maps to evaluate the objective function at certain points such that the composition of the maps approximates a gradient descent step. 
The class of proposed algorithms is built upon \textit{two main ingredients}: an \textit{exploration sequence} indicating where the objective is to be evaluated, and the \textit{(gradient) generating functions}, which are composed with the objective function in such a way that an approximation of a gradient descent step is obtained.
The resulting algorithms have several noteworthy properties. For example, 
the algorithms are sometimes able to overcome local minima and robust against noisy objective function evaluations.
Such properties are also known from so-called extremum seeking algorithms (cf. e.g. \cite{durr2013lie,ariyur2003real, guay2003adaptive, tan2010extremum}), which are related to our proposed algorithms \cite{wildhagen2018characterizing,jfDerFree}.
%

%
%
In our preliminary work \cite{jfDerFree}, the algorithms were limited to optimization problems with one decision (optimization) variable or to a coordinate-wise application of 
the gradient approximation scheme.
Moreover, only a special case of generating functions were discussed and
no full characterization was given. In another related work \cite{jfESNCM}, the optimization procedure of \cite{jfDerFree} was extended to discrete-time extremum seeking problems, but still limited to one optimization variable.
%

%
%
More broadly related work in terms of gradient approximation schemes are for example finite difference approximations \cite{blum1954multidimensional, kiefer1952stochastic}, simultaneous perturbation stochastic approximations \cite{spall1992multivariate}, and random directions stochastic approximations \cite{kushner2012stochastic}; in \cite{khong2015extremum} those approximation techniques are applied to the aforementioned extremum seeking problems.
These methods are based on so-called sample averaging of function evaluations, i.e., the neighborhood of the current candidate solution is explored to approximate the local gradient of the optimization objective.
In contrast, in the presented work, no numerical differentiation is performed to extract gradient information, instead a kind of numerical integration scheme is utilized to approximate first order information.
%

%
%
The main contribution of this work is fourfold: 
1) a constructive procedure for determining suitable exploration sequences for multi-variable optimization problems is presented,
2) a general class of (gradient) generating functions is characterized, 
3) the so-called single and two-point algorithms in \cite{jfDerFree} and \cite{jfESNCM} are extended to the multi-variable case, 
and 4) a toolbox is developed to easily design and apply the novel class of optimization algorithms to unconstrained optimization problems.
%

%
%
\textit{Notation.} The set of real numbers equal or greater than $k$ is denoted by $\mathbb{R}_{\ge k} = \{ x\in\mathbb{R}\,|\, x\ge k \}$. The class of $k$-times continuously differentiable functions is denoted by $C^k(\mathbb{R}^n;\mathbb{R})$. $I \in \mathbb{R}^{n \times n}$ stands for the $n$-dimensional unit matrix, $\allone{} \in \mathbb{R}^{n}$ for the $n$-dimensional all-one vector, and $e_i \in \mathbb{R}^{n}$ for the \correct{$i$-th} $n$-dimensional unit vector. The matrix $P\in\mathbb{R}^{n\times n}$ has the principal submatrix $P_{1:r}\in\mathbb{R}^{r\times r}$ with $r<n$. 
The bijective mapping $\pi:\{1,\ldots,n\} \rightarrow \{1,\ldots,n\}$ with $n \in \mathbb{N}$ denotes a permutation function. 
A sequence $\corrected{w_{0},\ldots,w_{m-1}}$ of length $m$ is denoted by $\{w_\ell\}_{\ell=0}^{m-1}$. The ceiling and floor operator are defined as $\lfloor x \rfloor := \max\{k\in\mathbb{Z}\,|\,k\le x\}$ and $\lceil x \rceil := \min\{k\in\mathbb{Z}\,|\,k\ge x\}$, respectively. 
A function $f(x;\epsilon): \mathbb{R}^n \times \mathbb{R} → \mathbb{R}^{n}$ is said to be of order $\mathcal{O}(\epsilon)$, if for all compact sets $\mathcal{V}\subseteq \mathbb{R}^n$ there exist an $M\in\mathbb{R}_{>0}$ and $\bar{\epsilon}\in\mathbb{R}_{>0}$ such that for all $x \in \mathcal{V}$ and $\epsilon \in [0,\bar{\epsilon}]$, $|f(x;\epsilon)| \le M \epsilon$. \corrected{The operator $\text{mod}$ takes to integers $k$ and $n$ and returns an integer $k\,\text{mod}\, n$, equal to the remainder of the division of $k$ by $n$.}
A compact set with center point $x^*\in \mathbb{R}^n$ radius $\delta\in\mathbb{R}_{\ge 0}$ and denoted by $\mathcal{U}^{\delta}_{x^*} \subseteq \mathbb{R}^n$ is defined as $\{x \in \mathbb{R}^n:\|x− x^*\|_2 \le \delta \}$.

\section{PROBLEM STATEMENT AND PRELIMINARIES}\label{sec:problem}
\subsection{Problem Statement}
\label{sec:gen_prob}
In this work, we develop a class of algorithms to solve unconstrained minimization problems
\begin{align}
	\min_{x\in\mathbb{R}^n} J(x)
	\label{eq:opt_prob}
\end{align}
for which a closed form expression of $J:\mathbb{R}^n\rightarrow \mathbb{R}$ may be lacking, and only zero-order information in terms of function evaluations are available to find a local minimizer $x^*\in\mathbb{R}^n$ of $J$. The algorithms we propose are of the form 
\begin{align}
	x_{k+1} = \mapA{M}{k}{x_k,J(x_k)},\quad k\ge 0,
	\label{eq:algo}
\end{align} 
where we call $\map{M}{k}:\mathbb{R}^n \times \mathbb{R} \rightarrow \mathbb{R}^n$
 the \textit{transition map}  and $h \in \mathbb{R}_{>0}$ is the step size.
The main idea is to design the transition maps in such a way that for every $k\in\mathbb{N}$, the $m$-fold composition of these maps, i.e., 
\begin{align}
	x_{k+m} = \left( \map{M}{k+m-1} \circ \cdots \circ \map{M}{k}  \right) (x_k,J(x_k))
	\label{eq:composition}
\end{align}
approximates a gradient descent step, i.e.,
\begin{align}
	x_{k+m} = x_k - h\nabla J(x_k) + \mathcal{O}(h^{3/2})
	\label{eq:grad_descent}
\end{align}
as visualized in \Cref{fig1:macu_motivation}.
\begin{figure}[t]
\centering
\begin{tikzpicture}[>=latex,scale = 1.0]
  \pgfdeclarelayer{pre main}
  \pgfsetlayers{pre main,main}
        \def\xa{-1.5};
        \def\ya{5.5};
        \def\xb{2};
        \def\yb{4.5};
        \def\xc{2.5};
        \def\yc{1.9};
        \def\xd{1};
        \def\yd{1.5};
        \def\xe{-1.5};
        \def\ye{3.2};
        \def\sf{12}
        \def\os{3.5}
        \def\ds{-2}
  \begin{axis}[
      hide axis,
      width = 9.5cm,
      domain y = -4:6,
      zmax   = 8,
      colormap/bone,
      xmin = -6,
      xmax = 6,
      ymin = -3,
      ymax = 6
    ]
    \begin{pgfonlayer}{pre main}
        \draw[->, black!99,thick] (axis cs:-5.9,-3,\ds) to (axis cs:4,-3,\ds) {};
        \draw[->, black!99,thick] (axis cs:-5.9,-3,\ds) to (axis cs:-5.9,4.5,\ds) {};
        \draw[->, black!99,thick] (axis cs:-5.9,-3,\ds) to (axis cs:-5.9,-3,9) {};
      \addplot3 [surf, opacity=0.6] {(x^2+y^2)/\sf+\os};
      \addplot3 [surf, opacity=0.03] {\ds+0.9};
      \draw[blue!50,dashed,thick] (axis cs:\xa,\ya,\xa^2/\sf+\ya^2/\sf+\os) to (axis cs:\xa,\ya,\ds) {};
      \draw[blue!50,dashed,thick] (axis cs:\xb,\yb,\xb^2/\sf+\yb^2/\sf+\os) to (axis cs:\xb,\yb,\ds) {};
      \draw[blue!50,dashed,thick] (axis cs:\xc,\yc,\xc^2/\sf+\yc^2/\sf+\os) to (axis cs:\xc,\yc,\ds) {};
      \draw[blue!50,dashed,thick] (axis cs:\xd,\yd,\xd^2/\sf+\yd^2/\sf+\os) to (axis cs:\xd,\yd,\ds) {};
      \draw[blue!50,dashed,thick] (axis cs:\xe,\ye,\xe^2/\sf+\ye^2/\sf+\os) to (axis cs:\xe,\ye,\ds) {};
      \addplot3[mark=square*,blue!90] coordinates {(\xa,\ya,\ds)};
      \addplot3[mark=square*,blue!90] coordinates {(\xb,\yb,\ds)};
      \addplot3[mark=square*,blue!90] coordinates {(\xc,\yc,\ds)};
      \addplot3[mark=square*,blue!90] coordinates {(\xd,\yd,\ds-0.1)};
      \addplot3[mark=square*,blue!90] coordinates {(\xe,\ye,\ds)};
    \end{pgfonlayer}
    \begin{pgfonlayer}{main}
        \addplot3[mark=square*,blue] coordinates {(\xa,\ya,\xa^2/\sf+\ya^2/\sf+\os)};
        \addplot3[mark=square*,blue] coordinates {(\xb,\yb,\xb^2/\sf+\yb^2/\sf+\os)};
        \addplot3[mark=square*,blue] coordinates {(\xc,\yc,\xc^2/\sf+\yc^2/\sf+\os)};
        \addplot3[mark=square*,blue] coordinates {(\xd,\yd,\xd^2/\sf+\yd^2/\sf+\os)};
        \addplot3[mark=square*,blue] coordinates {(\xe,\ye,\xe^2/\sf+\ye^2/\sf+\os)};
        %
        %
        \draw[->, very thick,green60,shorten >= 4pt, shorten <= 4pt] (axis cs:\xa,\ya,\ds) [out=0, in=130] to (axis cs:\xb,\yb,\ds) {};
        \draw[-, very thick,black!80,opacity=0.3,shorten >= 5pt, shorten <= 5pt] (axis cs:\xa,\ya,\xa^2/\sf+\ya^2/\sf+\os) [out=315, in=180] to (axis cs:\xb,\yb,\xb^2/\sf+\yb^2/\sf+\os) {};
        %
        \draw[->, very thick,green60,shorten >= 4pt, shorten <= 5pt] (axis cs:\xb,\yb,\ds) [out=270, in=25] to (axis cs:\xc,\yc,\ds) {};
        \draw[-, very thick,black!80,opacity=0.3, shorten >= 5pt, shorten <= 5pt] (axis cs:\xb,\yb,\xb^2/\sf+\yb^2/\sf+\os) [out=260, in=40] to (axis cs:\xc,\yc,\xc^2/\sf+\yc^2/\sf+\os) {};
        %
        \draw[->, very thick,green60, shorten >= 4pt, shorten <= 4pt] (axis cs:\xc,\yc,\ds) [out=180, in=0] to (axis cs:\xd,\yd,\ds-0.1) {};
        \draw[-, very thick,black!80,opacity=0.3, shorten >= 4pt, shorten <= 4pt] (axis cs:\xc,\yc,\xc^2/\sf+\yc^2/\sf+\os) [out=210, in=0] to (axis cs:\xd,\yd,\xd^2/\sf+\yd^2/\sf+\os) {};
        %
        \draw[->, very thick,green60, shorten >= 4pt, shorten <= 4pt] (axis cs:\xd,\yd,\ds-0.1) [out=170, in=280] to (axis cs:\xe,\ye,\ds) {};
        \draw[-, very thick,black!80,opacity=0.3, shorten >= 4pt, shorten <= 4pt] (axis cs:\xd,\yd,\xd^2/\sf+\yd^2/\sf+\os) [out=150, in=280] to (axis cs:\xe,\ye,\xe^2/\sf+\ye^2/\sf+\os) {};
        %
        \draw[->,very thick,orange90, shorten >= 4pt, shorten <= 4pt] (axis cs:\xa,\ya,\ds) to (axis cs:\xe,\ye,\ds) {};
        \node[] at (axis cs:-3.0,4,\ds) {\color{orange90} ${-\nabla J}$};
    \end{pgfonlayer}
  \end{axis}
\end{tikzpicture}
	\caption[]{
	An illustration of the presented optimization algorithms based on non-commutative maps. Effects of non-commutativity are utilized to approximate the negative gradient of the optimization objective in $m$ steps (i.e., $m=4$ in this illustration).}
	\label{fig1:macu_motivation}
\end{figure}
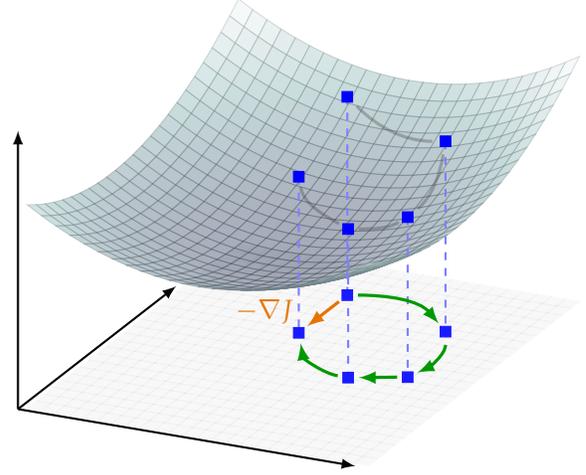
Hereby, we impose the following  structure for the transition maps
\begin{align}
\begin{split}
&\mapA{M}{k}{x_k,J(x_k)} = x_k + \sqrt{h}\alpha_1 s_k\Big(J(x_k)\Big)  
\\
&~~~~~~~~~~~~~~~~\,+ \sqrt{h}\alpha_2 s_k\Big(J\big(x_k+\sqrt{h}s_k(J(x_k))\big)\Big) \\
&s_\ell (J(x_k)) = f(J(x_k))u_\ell+g(J(x_k))v_\ell
\end{split}
\label{eq:TM}
\end{align}
with parameters $\alpha_1,\alpha_2 \in \mathbb{R}$
where $\alpha_1 + \alpha_2 \not = 0$.
We call $s_\ell:\mathbb{R}\rightarrow \mathbb{R}^n$ the \textit{evaluation map}, $f,g:\mathbb{R}\rightarrow \mathbb{R}$ the \textit{generating functions} and $u_\ell,v_\ell\in\mathbb{R}^n$ the $m$-periodic \textit{exploration sequences}. Note that for $\alpha_2 \neq 0$, only two evaluations of $J$ per iterations are necessary. We elaborate on the choice of this structure for the algorithm in the next section.
The main goal of this work is to characterize and design
\begin{enumerate}[leftmargin=13pt]
	\item $m$-periodic exploration sequences $u_\ell$, $v_\ell$, and
	\item gradient generating functions $f$ and $g$,
\end{enumerate}
 such that \eqref{eq:algo} with transition map \eqref{eq:TM} yields \eqref{eq:grad_descent}.

We will at various points make the use of one or both of the following assumptions.
\begin{enumerate}[leftmargin=2.5em, label=(A\arabic*)]
	\item The functions $f,g$ are of class $C^2(\mathbb{R},\mathbb{R})$ and the objective function $J$ is of class $C^2(\mathbb{R}^n,\mathbb{R})$.
	\label{assump:smooth_fg}
	\item The objective function $J$ is radially unbounded and there exists an $x^* \in \mathbb{R}^n$ such that $\nabla J (x)^\top (x-x^*) > 0$ for all $x\in \mathbb{R}^n \backslash \{x^*\}$.
	\label{assump:smooth_J} 
\end{enumerate} 
We note that \ref{assump:smooth_J} will not be required for the design of the algorithms, but only when we analyze their performance.
The implementation of the algorithms, however,  is not limited to the class of objective functions satisfying \ref{assump:smooth_fg}.
%
%
\subsection{Related Results}
\label{sec:existing_results}
The structure of the transition map in \eqref{eq:TM}  was introduced by the authors of the present work in \cite{jfDerFree} for one dimensional problems. Therein,  two cases were considered, specified by the parameter setting $[\alpha_1\ \alpha_2] = [1 \ 0]$, as so-called \textit{single-point} map 
\begin{align}
	\!\mapA{M}{k}{x_k,J(x_k)}\! = \mapA{E}{k}{x_k,J(x_k)}\! := x_k + \sqrt{h}s_k(J(x_k))
	\label{eq:map_E}
\end{align}
and by $[\alpha_1 \ \alpha_2] = [\nicefrac{1}{2}\ \nicefrac{1}{2}]$, as so-called \textit{two-point} map 
\begin{align}
\mapA{M}{k}{x_k,J(x_k)} &= \mapA{H}{k}{x_k,J(x_k)} := x_k + \frac{\sqrt{h}}{2}\left[s_k\Big(J(x_k)\Big) \right. \nonumber \\
 &\left.~+s_k\Big( J(x_k+\sqrt{h}s_k\big(J(x_k)\big) \Big)\right].
 \label{eq:map_H}
\end{align}
The algorithms relying on transition maps \eqref{eq:map_E} and \eqref{eq:map_H} are called single and two-point algorithm, respectively, reflecting that  the number of function evaluations of $J$ at each iteration is one and two. This type of map structure was inspired by the well-known Euler and Heun (trapezoidal) numerical integration methods respectively (thus the naming of the maps $E$ and $H$), i.e., executing a single integration step with step size $\sqrt{h}$ of 
the differential equation
\begin{align}
	\dot{x}(t) = s\big(J(x(t))\big) = f\big(J(x(t))\big)u(t) + g\big(J(x)\big)v(t)
	\label{eq:ode}
\end{align}
with piece-wise constant $m\sqrt{h}$-periodic inputs $u(t),v(t)\in\mathbb{R}^n$ for $t\in[\ell\sqrt{h},(\ell+1)\sqrt{h}]$ with $\ell \in \mathbb{N}$, yields \eqref{eq:map_E} and \eqref{eq:map_H}, respectively.
Note that \eqref{eq:ode} is well known 
as an approximate gradient descent flow in the context of extremum seeking control (cf.~\cite{durr2013lie}). 
For a detailed explanation of the proposed class of algorithms and the continuous-time algorithm \eqref{eq:ode} plus how non-commutativity comes into play, we refer to \cite{jfDerFree,jfESNCM}.
For the coordinate-wise descent case (see Lemma~1 and Lemma~2 in \cite{jfDerFree})  the choice of exploration sequences 
\begin{align}
\label{eq:input_coord}
\begin{split}
	&u_\ell = \bar{u}_\ell e_i,\ \ v_\ell = \bar{v}_\ell e_i \ \ \text{with} \ \
	i = \corrected{\lfloor\ell/4\rfloor\,\text{mod}(n)+1}
	\\ 
	&\bar{u}_\ell = \begin{cases}
	~~~1 &\ell=0 \\
	~~~0 &\ell=1 \\
	-1 &\ell=2 \\
	~~~0 &\ell=3 \\
	\bar{u}_{\ell-4} & \text{else}
	\end{cases}, \quad 
	\bar{v}_\ell = \begin{cases}
	~~~0 &\ell=0 \\
	~~~1 &\ell=1 \\
	~~~0 &\ell=2 \\
	-1 &\ell=3 \\
	\bar{v}_{\ell-4} & \text{else}
	\end{cases}
	\end{split}
\end{align}
with $m=4n$ leads to the evolution of $x_k$ with $[\alpha_1\ \alpha_2]~=~[1\ 0]$ such that
\begin{align}
	x_{k+m} &= x_k + h\Big\lbrace([f,g](J(x_k))  \nonumber \\
	 & - \frac{1}{2}\frac{\partial (f^2+g^2)}{\partial J}(J(x_k)) \Big\rbrace \nabla J(x_k) + \mathcal{O}(h^{3/2})
	 \label{eq:taylor_E}
\end{align}
and with $[\alpha_1\ \alpha_2] = [\nicefrac{1}{2}\ \nicefrac{1}{2}]$ such that
\begin{align}
x_{k+m} &= x_k + h\Big\lbrace([f,g](J(x_k)) \Big\rbrace \nabla J(x_k) + \mathcal{O}(h^{3/2}),
\label{eq:taylor_H}
\end{align}
where $[f,g] := \frac{\partial g}{\partial J}f - \frac{\partial f}{\partial J}g$ is the \textit{Lie bracket} of $f$ and $g$. A simple calculation shows that the term in brackets in \eqref{eq:taylor_E} and \eqref{eq:taylor_H} is identical to $-1$ for $f(J(x)) = \sin(J(x))$ and $g(J(x)) = \cos(J(x))$, hence \eqref{eq:grad_descent} is recovered.
The  exploration sequence above is constructed in such a way that components of the gradient are approximated sequentially for the multi-dimensional setting, hence, coordinate-wise.
In \Cref{fig:ncm}, the exploration sequence and the gradient approximation is visualized
for the scalar case $x_k \in \mathbb{R}$ ($n=1, m=4$).
In summary, the existing procedure is limited and mimics a coordinate-wise descent algorithm.
Further, only a single exploration sequence was presented as well as a single pair of generating functions. 
There are, however, many ways to construct exploration sequences and generating functions, especially in the multi-variable case.
Since different exploration sequences and generating functions lead to different 
properties of the algorithm, it is the goal of this work to provide solutions for a flexible design and constructions of exploration sequences in the multi-variable setting and to characterize a large class of generating functions.
\begin{figure}[t]
	\centering
	\begin{tikzpicture}[>=latex]
	\node[draw=black,circle,name=c1, inner sep = 2pt] at (-0.6,0) {};
	\node[draw=black,circle,name=c2, inner sep = 2pt, above right = 1cm and 1cm of c1.center ] {};
	\node[draw=black,circle,name=c3, inner sep = 2pt,above left = 1cm and 1cm of c2.center ] {};
	\node[draw=black,circle,name=c4, inner sep = 2pt, below left = 0.8cm and 1.1cm of c3.center ] {};
	\node[draw=black,circle,name=c5, inner sep = 2pt, above left = 0.3cm and 0.6cm of c1.center ] {};
	%
	%
	\draw[->,thick,color = green60] (c1) to (c2);
	\draw [->,thick,color = green60] (c2) to [out=100,in=350,looseness=0.8] (c3);
	\draw [->,thick,color = green60] (c3) to (c4);
	\draw [->,thick,color = green60] (c4) to [out=330,in=100,looseness=0.8] (c5);
	\draw [->,thick,color = orange90] (c1) to (c5);

	\node[below = -0.01cm of c1.south ] {\small{$x_0$}};
	\node[right = -0.05cm of c2.east ] {\small{$x_1$}};
	\node[above = -0.05cm of c3.north ] {\small{$x_2$}};
	\node[left = -0.1cm of c4.west ] {\small{$x_3$}};
	\node[left = -0.05cm of c5.west ] {\small{$x_4$}};
	\node[below = 0.7cm of c1.north] {(b)};
	\node[below left = 0.7cm and 3.8cm of c1.north] {(a)};
	
	\node[above right = -0.02cm and 0.45cm of c1.center ] {\small{$M_{0}^{\sqrt{h}}$}};
	\node[above left = 0.45cm and -0.5cm of c2.center ] {\small{$M_{1}^{\sqrt{h}}$}};
	\node[below left = -0.3cm and 0.5cm of c3.center ] {\small{$M_2^{\sqrt{h}}$}};
	\node[below right = -0.08cm and 0.4cm of c4.center ] {\small{$M_{3}^{\sqrt{h}}$}};
	\node[below left = -0.11cm and 0.25cm of c1.center ] {{\color{orange90}\small{$\approx\nabla J(x_0)$}}};
	
	\node[name=1, inner sep = 0pt] at (-6.25,1) {};
	\node[name=2, right = 3.4cm of 1.center ] {};
	\node[name=3, above = 1.5cm of 1.center ] {};
	\node[name=3_1, below = 1.3cm of 1.center ] {};
	\node[name=4, below = 1.0cm of 1.center ] {};
	\node[name=5, above = 0.9cm of 1.center ] {};
	\node[name=6, right = 0.3cm of 5.center ] {};
	\node[fill = blue, circle, inner sep = 2pt] at (6.center){};
	\node[name=7, right = 0.3cm of 1.center ] {};
	\node[cross = 3pt, thick, draw = red] at (7.center){};
	\node[name=8, right = 0.6cm of 7.center ] {};
	\node[fill = blue, circle, inner sep = 2pt] at (8.center){};
	\node[name=9, above = 0.9cm of 8.center ] {};
	\node[cross = 3pt, thick, draw = red] at (9.center){};
	\node[name=10, right = 0.6cm of 8.center ] {};
	\node[cross = 3pt, thick, draw = red] at (10.center){};
	\node[name=11, below = 0.9cm of 10.center ] {};
	\node[fill = blue, circle, inner sep = 2pt] at (11.center){};
	\node[name=22, right = 0.6cm of 10.center ] {};
	\node[fill = blue, circle, inner sep = 2pt] at (22.center){};
	\node[name=23, below = 0.9cm of 22.center ] {};
	\node[cross = 3pt, thick, draw = red] at (23.center){};
	
	\node[name=12, right = 0.3cm of 1.center ] {};
	\node[name=13, below left = 0.5cm and 0.5cm of 12.center ] {};
	\node[name=14, right = 0.6cm of 13.center ] {};
	\node[name=15, above right = 0.5cm and 0.5cm of 14.center ] {};
	\node[name=16, right = 0.6cm of 15.center ] {};
	\node[name=17, above right = 0.5cm and 0.5 of 16.center ] {};
	\node[name=18, right = 0.6cm of 17.center ] {};
	\node[name=19, below left = 0.5cm and 0.5cm of 18.center ] {};
	
	\node[name=20, below left = 0.7cm and 0.7cm  of 1.center ] {};
	\node[name=21, above right = 0.7cm and 0.7cm  of 1.center ] {};
	
	%
	\draw[->] (1.center) to (2);
	\draw[->] (1.center) to (3);
	\draw[-] (1.center) to (3_1);	
	\draw[thick,color=blue] (7.center) to (6.center);
	\draw[thick,color=red] (8.center) to (9.center);
	\draw[thick,color=blue] (10.center) to (11.center);
	\draw[thick,color=red] (22.center) to (23.center);
	
	\node[below left = -0.05cm and -0.15cm of 3.west ] {\small{$u_k,v_k$}};
	\node[below = -0.05cm of 7.south ] {\footnotesize{$0$}};
	\node[below = -0.05cm of 8.south ] {\footnotesize{$1$}};
	\node[below = -0.35cm of 16.north ] {\footnotesize{$2$}};
	\node[below = -0.35cm of 19.north ] {\footnotesize{$3$}};
	\node[below right = -0.1cm and 0.5 cm of 19.south ] {\footnotesize{$k$}};

	\end{tikzpicture}
	\caption[]{(a) periodic inputs $u_k$ ({\color{blue}$\bullet$}) and $v_k$ ({\color{red}$\cross$}) depicted for one period $m=4$ as specified in \eqref{eq:input_coord}; (b) non-commutative maps as in \eqref{eq:TM} with initial point $x_0$.}
	\label{fig:ncm}
\end{figure}
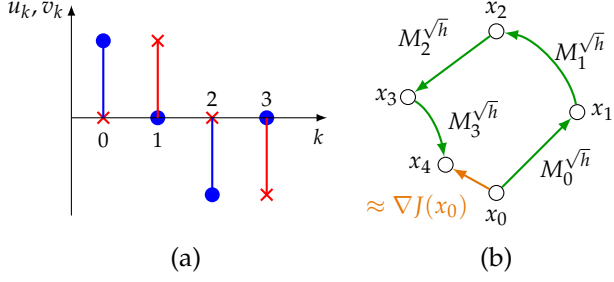
%

\section{MAIN RESULTS}\label{sec:main_part}
\subsection{Problem Statement Reformulation and Convergence}

As described in \cref{sec:gen_prob} we aim to construct $m$-periodic exploration sequences $u_\ell$ and $v_\ell$ and generating functions $f$ and $g$ such that \eqref{eq:algo} with transition map \eqref{eq:TM} yields \eqref{eq:grad_descent}. 
Our first result restates the problem in terms of solving a system of nonlinear equations.
\begin{theorem}
	\label{thm:cond_equ}
     Let \ref{assump:smooth_fg} hold. Then the $m$-th step of the evolution of \eqref{eq:algo} with transition map \eqref{eq:TM} is given by
     \begin{align}
	x_{k+m} &= x_k + \sqrt{h}(\alpha_1+\alpha_2)Y(f(J(x_k)),g(J(x_k)))W\mathds{1}  \nonumber \\
	&+ h \tilde{Y}(f(J(x_k)),g(J(x_k))) T(W) \nonumber \\
	&\times Y(f(J(x_k)),g(J(x_k)))^\top \nabla J(x_k) + \mathcal{O}(h^{3/2}) .
	\label{eq:taylor_matrix}
	\end{align}
	Here, $W = [w_k\ w_{k+1}\ \cdots \ w_{k+m-1}]\in\mathbb{R}^{2n\times m}$ with $w_i = [u_i^\top\ v_i^\top]^\top$ is the exploration sequence matrix and $T(W)\in\mathbb{R}^{2n\times 2n}$
	is given by
	\begin{align}
	&T(W) := \sum_{i=0}^{m-1}  \left( \alpha_2 w_iw_i^\top + (\alpha_1 + \alpha_2)^2 \sum_{j=0}^{i-1} w_iw_j^\top \right).
	\label{eq:cond_equ_gen} 
	\end{align} 
	Furthermore, $Y(f(z),g(z)):= [f(z)I \ \ g(z)I]\in \mathbb{R}^{n \times 2n}$ and $\tilde{Y}(f(z),g(z)) := [\frac{\partial f}{\partial z}(z)I \ \ \frac{\partial g}{\partial z}(z) I]= \frac{\partial}{\partial z}Y(f(z),g(z))\in \mathbb{R}^{n \times 2n}$.
\end{theorem} 
The proof of \cref{thm:cond_equ} is given in \cref{sec:proof_thm_cond_equ}. 
If there exist $m$-periodic exploration sequences $\{w_\ell\}_{\ell=0}^{m-1}$ (equivalently, an exploration sequence matrix $W$), and generating functions $f$ and $g$ such that 
\begin{align}
	\tilde{Y}(f(z),g(z)) T(W) Y(f(z),g(z))^\top &= -I,\ \forall z \in \mathbb{R}
	\label{eq:problem}
	\\
	 W\mathbb{1} &= 0
	\label{eq:cond_equ_gen_zero}
\end{align}
are satisfied, then \eqref{eq:grad_descent} holds.
Thus, this system of nonlinear ordinary differential equations (w.r.t. $f(z)$ or $g(z)$) with unknown coefficients is key in designing the algorithm.
The idea to solve this highly under-determined system of equations is now to proceed in two steps:
\begin{enumerate}[leftmargin=36pt, label=Step \arabic*)]
	\item For a class of normal (skew-symmetric) matrices $T_d$,
	we construct exploration sequence matrices $W$ such that 
	\eqref{eq:cond_equ_gen_zero} and
	$T(W)=T_d$ hold.
    \label{step1}
	\item We characterize gradient generating functions $f,g$ and normal (skew-symmetric) matrices $T_d$ 	such that \eqref{eq:problem} holds.
	\label{step2}
\end{enumerate}
These two constructions are presented in the following two subsections. We start with a remark on $T(W)$ and the convergence result of the proposed algorithms.
\begin{rem}
    \label{rem:T}
	To get a sense of equation	\eqref{eq:problem} and the role of $T(W)$, partition $T(W)$ as 
	\begin{align}
		T(W) = \begin{bmatrix}
			T_{11}(W) & T_{12}(W) \\ T_{21}(W) & T_{22}(W)
		\end{bmatrix},
		\label{eq:structure_T}
	\end{align}
	with $T_{11}(W),T_{12}(W),T_{21}(W), T_{22}(W) \in \mathbb{R}^{n\times n}$. Note that $T_{11}(W)$ is defined solely by $\{u_\ell\}_{\ell = 0}^{m-1}$, $T_{22}(W)$ solely by $\{v_\ell\}_{\ell = 0}^{m-1}$, and $T_{12}(W)$ and $T_{21}(W)$ by both $\{u_\ell\}_{\ell = 0}^{m-1}$ and $\{v_\ell\}_{\ell = 0}^{m-1}$. Then \eqref{eq:problem} with \eqref{eq:structure_T} yields
	\begin{align}
		\frac{\partial f}{\partial J}fT_{11} + \frac{\partial f}{\partial J}gT_{12} + \frac{\partial g}{\partial J}fT_{21} + \frac{\partial g}{\partial J}gT_{22} = -I
		\label{eq:problem_w_structure}
	\end{align}
	where the arguments $J(x_0)$ of the maps $f$ and $g$ and their derivatives and $W$ of $T_{ij}$ with $i,j\in \{1,2\}$ are omitted for the sake of readability. By plugging $\{u_\ell\}_{\ell = 0}^{m-1}$ and $\{v_\ell\}_{\ell = 0}^{m-1}$ from \eqref{eq:input_coord}, with $W_1$ and $W_2$ for $[\alpha_1\ \alpha_2] = [1\ 0]$ and $[\alpha_1\ \alpha_2] = [\nicefrac{1}{2}\ \nicefrac{1}{2}]$, respectively, into \eqref{eq:cond_equ_gen}, one obtains
	\begin{align}
		T(W_1) = \begin{bmatrix}	-I & -I \\ I & -I \end{bmatrix},
		\ \ \text{and} \ \
		T(W_2) = \begin{bmatrix} 0 & -I \\ I & 0 \end{bmatrix}.
		\label{eq:T_examp}
	\end{align} 
	Hence, the left hand side of \eqref{eq:problem_w_structure} translates into the terms in the curly brackets in \eqref{eq:taylor_E} and \eqref{eq:taylor_H}, respectively. A geometric interpretation of T(W) is discussed in \Cref{sec:numerical_results}.
\end{rem}

Due to 	property \eqref{eq:grad_descent}, for example, semi-global practical asymptotic convergence to the optimizer $x^*$ can be established \correct{(see \cite{moreau2000practical})}:
\begin{theorem}
	\label{thm:conv}
	Let \ref{assump:smooth_fg} and \ref{assump:smooth_J} hold. Assume that there exist generating functions $f(J(x))$ and $g(J(x))$ and an exploration sequence matrix $W$ such that \eqref{eq:problem} and \eqref{eq:cond_equ_gen_zero} are satisfied. Then, for all $\delta_1,\delta_2 \in \mathbb{R}_{>0}$ with $\delta_2 < \delta_1$, there exist an $h^* \in \mathbb{R}_{>0}$ and $N(h) \in \mathbb{N}$, such that for all $h\in \{\bar{h}\ |\ 0 < \bar{h} < h^*\}$ and $x_0 \in \mathcal{U}_{x^*}^{\delta_1}$, it holds $x_k \in \mathcal{U}_{x^*}^{\delta_2}$ for all $k\ge N(h)$.
\end{theorem}
The proof of \cref{thm:conv} follows along the lines of the proof of \cite[Theorem 2]{jfDerFree} by utilizing \cref{thm:taylor_gen} in \cref{sec:pre_lem} and \cref{thm:cond_equ}.

\begin{rem}
    \label{rem:var_h}
    \Cref{thm:conv} is based on a constant step size $h$. Applying a variable decreasing step size $h_k$, but constant over a period of length $m$, i.e., $h_0=h_{1}=\cdots=h_{m-1},\ h_{m}=h_{m+1}=\cdots=h_{2m-1},\ \ldots,$ with 
    \begin{align}
        \sum_{p=0}^{\infty} h_{pm} = \infty, \quad  \sum_{p=0}^{\infty} h_{pm}^2 < \infty,\quad 
    \end{align}
    e.g. $h_k = \nicefrac{1}{(\lfloor k/m\rfloor + 1)}$ (cf. Proposition 1 in \cite{spall1992multivariate}) lead to a semi-global asymptotic convergence result and a potential numerical performance improvement. 
    Note that the requirement of periodically $m$ constant steps preserves
    the $\mathcal{O}(\sqrt{h})$-order terms in \eqref{eq:taylor_matrix} (cf. proof of \Cref{thm:taylor_gen}).
\end{rem}

%
\subsection{Exploration Sequences}
In this part we characterize the conditions under which there exists an exploration sequence matrix $W$ for a given $T_d$ such that $T(W) = T_d$ together with $W\mathbb{1}=0$ are satisfied, hence, addressing \ref{step1} as stated above. The next lemma represents
$T(W)$, i.e. \eqref{eq:cond_equ_gen}, in combination with \eqref{eq:cond_equ_gen_zero} in
a more compact form. 
%
%
\begin{lem}
	\label{thm:matrices}
	Consider \eqref{eq:cond_equ_gen} and suppose
	the exploration sequence matrix $W\in \mathbb{R}^{2n\times m}$ satisfies
	\eqref{eq:cond_equ_gen_zero}. Then $T(W)$ can be expressed as
	\begin{align}
	T(W) = W P W^\top ,	\label{eq:w_mat_cond}
	\end{align}
	with $P \in \mathbb{R}^{m\times m}$ defined as
	\begin{align}
	P =
	\begin{bmatrix}
		c_1\! & \! c_2 \! &\! \cdots \!& \!c_2\! &\! 0 \\
		\alpha_2\! & \!\ddots\! & \!\ddots\! & \!\vdots\! & \!\vdots \\
		\vdots\! & \!\ddots\! & \!\ddots & \!c_2\! & \!\vdots  \\
		\alpha_2\! & \!\cdots\! & \!\alpha_2\! & \! c_1\! & \!0 \\
		0\! & \!\cdots\! & \!\cdots\! & \!0\! & \!0
	\end{bmatrix}
	\label{eq:def_mat}
	\end{align}
	where $c_1 = 2\alpha_2 - (\alpha_1+\alpha_2)^2$, $c_2 = \alpha_2-(\alpha_1+\alpha_2)^2$, and $\alpha_1,\alpha_2$ defined in \eqref{eq:TM}.
\end{lem}
The proof of \cref{thm:matrices} is given in \cref{sec:proof_matrices}. 
Consequently, when proceeding according to \ref{step1} and \ref{step2} described in the section above, the key equations for the design of the exploration sequences are %
\begin{align}
	\begin{split}
		WPW^\top &=  T_d \\
		W \allone{} &= 0.
	\end{split}
    \label{eq:svd_prob}
\end{align}
The following theorem, which provides a constructive design of the exploration sequence matrix $W$, is of central importance. It also provides structural insight in terms of obtaining lower bounds on the length (period) of the exploration sequence $m$, suitable choices of the parameters $\alpha_1$ and $\alpha_2$ and admissible structures for 
the desired target matrices $T_d$. 
%
%
\begin{theorem}
	\label{thm:existence_svd}
	Given $\alpha_1, \alpha_2$ and $T_d \in \mathbb{R}^{2n\times 2n}$.
	Suppose that either $T_d$ is normal, $(2\alpha_2-(\alpha_1+\alpha_2)^2) (T_d+T_d^\top)$ positive definite, and \cref{property:interlacing} (see below) is satisfied or that  $T_d$ is skew-symmetric with $2\alpha_2-(\alpha_1+\alpha_2)^2 = 0$. 
	Then there exists a $m\ge \text{rk}(T_d)+1$ such that $W\in\mathbb{R}^{2n\times m}$ satisfies the system of equations \eqref{eq:svd_prob}.
\end{theorem}
%
%

The role of \Cref{property:interlacing} in \Cref{thm:existence_svd} is discussed in \Cref{rem:pm}.
\begin{property}
\label{property:interlacing}
Let the skew-symmetric matrix $C(m) \in \mathbb{R}^{m\times m}$ be defined as 
\begin{align}
\begin{split}
    C(m) &:= A(m) + \epsilon(m+1) B(m)
\end{split}
\label{eq:pm}
\end{align}
\begin{align}
    A(m) &:= 
    \begin{bmatrix}
    0 & 1 & \cdots & 1 \\
    -1 & \ddots & \ddots & \vdots \\
    \vdots & \ddots & \ddots & 1 \\
    -1 & \cdots & \!\!\!\!-1 & 0
    \end{bmatrix} \in \mathbb{R}^{m \times m}, \\
	\begin{split}
	    B(m) &:=
	    \allone{} [0~2~4~\ldots~2(m-1)]\\
	    &-(\allone{} [0~2~4~\ldots~2(m-1)])^T\in \mathbb{R}^{m \times m} 
	\end{split}    
\end{align}
with 
$\epsilon(m) = (m-1)^{-1}(1-m^{-1/2})$.
Then for any $m \ge 2$, $C(m)$ and $C(m+1)$ satisfy the eigenvalue interlacing property (cf. \Cref{thm:fan_pall} in \Cref{sec:pre_lem}). 
\begin{align}
    \omega_{k}^{m+1}>\omega_k^m > \omega_{k+1}^{m+1}\ge 0,
    \label{eq:interlacing}
\end{align}
for $k = 1,\ldots,\lfloor m/2\rfloor$, where $\pm \omega_k^m i$ are the eigenvalues of $C(m)$, with $\{\omega^m_k\}_{k=1\ldots m}$ sorted in non-decreasing order \corrected{in $k$ for $m$ fixed}.
\end{property}
\begin{rem}
\label{rem:pm}
We verified numerically that \Cref{property:interlacing}
is always true (we verified it up
to dimension $m=10000$, see \Cref{thm:interlacing_P}), but a proof is still lacking.
Further notice, that $C(m)$ is part of the following equation:
\begin{align}
    \tilde{P}(m) &= 
    \big(P-\epsilon(m)(\mathds{1}\mathds{1}\!^\top\! P + P\mathds{1}\mathds{1}\!^\top\!) 
    \nonumber \\
    &+\epsilon^2(m) \mathds{1}\mathds{1}\!^\top\! P \mathds{1}\mathds{1}\!^\top\! \big)_{1:m-1} 
    \nonumber \\
    & = \Big( \frac{1}{2}(\alpha_1 + \alpha_2)^2-\alpha_2\Big) I 
    \nonumber \\
    &+\frac{1}{2}(\alpha_1+\alpha_2)^2 C(m-1)
    \label{eq:P_tilde}
\end{align}
with $P$ in \eqref{eq:def_mat} and $\epsilon(m)$ defined in \Cref{property:interlacing}.  Note that the interlacing property \eqref{eq:interlacing} holds also for $\tilde{P}(m)$, since it is arranged by a scaled unit matrix and the skew-symmetric matrix $C(m-1)$ (cf. \cite{connes1998noncommutative}). The 
interlacing property of  $\tilde{P}(m)$
is utilized in the proof of \cref{thm:existence_svd}.
\end{rem}
The proof of \cref{thm:existence_svd} is constructive and presented in \cref{sec:proof_existence_svd}, where in particular in  \eqref{pf:2_cond_normal}, \Cref{property:interlacing} enters. A step-by-step construction of $W$ for a
given $T_d$ is provided in \cref{sec:construction_svd}.
Moreover we obtain as a corollary (which follows by the proof of \cref{thm:existence_svd}):
\begin{corollary}
	\label{thm:W_sym}
	If $2\alpha_2-(\alpha_1+\alpha_2)^2=0$ and $T_d$ skew-symmetric, then there always exists an $W\in\mathbb{R}^{2n\times m}$ with $m=\text{rk}(T_d)+1$. 
\end{corollary}
%
%
%
\begin{rem}
	\label{rem:brockett_int2}
	It is worthwhile to point out an interesting connection between the equations in \eqref{eq:svd_prob} 
and nonlinear control theory, i.e. the controllability of the so-called nonholonomic integrator.
	Suppose $\{w_\ell\}_{\ell = 0}^{m-1}$ is a solution of \eqref{eq:svd_prob}, then 
	it can be verified by direct calculations (see also proof of \cref{thm:matrices} in \cref{sec:proof_matrices}) that it is
	also a solution of the two point boundary value problem 
	\begin{align}
		\begin{split}
		y_0 &= 0,\ \ Z_0 = 0,\ \ y_m = 0,\ \ Z_{m} = T_d\\
			y_{k+1} &= y_k+w_k
			\\
			Z_{k+1} &= Z_k + (\alpha_1+\alpha_2)^2w_k y_k^\top + \alpha_2w_kw_k^\top 
		\end{split}
		\label{eq:brockett_gen2}
	\end{align}
	with $k = 0,\ldots,m-1$, states $y_k \in \mathbb{R}^{2n}$, $Z_k \in \mathbb{R}^{2n \times 2n}$, input $w_k \in \mathbb{R}^{2n}$, and vice versa.
	In particular with $W \allone{} = 0$, i.e., $w_{m-1} = -\sum_{i=0}^{m-2}w_i$, \eqref{eq:brockett_gen2} translates into
	\begin{align}
		\begin{split}
		y_0 &= 0,\ \ Z_0 = 0,\ \ Z_{m-1} = T_d\\
			y_{k+1} &= y_k+w_k\\
			Z_{k+1} &= Z_k + \alpha_2w_ky_k^\top+(\alpha_2-(\alpha_1+\alpha_2)^2)y_kw_k^\top\\
			&+ (2\alpha_2-(\alpha_1+\alpha_2)^2)w_kw_k^\top 
		\end{split}.
		\label{eq:brockett_gen}
	\end{align}

	Considering now the case $[\alpha_1\ \alpha_2] = [\nicefrac{1}{2}\ \nicefrac{1}{2}]$
	shows that \eqref{eq:brockett_gen} is the state-transition of the generalized discrete-time nonholonomic integrator \cite{altafini2016nonintegrable} with given initial and final states. Problem \eqref{eq:brockett_gen2} with $[\alpha_1\ \alpha_2] = [1\ 0]$ has a  similar structure.
	Hence, \Cref{thm:existence_svd} provides an explicit solution to this 
	state transition problem. Moreover, this viewpoint underlines the relationship to non-commutative maps and flows
	as indicated in \Cref{sec:existing_results} (\Cref{fig:ncm}).
	Another, more geometric, interpretation of \eqref{eq:svd_prob}
	is also provided in \Cref{sec:numerical_results}.
\end{rem}
%
%
\subsection{Gradient Generating Functions}
This part addresses \ref{step2}, i.e. solving the (functional)  equation \eqref{eq:problem}
for $f,g$, and \corrected{$T_d \in \mathbb{R}^{2n \times 2n}$} with $T(W)=T_d$. 
First, solutions $(T_d,f,g)$ for the parameter setting $2\alpha_2 - (\alpha_1+\alpha_2)^2 = 0$ are presented. 
\begin{theorem}
	\label{thm:structure_W_H}
    Let $2\alpha_2 -(\alpha_1+\alpha_2)^2=0$ and $T_d$ skew-symmetric, then \eqref{eq:problem} is satisfied by the following triples $(T_d,f,g)$, where $a,b\in\mathbb{R}_{>0}$ and $c, \phi \in \mathbb{R}$:
	\begin{flalign}
		\bullet \quad	&T_d  = \begin{bmatrix}
				0 & -I \\ I & 0
			\end{bmatrix},
			\nonumber \\
			&g(z) = -f(z) \int f(z)^{-2} dz,\ f:\mathbb{R}\rightarrow\mathbb{R}
			\label{eq:cases_H1}
			&&
	\end{flalign}
	\vspace*{-1em}
	\begin{flalign}
		\bullet \quad &T_d  = \begin{bmatrix}
				aQ & -I \\ I & b Q
			\end{bmatrix},\ Q = -Q^\top, 
			\nonumber\\
			&f(z) = a^{-\nicefrac{1}{2}}\sin\Big(\sqrt{ab}z+\phi\Big),
			\nonumber \\ 
			&g(z) = b^{-\nicefrac{1}{2}}\cos\Big(\sqrt{ab}z+\phi\Big)
			\label{eq:cases_H2}
			 &&
	\end{flalign}
	\vspace*{-1em}
	\begin{flalign}
		\bullet \quad &T_d  = \begin{bmatrix}
				aQ & -I \\ I & -b Q
			\end{bmatrix},\ Q = -Q^\top, 
			\nonumber\\
			&f(z) = \pm a^{-\nicefrac{1}{2}}\cosh\Big(\sqrt{ab}z+\phi\Big),
			\nonumber \\ 
			&g(z) = \mp b^{-\nicefrac{1}{2}}\sinh\Big(\sqrt{ab}z+\phi\Big)
			\label{eq:cases_H3}
			 &&
	\end{flalign}
	\vspace*{-1em}
	\begin{flalign}
		\bullet \quad	& T_d  = \begin{bmatrix}
				Q & -I \\ I & 0
			\end{bmatrix},\  Q = -Q^\top, 
			\nonumber \\ 
			&f(z) = \pm \sqrt{a},\ g(z) = \mp \frac{z}{\sqrt{a}}
			\label{eq:cases_H4}
			&&
	\end{flalign}
	\vspace*{-1em}
	\begin{flalign}
		\bullet \quad	& T_d  = \begin{bmatrix}
				0 & -I \\ I & Q
			\end{bmatrix},\  Q = -Q^\top, 
			\nonumber \\ 
			&f(z) = \pm \frac{z}{\sqrt{a}},\ g(z) = \pm \sqrt{a}, 
			\label{eq:cases_H5}
			&&
	\end{flalign}
	\vspace*{-1em}
	\begin{flalign}
		\bullet \quad &T_d  = \begin{bmatrix}
		0 & -I-Q \\ I-Q & 0
		\end{bmatrix},\ Q = -Q^\top, 
		\nonumber \\ 
		&f(z) = \pm\frac{1}{\sqrt{a}}e^{-\frac{a}{2}z},
		\ g(z) = \mp \frac{1}{\sqrt{a}}e^{\frac{a}{2}z},
		\label{eq:cases_H6}
		&&
	\end{flalign}
	\vspace*{-1em}
	\begin{flalign}
		\bullet \quad &T_d  = \begin{bmatrix}
		\corrected{a}Q & -I-cQ \\ I-cQ & \corrected{b}Q
		\end{bmatrix},\ Q = -Q^\top, 
		\nonumber \\ 
		&f(z) = \sqrt{\frac{b}{ab-c^2}}\sin\Big(\sqrt{ab-c^2}z+\phi\Big),
		\nonumber \\
		&g(z) = b^{-\nicefrac{1}{2}}\cos\Big(\sqrt{ab-c^2}z+\phi\Big)
		\label{eq:cases_H7}
		&&
	\end{flalign}
	In addition, for each $T_d$ in \eqref{eq:cases_H1}-\eqref{eq:cases_H7} there exists an $W$, such that $T(W)=T_d$ in \eqref{eq:cond_equ_gen}.
	In~\eqref{eq:cases_H7}, we require that $a,b>c$. 
	%
\end{theorem}
\begin{rem}
    Every pair $f,g$ in \eqref{eq:cases_H2}-\eqref{eq:cases_H7} satisfy \eqref{eq:cases_H1}, hence, these generating functions are valid for the given $T_d$ in \eqref{eq:cases_H1}, too. The advantage of the specified $T_d$'s are discussed in \Cref{sec:numerical_results}. 
\end{rem}
\begin{rem}
    Consider the indefinite integral in \eqref{eq:cases_H1}. Let $F:\mathbb{R}^n\rightarrow\mathbb{R}$ be an anti-derivative of $f(z)^{-2}$. Then so is $F+\bar c$ for any $\bar c \in \mathbb{R}$. Set $g(z) = -f(z)(F(z)+\bar{c}).$ The constant  $\bar{c}$ is chosen such that $g'(z)f(z)-f'(z)g(z)=-1$.
\end{rem}
The proof of \cref{thm:structure_W_H} is given in \cref{sec:p_sructure_W_H}.
%
%
Solutions $(T_d,f,g)$ of \eqref{eq:problem} for the parameter setting $2\alpha_2 - (\alpha_1+\alpha_2)^2 \neq 0$ are presented next. 
\begin{theorem}
	\label{thm:structure_W_E}
	Let \corrected{$T_d \in \mathbb{R}^{2n \times 2n}$} be normal and $(2\alpha_2 -(\alpha_1+\alpha_2)^2) (T_d+T_d^\top)$ be positive definite, then \eqref{eq:problem} is satisfied by the following triples $(T_d,f,g)$, where $r:\mathbb{R}\rightarrow\mathbb{R}_{>0}$, $a\in\mathbb{R}\backslash\{0\}$, $b\in\mathbb{R}_{>0}$, and $\phi \in \mathbb{R}$:
	\begin{flalign}
			\bullet \ \	&T_d  = \begin{bmatrix}
				a I & -I \\ I & a I
			\end{bmatrix},\ a(2\alpha_2 - (\alpha_1+\alpha_2)^2)>0
			\nonumber \\
			&\! f(z) \!=\! \sqrt{r(z)}\sin(\varphi(z)),\ g(z)\! =\! \sqrt{r(z)}\cos(\varphi(z)), 
			\nonumber\\
			&\varphi(z) = \frac{a}{2}\ln(r(z))+\int \frac{1}{r(z)}dz+\phi 
			\label{eq:cases_E1}
			&&
	\end{flalign}
	\vspace*{-1em}
	\begin{flalign}
			\bullet \quad &T_d  = \begin{bmatrix}
				Q & -I \\ I & Q
			\end{bmatrix},
			\nonumber \\
			& (2\alpha_2 - (\alpha_1+\alpha_2)^2)(Q+Q^\top) \text{ pos. def. and normal}, 
			\nonumber \\
			&f(z) = b^{-\nicefrac{1}{2}}\sin\Big(bz+\phi\Big),
			\nonumber \\ 
			&g(z) = b^{-\nicefrac{1}{2}}\cos\Big(bz+\phi\Big)
			\label{eq:cases_E2}
			&&
	\end{flalign}
In addition for every $T_d$ in \eqref{eq:cases_E1}-\eqref{eq:cases_E2} there exists an $W$, such that $T(W)=T_d$ in \eqref{eq:cond_equ_gen}. 
\end{theorem}
The proof of \cref{thm:structure_W_E} is given in \cref{sec:p_sructure_W_E}.
\begin{rem}
    The list of triples $(T_d,f,g)$ in \cref{thm:structure_W_H} is essentially exhaustive, save for some scaled version of the presented cases. A case by case study is presented in the proof of \Cref{thm:structure_W_H} in \Cref{sec:p_sructure_W_H}. Whereas the list of triples $(T_d,f,g)$  in \cref{thm:structure_W_E} is not exhaustive (cf. \Cref{sec:p_sructure_W_E}).
\end{rem}
\cref{thm:structure_W_H} and \cref{thm:structure_W_E} together with \cref{thm:existence_svd} solve \eqref{eq:problem} and \eqref{eq:cond_equ_gen_zero} and thus ensure the existence of a exploration sequence $W$. Hence, a gradient descent step is approximated by the proposed algorithm \eqref{eq:algo} with transition maps \eqref{eq:TM}.

\section{ALGORITHM, PARAMETERS AND NUMERICAL RESULTS}\label{sec:numerical}

In this section we present some numerical studies of the proposed 
class of algorithms.
We carry out simulations and discuss
how various choices of $T_d$,
of the singular values of $W$, of the sequence length (period) $m$, 
or of the parameters $\alpha_1,\alpha_2$ influence the qualitative behavior of the algorithm.
Further, a numerical approach to  construct  the exploration 
sequences using nonlinear programming is presented.

We start by summarizing the 
proposed optimization algorithm and the involved parameters.

\subsection{Algorithm and Parameters}
\label{sec:algo_params}
\corrected{
The design parameters and functions involved in the proposed algorithm are
\begin{enumerate}[leftmargin=*]
	\item map parameters $\alpha_1,\alpha_2 \in \mathbb{R}$ with $\alpha_1 + \alpha_2 \not = 0$,
	\item gradient generating functions $f,g: \mathbb{R} \to \mathbb{R}$,
	\item matrix $T_d$ and exploration sequence matrix $W$, in particular
	singular values $\sigma_{i},\ i = 1,\ldots,\text{rk}(T_d)$ of $W$, 
	\item step size $h>0$.
\end{enumerate}
Hence, algorithm \eqref{eq:algo} with \eqref{eq:TM} is defined in terms of $W$ as follows: 
\begin{algorithm}[h]
	\corrected{
	\caption{Derivative-free optimization algorithm with non-commutative maps}
	\begin{algorithmic}[1]
		\State \textbf{Input:} $x_0$, $h$, $\alpha_1, \alpha_2$, $f(J(\cdot)), g(J(\cdot))$, $T_d$, $\sigma_i\ (i = 1,\ldots,\text{rk}(T_d))$, stop criterion
		\State Calculate $W$ and $m$ as described in \cref{sec:construction_svd}
		\State $k=0$
		\While {stop criterion is not fulfilled}
		\State $\ell = k \bmod (m) + 1$
		\State $e_\ell = [0_{\ell},1,0_{n-1-\ell}]^\top$
		\State $Y(J(x_k)) = [f(J(x_k))I \ g(J(x_k))I]$
		\State $\hat{x}_{k} = x_k + \sqrt{h}Y\big(J(x_k)\big)W e_{\ell}$
		\State $Y(J\big(\hat{x}_{k})\big) = [f(J(\hat{x}_k))I \ g(J(\hat{x}_k))I]$
		\State $x_{k+1} = x_k + \sqrt{h}\big(\alpha_1 Y\big(J(x_k)\big)+\alpha_2 Y(J\big(\hat{x}_{k})\big)\big) We_{\ell} $
		\State $k\leftarrow k+1$
		\EndWhile
		\State \textbf{return} \textit{$[x_{0},x_1,\ldots]$}
	\end{algorithmic}
	\label{alg:alg}}
\end{algorithm}
Notice that $Y(J(x_k))$ is a $n \times 2n$ matrix. Moreover, if $\alpha_2 = 0$, then line \correct{8} and \correct{9} in \cref{alg:alg} can be skipped.}
In the following we discuss the influence of the design parameters on the algorithm's behavior. Additionally, a set of parameters working well in generic situations, which can then be used as a starting point to obtain optimized parameters for a particular application, is provided.
\textit{1. Map parameters ${\alpha_1,\, \alpha_2}$.} 
The parameters weigh $Y(J(x_k)$ and  $Y(J(\hat{x}_k)$ in \cref{alg:alg}, respectively.
In particular, they can be utilized to choose between a single-point ($\alpha_2 = 0$) or a two-point algorithm. 
They are to be normalized according to $\alpha_1+\alpha_2 = 1$ and to tune
 according to the ratio of $\alpha_1,\alpha_2$, while utilizing the step size $h$ to tune the convergence speed. 
Moreover, the choice of $\alpha_1,\alpha_2$ restricts the choice of $T_d$ to be skew-symmetric for $2\alpha_2 \corrected{-} (\alpha_1+\alpha_2)^2 = 0$ and otherwise normal (cf. \cref{thm:structure_W_H} and \cref{thm:structure_W_E}). 
In practice, the parameter sets we found providing the best performance were $[\alpha_1\ \alpha_2] = [1\ 0]$ for the single-point and $[\alpha_1\ \alpha_2] = [\nicefrac{1}{2}\ \nicefrac{1}{2}]$ for the two-point gradient-approximation scheme.
\textit{2. Generating functions ${f,\, g}$.} 
The generating functions comprise a scaling of the  objective evaluated at $x_k$ and $\hat{x}_k$ as stated in \cref{alg:alg}. 
Various choices are presented in \cref{thm:structure_W_H} and \cref{thm:structure_W_E}; depending on $\alpha_1,\alpha_2$. Often we have chosen $f,g$ as sinusoidal functions, since the algorithm showed a very stable behavior
for that cases. 
Note that high function values of $f,g$ or if $f,g$ scale arbitrarily large with $J(x_k)$, the algorithm performs large steps which may cause instabilities. 
In the case of bounded functions, such as sinusoidal functions, (arbitrarily) large steps sizes are avoided.
Further, if $J,f,g$ vanish at a minimum $x^*$, asymptotic convergence to $x^*$ (instead of practical convergence) has been observed in our studies.
\textit{3. Exploration sequence matrix $W$ and $T_d$.} 
The exploration sequence matrix $W$ depends on the choice of $T_d$, specifically on the eigenvalues of $T_d$. 
A step-by-step construction of $W$ based on the algorithm parameters is presented in \cref{sec:construction_svd}. 
As explained in this construction, the singular values of $W$ can be chosen (see \Cref{thm:Td_normal} below), hence this degree of freedom can be used in the algorithm tuning. As shown in numerical examples below, smaller singular values lead to smoother trajectories, but more steps (larger $m$) are needed to perform one gradient approximation step. 
There exists a set of optimal singular values in the sense of minimal number of steps $m$, which is $m=\text{rk}(T_d)+1$. Therewith, the choice of $T_d$ influences the lower bound on $m$ (see \Cref{thm:min_steps} below). 
Note that minimal sequence length $m$ does not always lead to the fastest convergence behaviour. 
\textit{4. Step size $h$.}  
The approximated gradient is scaled with the step size $h$, and hence, $h$ influences the speed of convergence, as well as the area of exploration around $x_k$. As stated in \Cref{thm:conv} there exists an upper bound on $h$ such that \correct{semi-global practical asymptotic convergence (if (A1) and (A2) holds)} is ensured. In our numerical studies, $h$ is often chosen as $0.001 \le h \le 0.5$. 
\begin{corollary}
    \label{thm:Td_normal}
    In the case $2\alpha_2-(\alpha_1+\alpha_2)^2 = 0$, the singular values $\sigma_{2\ell-1},\sigma_{2\ell}$, $\ell = 1,\ldots,\lceil \text{rk}(T_d)/2 \rceil$, of $W$ can  be chosen arbitrarily. Otherwise, the singular values of $W$ have to satisfy $\sigma_{2\ell-1}=\sigma_{2\ell}$ and
    \begin{align}
    	T_d = 
    	\begin{bmatrix}
    		\text{diag}([\gamma_1\ \cdots \gamma_n]) & - I \\
    		I & \text{diag}([\gamma_1\ \cdots \gamma_n])
    	\end{bmatrix},
    	\label{eq:Td_E}
    \end{align}
    with
    \begin{align}
    	\gamma_\ell = \Big(\alpha_2-\frac{1}{2}(\corrected{\alpha_1+\alpha_2})^2\Big)\sigma^2_{2\ell-1}.
    	\label{eq:eval_Td_normal}
    \end{align}
    \end{corollary}
\begin{proof}
    Following directly from the proof of \Cref{thm:existence_svd} in \Cref{sec:proof_existence_svd}. Specifically, $T_d$ in \eqref{eq:Td_E} satisfies \eqref{eq:cases_E2} and \eqref{eq:eval_Td_normal} corresponds to \eqref{pf:2_cond1}.
\end{proof}
\corrected{
\begin{corollary}
	\label{thm:min_steps}
	The minimal number of steps to approximate a gradient according to \eqref{eq:algo} - \eqref{eq:TM} and using the $T_d$'s in \cref{thm:structure_W_H} and \cref{thm:structure_W_E} is $m = n+1$.
	\begin{proof}
		Since the first $n$ rows of each $T_d$ in \cref{thm:structure_W_H} and \cref{thm:structure_W_E} are linearly independent, we know $\min\{\text{rk}(T_d)\} \ge n$. This implies with \Cref{thm:existence_svd} that $m \ge n + 1$. 
	\end{proof}
\end{corollary}}
\corrected{
\begin{rem}
	A gradient step is a approximated in $m = n+1$ steps 
	for $T_d$ as specified in \eqref{eq:cases_H2} and $Q$ with elements $\{q_{ij}\}_{i,j=1}^{2n}$ such that
	\begin{align}
	q_{ij} = \begin{cases}
	~~~1\quad \text{ if } i+j = 2n+1,\ i>j \\
	-1\quad \text{ if } i+j = 2n+1,\ i<j \\
	~~~0\quad \text{ else} 
	\end{cases}
	\end{align}
	holds, while the singular values $\sigma_{2\ell-1},\sigma_{2\ell}$ of $W$ satisfy \eqref{pf:2_cond_skew} for $\ell = 1,\ldots,\lceil n/2 \rceil$.
\end{rem}}
%
%
%
%
\subsection{Numerical Results}
\label{sec:numerical_results}
In the following, various simulation results are presented to illustrate the influence of the algorithm parameters, i.e., matrix $T_d$, singular values $\sigma_k$, $k = 1,\ldots,r$, of $W$ and map parameters $\alpha_1,\alpha_2$.  For the sake of visualization, we focus on examples with $n=2$. 
\begin{figure}[t]
	\centering
	\setlength\figureheight{4.5cm}
	\setlength\figurewidth{4.5cm}
	{\small \input{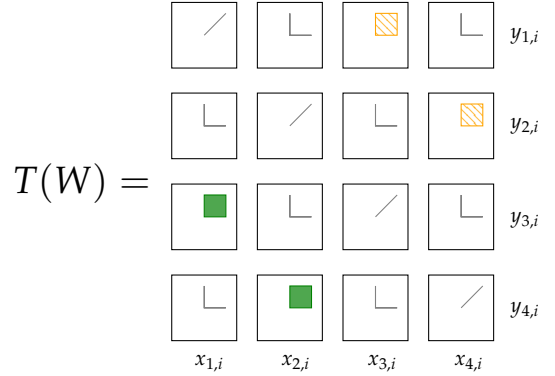}} 
	\caption[]{Generated areas of the exploration sequence $\{w_k\}_{k=0}^{m-1}$ in \eqref{eq:input_coord} for $T(W) \in \mathbb{R}^{2n \times 2n }$ as given in \eqref{eq:TW_pq} for $n = 2$. Hence, $T(W)_{1,3}= T(W)_{2,4} = -1,\  T(W)_{3,1}= T(W)_{4,2} = 1$ and the rest $0$.
	The filled green area ($\areapos$) have a area surface value of $1$, the striped orange area ($\areaneg$) of  $-1$, and the rest of $0$.
	 \correct{The coordinates of each subplot are given by $x_{p,i}$ and $y_{q,i}$ defined in \eqref{eq:area_points} for $p,q=1,\ldots,2n$.}}
	\label{fig:TW_examp}
\end{figure}
\begin{figure*}[t]
	\centering
	\setlength\figureheight{2.8cm}
	\setlength\figurewidth{3.2cm}
	{\small \input{Fig/2D_Lie_1-1_1-1_1_1.tikz}} 
	{\small \input{Fig/2D_Lie_1-5_0-2_1-5_0-2.tikz}} 
	\caption[]{\corrected{An illustration of the algorithm's behavior (setup as in Simulation 1) with $T_d$ as in \eqref{eq:cases_H1} and singular values $[\sigma_1\ \sigma_2\ \sigma_3\ \sigma_4]=[1\ 1\ 1\ 1]$ (top) and $[\sigma_1\ \sigma_2\ \sigma_3\ \sigma_4]=[1.5\ 0.2\ 1.5\ 0.2]$ (bottom) of the exploration matrix $W$. In (a)+(e) the trajectories $x_k$ ({\color{orange90}$\thinmin$}), the first and last $m$ steps ({\color{red}$\boldmin$}), the exact gradient descent algorithm ({\color{blue}$\thinmins\,\thinmins$}), initial state $x_0$ ({\color{green60}$\bullet$}), and optimizer $x^*$ ({\color{green60}$\blacklozenge$}) are depicted. The plots (b)+(f) and (c)+(g) show the exploration sequence $\{u_k\}_{k=0}^{m-1}$ and $\{v_k\}_{k=0}^{m-1}$, respectively, where ({\color{green60}$\bullet$}) is the first and ({\color{orange90}$\cross$}) the second component. In (d)+(h) $T(W)$ in form of the areas generated by $\{w_k\}_{k=0}^{m-1}$ is visualized, where the filled green areas ($\areapos$) have an area surface value of $1$, the striped orange areas ($\areaneg$) of  $-1$, and the rest of $0$.}}
	\label{fig:sim1}
\end{figure*}
\begin{figure*}[t]
	\centering
	\setlength\figureheight{2.8cm}
	\setlength\figurewidth{3.2cm}
	{\small \input{Fig/2D_Lie_2_2.tikz}} 
	{\small \input{Fig/2D_Lie_0-2_0-2.tikz}} 
	\caption[]{\corrected{An illustration of the algorithm's behavior (setup as in Simulation 2) with $T_d$ as in \eqref{eq:cases_H2} and singular values $[\sigma_1\ \sigma_2]=[2\ 2]$ (top) and $[\sigma_1\ \sigma_2]=[0.2\ 0.2]$ (bottom) of the exploration matrix $W$. In (a)+(e) the trajectories $x_k$ ({\color{orange90}$\thinmin$}), the first and last $m$ steps ({\color{red}$\boldmin$}), the exact gradient descent algorithm ({\color{blue}$\thinmins\,\thinmins$}), initial state $x_0$ ({\color{green60}$\bullet$}), and optimizer $x^*$ ({\color{green60}$\blacklozenge$}) are depicted. The plots (b)+(f) and (c)+(g) show the exploration sequence $\{u_k\}_{k=0}^{m-1}$ and $\{v_k\}_{k=0}^{m-1}$, respectively, where ({\color{green60}$\bullet$}) is the first and ({\color{orange90}$\cross$}) the second component. In (d)+(h) $T(W)$ in form of the areas generated by $\{w_k\}_{k=0}^{m-1}$ is visualized, where filled green area ($\areapos$) have an area surface value of $1$, the striped orange area ($\areaneg$) of  $-1$, and the rest of $0$.}}
	\label{fig:sim2}
\end{figure*}
An extensive bench-marking study, including the best choice of parameters for certain classes of the objective, is beyond the scope of this paper and carried out in ongoing and future work. Hence, we keep $f,g,$ (sinusoidal) and $h$ fixed and
provide only a limited number of simulation examples to get some qualitative insight in the degrees of freedom and how they influence the algorithms behavior. 
In the figures below, we show trajectories and the exploration sequences.
In addition, we provide a geometric interpretation and a visualization of the matrix $T(W)$ 
which is of interest by its own and which is explained next. 

The values of the components of $T(W)$ in \eqref{eq:cond_equ_gen} for $[\alpha_1\ \alpha_2] = [\nicefrac{1}{2}\ \nicefrac{1}{2}]$, i.e.,
%
\begin{align}
	T(W)_{pq} = \sum_{i=0}^{m-1}\sum_{j=0}^{i-1} \frac{1}{2} e_p^\top w_i w_i^\top e_q +   e_p^\top w_i w_j^\top e_q , 
	\label{eq:TW_pq}
\end{align}
with \correct{$p,q=1,\ldots,2n$}, where the index $pq$ specifies the element of $T(W)$ in the $p$-th row and $q$-th column,  can be interpreted as the projected areas spanned by the exploration sequences $\{e_p^\top w_\ell\}_{\ell=0}^{m-1},\ \{e_q^\top w_\ell\}_{\ell=0}^{m-1}$. 
The net area $A_{pq}$ of an $n$-sided polygon with corner points $(x_{p,i},y_{q,i}) \in \mathbb{R}^2$, $i=0...n-1$ and \correct{$p,q=1,\ldots,2n$}, known as \textit{Shoelace} or \textit{Gauss} area formula \cite{spivak2018calculus} is obtained as a  special case of Green's Theorem and is given by 
%
\begin{align}
	\correct{A_{pq} = \frac{1}{2}\sum_{i=0}^{n-1} \left(x_{p,i+1}  y_{q,i} - x_{p,i} y_{q,i+1}\right)} 
	\label{eq:area}
\end{align}
with \correct{$x_{p,0} = x_{p,n} = 0$ and $y_{q,0}=y_{q,n} = 0$}. In \Cref{thm:gen_area_skew} (\Cref{sec:pre_lem}) we show that \eqref{eq:area} is equivalent to \eqref{eq:TW_pq} for 
%
\begin{align}
    \correct{x_{p,i} = \sum_{k=0}^{i-1}e_p^\top w_{k},\ y_{q,i} = \sum_{k=0}^{i-1}e_q^\top w_{k}.}
    \label{eq:area_points}
\end{align}
This geometric interpretation is not surprising in the light of 
\Cref{rem:brockett_int2} and the area generating rule appearing in the study of nonholonomic systems \cite{bloch2004nonholonomic}.
In particular, \eqref{eq:area} represents a (double) iterated summation
over the exploration sequence.
In the continuous-time setting, this would correspond to double iterated integrals (or in general $k$-fold iterated integrals, which are called the signature of a path) and which play a fundamental role in nonholonomic control systems.
Note that for $[\alpha_1\ \alpha_2] \neq [\nicefrac{1}{2}\ \nicefrac{1}{2}]$, \eqref{eq:area} with \eqref{eq:area_points} does not hold.
However, we believe it is related to some kind of weighted area.

To illustrate this geometric interpretation, the exploration sequence in \eqref{eq:input_coord} with $[\alpha_1\ \alpha_2] = [\nicefrac{1}{2}\ \nicefrac{1}{2}]$ generates $T(W)$ as depicted in \Cref{fig:TW_examp}, where the singular values of $W$ are $\sigma_1 = \sigma_{2} = \sigma_3 =  \sigma_4 = \sqrt{2}$. Obviously and as presented in the sequel, the singular values of $W$ influence the shape of the areas and give rise to various interpretations of the algorithms behavior. 
\subsubsection*{Simulation 1}
\label{sec:sim1}
In the first simulation setup we consider the objective $J(x) = \|x-[1\ 2]^\top\|^2_2$ and setting $f(J(x)) = \sin(J(x)),\ g(J(x))=\cos(J(x))$, $[\alpha_1\ \alpha_2] = [\nicefrac{1}{2}\ \nicefrac{1}{2}]$, $T_d$ as in \eqref{eq:cases_H1}, $h = 0.05$, and $x_0 = [0\ 1]^\top$. The simulation results for two different choices of singular value pairs are depicted in \Cref{fig:sim1}. On the one hand $\sigma_1=\sigma_2=\sigma_3=\sigma_4=1$ yields $m = 8$ and on the other hand $\sigma_1=\sigma_3=1.5,\ \sigma_2 = \sigma_4= 0.2$ results in $\corrected{m = 21}$. In the latter case the areas with net area value $\pm 1$ have an elongated elliptical shape, which yield a small amplitude of $\{v_k\}_{k=0}^{m-1}$ and therefore a small steady-state amplitude since $\sin(J(x^*))=0$ and $\cos(J(x^*))=1$ (cf. \eqref{eq:TM}). Concluding, the amplitude of $\{e_i^\top w_\ell\}_{\ell=0}^{m-1}$ is proportional to $\sigma_i$ with $i=1,\ldots,2n$ for this choice of $T_d$. 
\subsubsection*{Simulation 2.}
\label{sec:sim2}
In the second simulation study, we consider the same setup as in Simulation 1, but choose $T_d$ as in \eqref{eq:cases_H2} with $Q$ as specified in the proof of \Cref{thm:min_steps}. The simulation results for two different choices of singular values pairs are depicted in \Cref{fig:sim2}. On one hand $\sigma_1 = \sigma_2=2$ results in $m = 4$ and on the other hand $\sigma_1 =  \sigma_2 = 0.2$ in $m = 154$. The singular values can be interpreted as a kind of energy measurement of the exploration sequence, and therefore in the latter case, $m$ is much larger, but reveals a smoother behavior. Moreover, $\sigma_i$ influences the amplitude of $\{e_{2i-1}^\top w_\ell\}_{\ell=0}^{m-1}$ and $\{e_{2i}^\top w_\ell\}_{\ell=0}^{m-1}$ for $i=1,\ldots,n$ for the given choice of $T_d$, compared to the previous simulation example. However, due to space limitations we omit the plots for $\sigma_1 \neq \sigma_2$. 
Interestingly, as observed in \Cref{fig:sim1} and \Cref{fig:sim2} (d)+(h), the areas of $T(W)_{ij}$ for $(i,j) \in \{(1,2),(2,1)\}$ and $(i,j) \in \{(3,4),(4,3)\}$ have the same shape as the last $m$ steps of $x_k$. 
\subsubsection*{Simulation 3.}
\label{sec:sim3}
Consider the same setup as in Simulation 1, but choosing $[\alpha_1\ \alpha_2] = [1\ 0]$ and $T_d$ as described in \eqref{eq:Td_E}. Choosing $\sigma_1 = \sigma_2 = \sigma_3 = \sigma_4 = 1$ results in the behavior depicted in \Cref{fig:sim3} (a)-(c), which is similar to the behavior from \Cref{fig:sim1} (a)-(d), where $m=8$, too. On the one hand, the number of evaluations of the objective $J$ is reduced by half when compared to Simulation 1. On the other hand, in this parameter setup, the choice of singular values is restricted to $\sigma_1 = \sigma_2$ and $\sigma_3 = \sigma_4$, hence, a behavior as in \Cref{fig:sim1} (e)-(h) can not be achieved. However, reducing $\sigma_3 = \sigma_4$ to $0.4$ yields a scaling in the coordinate directions as illustrated in \Cref{fig:sim3} (d)-(f).
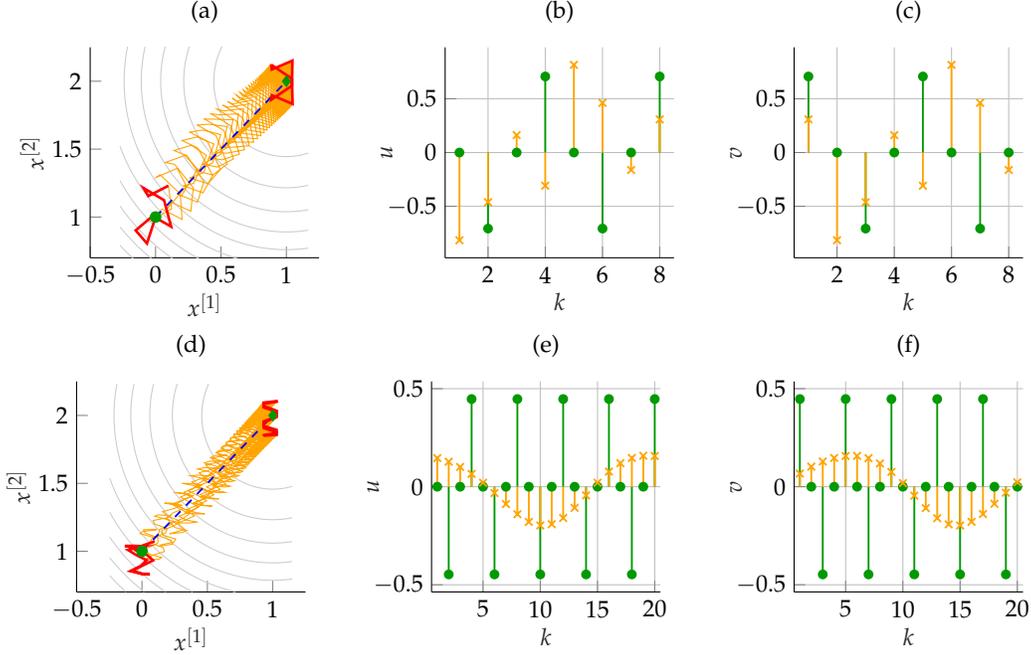
\begin{figure*}[t]
	\centering
	\setlength\figureheight{2.8cm}
	\setlength\figurewidth{3.2cm}
	{\small 
%
\definecolor{mycolor1}{rgb}{1.00000,0.64453,0.00000}%
\subfloat{ 
	\begin{tikzpicture}[baseline]
	\pgfplotstableread[col sep = comma]{Fig/csv/2D_a_1_0_s_1_0-9999_contour.csv}\dataContour
	\pgfplotstableread[col sep = comma]{Fig/csv/2D_a_1_0_s_1_0-9999.csv}\dataX
	\pgfplotstableread[col sep = comma]{Fig/csv/2D_a_1_0_s_1_0-9999_first.csv}\dataXf
	\pgfplotstableread[col sep = comma]{Fig/csv/2D_a_1_0_s_1_0-9999_last.csv}\dataXe
	\pgfplotstableread[col sep = comma]{Fig/csv/2D_a_1_0_s_1_0-9999_points.csv}\dataPoints
	\pgfplotstableread[col sep = comma]{Fig/csv/2D_a_1_0_s_1_0-9999_exact.csv}\dataXexact
	\pgfplotstableread[col sep = comma]{Fig/csv/2D_a_1_0_s_1_0-9999_minmax.csv}\dataMinMax
	\GetElement{\dataMinMax}{0}{[index]0}{\xmin}
	\GetElement{\dataMinMax}{0}{[index]1}{\xmax}
	\GetElement{\dataMinMax}{1}{[index]0}{\ymin}
	\GetElement{\dataMinMax}{1}{[index]1}{\ymax}
	\begin{axis}[%
	width=0.951\figurewidth,
	height=\figureheight,
	at={(0\figurewidth,0\figureheight)},
	scale only axis,
	xmin=\xmin,
	xmax=\xmax,
	xlabel style={font=\color{white!15!black},at={(0.5,0.06)}},
	xlabel={$x^{[1]}$},
	ymin=\ymin,
	ymax= \ymax,
	ylabel style={font=\color{white!15!black},at={(0.15,0.5)}},
	ylabel={$x^{[2]}$},
	axis background/.style={fill=white},
	axis x line*=bottom,
	axis y line*=left,
	title = (a)
	]
	\addplot[contour prepared, contour prepared format=matlab, contour/labels=false, contour/draw color=white!80!black] table[x index = {0}, y index = {1}]{\dataContour};
	\addplot [color=mycolor1, forget plot]
		table[x index = {0}, y index = {1}]{\dataX};
	\addplot [color=blue, line width=0.7pt, dashed, forget plot]
		table[x index = {0}, y index = {1}]{\dataXexact};
	\addplot [color=red, line width=1pt, forget plot]
	table[x index = {0}, y index = {1}]{\dataXf};
	\addplot [color=red, line width=1pt, forget plot]
	table[x index = {0}, y index = {1}]{\dataXe};
	\addplot [color=green60, draw=none, mark=diamond*, mark options={solid, fill=green60, green60}, forget plot, select coords between index={1}{1}]
	table[x index = {0}, y index = {1}]{\dataPoints};
	\addplot [color=green60, draw=none, mark=*, mark options={solid, fill=green60, green60}, forget plot, select coords between index={0}{0}]
	table[x index = {0}, y index = {1}]{\dataPoints};
	\end{axis}
	\end{tikzpicture}
}
\hspace*{1.5em}
\definecolor{mycolor1}{rgb}{1.00000,0.64453,0.00000}%
\subfloat{ 
	\begin{tikzpicture}[baseline]
	\pgfplotstableread[col sep = comma]{Fig/csv/2D_a_1_0_s_1_0-9999_W.csv}\dataW
	\GetElement{\dataW}{0}{[index]0}{\xminW}
	\GetElement{\dataW}{0}{[index]1}{\xmaxW}
	\GetElement{\dataW}{0}{[index]2}{\yminW}
	\GetElement{\dataW}{0}{[index]3}{\ymaxW}
	\GetElement{\dataW}{0}{[index]4}{\maxidx}
	\begin{axis}[%
	width=0.951\figurewidth,
	height=\figureheight,
	at={(0\figurewidth,0\figureheight)},
	scale only axis,
	xmin=\xminW,
	xmax=\xmaxW,
	xlabel style={font=\color{white!15!black},at={(0.5,0.06)}},
	xlabel={$k$},
	ymin=\yminW,
	ymax=\ymaxW,
	ylabel style={font=\color{white!15!black},at={(0.15,0.5)}},
	ylabel={$u$},
	axis background/.style={fill=white},
	axis x line*=bottom,
	axis y line*=left,
	xmajorgrids,
	ymajorgrids,
	title = (b)
	]
	\addplot[ycomb, color=green60, line width=0.7pt, mark=*, mark options={solid, fill=green60, green60, mark size=1.5}, forget plot , select coords between index={1}{\maxidx+1}]
	table[x index = {0}, y index = {1}]{\dataW};
	\addplot[ycomb, color=mycolor1, line width=0.7pt, mark=x, mark options={solid, fill=mycolor1, mycolor1, mark size = 2}, forget plot, select coords between index={1}{\maxidx+1}]
	table[x index = {0}, y index = {2}]{\dataW};
	\end{axis}
	\end{tikzpicture}%
}
%
%
\hspace*{1.5em}
%
\definecolor{mycolor1}{rgb}{1.00000,0.64453,0.00000}%
\subfloat{
	\begin{tikzpicture}[baseline]
	\pgfplotstableread[col sep = comma]{Fig/csv/2D_a_1_0_s_1_0-9999_W.csv}\dataW
	\GetElement{\dataW}{0}{[index]0}{\xminW}
	\GetElement{\dataW}{0}{[index]1}{\xmaxW}
	\GetElement{\dataW}{0}{[index]2}{\yminW}
	\GetElement{\dataW}{0}{[index]3}{\ymaxW}
	\GetElement{\dataW}{0}{[index]4}{\maxidx}
	\begin{axis}[%
	width=0.951\figurewidth,
	height=\figureheight,
	at={(0\figurewidth,0\figureheight)},
	scale only axis,
	xmin=\xminW,
	xmax=\xmaxW,
	xlabel style={font=\color{white!15!black},at={(0.5,0.06)}},
	xlabel={$k$},
	ymin=\yminW,
	ymax=\ymaxW,
	ylabel style={font=\color{white!15!black},at={(0.15,0.5)}},
	ylabel={$v$},
	axis background/.style={fill=white},
	axis x line*=bottom,
	axis y line*=left,
	xmajorgrids,
	ymajorgrids,
	title = (c)
	]
	\addplot[ycomb, color=green60, line width=0.7pt, mark=*, mark options={solid, fill=green60, green60, mark size = 1.5}, forget plot, select coords between index={1}{\maxidx}]
		table[x index = {0}, y index = {3}]{\dataW};
	\addplot[ycomb, color=mycolor1, line width=0.7pt, mark=x, mark options={solid, fill=mycolor1, mycolor1, mark size = 2}, forget plot, select coords between index={1}{\maxidx}]
		table[x index = {0}, y index = {4}]{\dataW};
	\end{axis}
	\end{tikzpicture}%
}} 
	{\small 
%
\definecolor{mycolor1}{rgb}{1.00000,0.64453,0.00000}%
\subfloat{ 
	\begin{tikzpicture}[baseline]
	\pgfplotstableread[col sep = comma]{Fig/csv/2D_a_1_0_s_1_0-4_contour.csv}\dataContour
	\pgfplotstableread[col sep = comma]{Fig/csv/2D_a_1_0_s_1_0-4.csv}\dataX
	\pgfplotstableread[col sep = comma]{Fig/csv/2D_a_1_0_s_1_0-4_first.csv}\dataXf
	\pgfplotstableread[col sep = comma]{Fig/csv/2D_a_1_0_s_1_0-4_last.csv}\dataXe
	\pgfplotstableread[col sep = comma]{Fig/csv/2D_a_1_0_s_1_0-4_points.csv}\dataPoints
	\pgfplotstableread[col sep = comma]{Fig/csv/2D_a_1_0_s_1_0-4_exact.csv}\dataXexact
	\pgfplotstableread[col sep = comma]{Fig/csv/2D_a_1_0_s_1_0-4_minmax.csv}\dataMinMax
	\GetElement{\dataMinMax}{0}{[index]0}{\xmin}
	\GetElement{\dataMinMax}{0}{[index]1}{\xmax}
	\GetElement{\dataMinMax}{1}{[index]0}{\ymin}
	\GetElement{\dataMinMax}{1}{[index]1}{\ymax}
	\begin{axis}[%
	width=0.951\figurewidth,
	height=\figureheight,
	at={(0\figurewidth,0\figureheight)},
	scale only axis,
	xmin=\xmin,
	xmax=\xmax,
	xlabel style={font=\color{white!15!black},at={(0.5,0.06)}},
	xlabel={$x^{[1]}$},
	ymin=\ymin,
	ymax= \ymax,
	ylabel style={font=\color{white!15!black},at={(0.15,0.5)}},
	ylabel={$x^{[2]}$},
	axis background/.style={fill=white},
	axis x line*=bottom,
	axis y line*=left,
	title = (d)
	]
	\addplot[contour prepared, contour prepared format=matlab, contour/labels=false, contour/draw color=white!80!black] table[x index = {0}, y index = {1}]{\dataContour};
	\addplot [color=mycolor1, forget plot]
		table[x index = {0}, y index = {1}]{\dataX};
	\addplot [color=blue, line width=0.7pt, dashed, forget plot]
		table[x index = {0}, y index = {1}]{\dataXexact};
	\addplot [color=red, line width=1pt, forget plot]
	table[x index = {0}, y index = {1}]{\dataXf};
	\addplot [color=red, line width=1pt, forget plot]
	table[x index = {0}, y index = {1}]{\dataXe};
	\addplot [color=green60, draw=none, mark=diamond*, mark options={solid, fill=green60, green60}, forget plot, select coords between index={1}{1}]
	table[x index = {0}, y index = {1}]{\dataPoints};
	\addplot [color=green60, draw=none, mark=*, mark options={solid, fill=green60, green60}, forget plot, select coords between index={0}{0}]
	table[x index = {0}, y index = {1}]{\dataPoints};
	\end{axis}
	\end{tikzpicture}
}
\hspace*{1.5em}
\definecolor{mycolor1}{rgb}{1.00000,0.64453,0.00000}%
\subfloat{ 
	\begin{tikzpicture}[baseline]
	\pgfplotstableread[col sep = comma]{Fig/csv/2D_a_1_0_s_1_0-4_W.csv}\dataW
	\GetElement{\dataW}{0}{[index]0}{\xminW}
	\GetElement{\dataW}{0}{[index]1}{\xmaxW}
	\GetElement{\dataW}{0}{[index]2}{\yminW}
	\GetElement{\dataW}{0}{[index]3}{\ymaxW}
	\GetElement{\dataW}{0}{[index]4}{\maxidx}
	\begin{axis}[%
	width=0.951\figurewidth,
	height=\figureheight,
	at={(0\figurewidth,0\figureheight)},
	scale only axis,
	xmin=\xminW,
	xmax=\xmaxW,
	xlabel style={font=\color{white!15!black},at={(0.5,0.06)}},
	xlabel={$k$},
	ymin=\yminW,
	ymax=\ymaxW,
	ylabel style={font=\color{white!15!black},at={(0.15,0.5)}},
	ylabel={$u$},
	axis background/.style={fill=white},
	axis x line*=bottom,
	axis y line*=left,
	xmajorgrids,
	ymajorgrids,
	title = (e)
	]
	\addplot[ycomb, color=green60, line width=0.7pt, mark=*, mark options={solid, fill=green60, green60, mark size=1.5}, forget plot , select coords between index={1}{\maxidx+1}]
	table[x index = {0}, y index = {1}]{\dataW};
	\addplot[ycomb, color=mycolor1, line width=0.7pt, mark=x, mark options={solid, fill=mycolor1, mycolor1, mark size = 2}, forget plot, select coords between index={1}{\maxidx+1}]
	table[x index = {0}, y index = {2}]{\dataW};
	\end{axis}
	\end{tikzpicture}%
}
%
%
\hspace*{1.5em}
%
\definecolor{mycolor1}{rgb}{1.00000,0.64453,0.00000}%
\subfloat{
	\begin{tikzpicture}[baseline]
	\pgfplotstableread[col sep = comma]{Fig/csv/2D_a_1_0_s_1_0-4_W.csv}\dataW
	\GetElement{\dataW}{0}{[index]0}{\xminW}
	\GetElement{\dataW}{0}{[index]1}{\xmaxW}
	\GetElement{\dataW}{0}{[index]2}{\yminW}
	\GetElement{\dataW}{0}{[index]3}{\ymaxW}
	\GetElement{\dataW}{0}{[index]4}{\maxidx}
	\begin{axis}[%
	width=0.951\figurewidth,
	height=\figureheight,
	at={(0\figurewidth,0\figureheight)},
	scale only axis,
	xmin=\xminW,
	xmax=\xmaxW,
	xlabel style={font=\color{white!15!black},at={(0.5,0.06)}},
	xlabel={$k$},
	ymin=\yminW,
	ymax=\ymaxW,
	ylabel style={font=\color{white!15!black},at={(0.15,0.5)}},
	ylabel={$v$},
	axis background/.style={fill=white},
	axis x line*=bottom,
	axis y line*=left,
	xmajorgrids,
	ymajorgrids,
	title = (f)
	]
	\addplot[ycomb, color=green60, line width=0.7pt, mark=*, mark options={solid, fill=green60, green60, mark size = 1.5}, forget plot, select coords between index={1}{\maxidx}]
		table[x index = {0}, y index = {3}]{\dataW};
	\addplot[ycomb, color=mycolor1, line width=0.7pt, mark=x, mark options={solid, fill=mycolor1, mycolor1, mark size = 2}, forget plot, select coords between index={1}{\maxidx}]
		table[x index = {0}, y index = {4}]{\dataW};
	\end{axis}
	\end{tikzpicture}%
}} 
	\caption[]{\corrected{An illustration of the algorithm's behavior (setup as in Simulation 3) with $T_d$ as in \eqref{eq:Td_E} with singular values $[\sigma_1\ \sigma_2\ \sigma_3\ \sigma_4]=[1\ 1\ 1\ 1]$ (top) and $[\sigma_1\ \sigma_2\ \sigma_3\ \sigma_4]=[1\ 1\ 0.4\ 0.4]$ (bottom) of the exploration matrix $W$. In (a)+(d) the trajectories $x_k$ ({\color{orange90}$\thinmin$}), the first and last $m$ steps ({\color{red}$\boldmin$}), the exact gradient descent algorithm ({\color{blue}$\thinmins\,\thinmins$}), initial state $x_0$ ({\color{green60}$\bullet$}), and optimizer $x^*$ ({\color{green60}$\blacklozenge$}) are depicted. The plots (b)+(e) and (c)+(f) show the exploration sequence $\{u_k\}_{k=0}^{m-1}$ and $\{v_k\}_{k=0}^{m-1}$, respectively, where ({\color{green60}$\bullet$}) is the first and ({\color{orange90}$\cross$}) the second component.}}
	\label{fig:sim3}
\end{figure*}
\subsubsection*{Simulation 4.}
\label{sec:sim4}
\corrected{
In this simulation scenario we consider the cost function $J(x) = \|x-[1\ 2]^\top + 0.5\sin(10\pi x)\|^2$, i.e., an objective with many local minima. We define $f(J(x))=\sin(J(x)),\ g(J(x)) = \cos(J(x)),\ [\alpha_1\ \alpha_2] = [\nicefrac{1}{2}\ \nicefrac{1}{2}],\ h=0.05$, and $x_0 = [1\ 2]^\top$. The simulation results with $T_d$ as in Simulation 2 and $\sigma_1= \sigma_2 =  1$ are depicted in \Cref{fig:sim4}. As motivated in the introduction, the proposed algorithm is able to overcome local minima and converges into a neighborhood of the global minimum. Due to the definition of the maps $M_k^{\sqrt{h}}$, the gradient of $J$ is gained by a procedure similar to numerical integration. It has been observed in simulations that this procedure is numerically more stable than numerical differentiation (finite differences).
In summary, this integrating behavior has the effect to even and flatten out local minima and noise respectively (cf. \cite{wildhagen2018characterizing}).}
\subsubsection*{Simulation 5.}
\label{sec:sim5}
\corrected{One of the key advantages of derivative-free optimization is demonstrated, namely to deal with non-smooth objectives. To this end, in its ability to a simulation experiment, the cost function is set to $J(x) = \|x\|$. We use the same algorithm parameters as in Simulation 2, i.e., $f(J(x))=\sin(J(x)),\ g(J(x)) = \cos(J(x)),\ [\alpha_1\ \alpha_2] = [\nicefrac{1}{2}\ \nicefrac{1}{2}],\ h=0.05$, and $x_0 = [1\ 2]^\top$ with $T_d$ as in \eqref{eq:cases_H2}, $Q$ specified in \cref{thm:min_steps}, and singular values $\sigma_1 = \sigma_2 =  0.4$. The algorithm's behavior is depicted in \cref{fig:sim5}. Since no first-order information (derivatives) needs to be competed, the algorithm is, as expected, converging to the local minima.}
\begin{rem}
	It is worth to mention that a very promising generating function class is 
	\begin{align}
	f(z) = \sqrt{z}	\sin(\text{ln}(z)\mu),\ g(z) = \sqrt{z}	\cos(\text{ln}(z)\mu)
	\end{align}
	with $\mu \in \mathbb{R}_{>0}$, which belong to the setting of \eqref{eq:cases_H1} and \eqref{eq:cases_E1} (with some adaptions). Specifically, the Lie bracket between the generating functions results in $[f,g](z) = -\mu$, i.e., it holds $x_{m+k} = x_k - h\mu\nabla J(x_k) + \mathcal{O}(h^{3/2})$. In this view, $\mu$ can be chosen large and $h$ small, hence, a large enough gradient step is executed while the oscillations can be kept small.
\end{rem}
\subsection{Exploration Sequences via Nonlinear Programming}
Besides the construction of the exploration sequence as described in the proof of \cref{thm:existence_svd} in \cref{sec:proof_existence_svd} and the step-by-step construction of $W$ in \cref{sec:construction_svd} one can compute a sequence with nonlinear programming
by solving the constrained optimization problem
\begin{align}
	\begin{split}
		\min~&~ \correct{\|\text{vec}(W)\|_p^2} \\
		\text{s.t.}~&~ T(W) = T_d \\
		&~ W\mathds{1} = 0
	\end{split}
	\label{eq:nlp}
\end{align}
Instead of the $p-$norm of $W \in \mathbb{R}^{2n\times m}$, one can in principle choose any other objective function, for example
a weighted norm etc. 
In contrast to our constructive approach, the sequence length $m$ with $m\ge \text{rk}(T_d)+1$ has to be specified in \eqref{eq:nlp}. 
An approximation of the lower bound of $m$ can be computed by following Step 4 of the step-by-step construction procedure for the exploration sequence matrix (see \cref{sec:construction_svd}) and choosing the singular values such that $m$ is minimal. 

\section{CONCLUSION}\label{sec:conclusions}

In this work, we proposed a novel class of derivative-free optimization algorithms.
The idea was to approximate the gradient of the objective function
by a $m$-fold composition of maps. These maps are defined by exploration sequences and
generating functions. We provided a general framework for the construction of those ingredients.
In particular, the construction of exploration sequences
is related to nonholonomic state-transition problems and 
is based on solving a system of quadratic
equations, which we encountered by a singular value decomposition (see \Cref{thm:existence_svd}).
The characterization of the generating functions was carried out by solving the 
functional equation \eqref{eq:problem} (see \Cref{thm:structure_W_H} and \Cref{thm:structure_W_E}).
Numerical simulations and a qualitative study of the dynamics of the algorithm 
were presented and the role of the algorithm parameters on the behavior
of the algorithm was discussed. It turned out that the singular values of
the exploration sequence matrix play a crucial role.
Due to space limitations, we leave an extensive benchmarking study and comparisons with other derivative-free optimization algorithms for follow up work. 
The tuning of the algorithm parameters---which entails a choice of exploration sequences and generating functions, and balancing exploration and exploitation by
proper step size rules or line search methods---requires systematic and intensive testing for suitable classes of objective functions, which is beyond the scope of this paper. 
Eventually, the tuning of the parameters can be approached, for example, by learning exploration sequences  based  on a training set of relevant objective functions using an hyperparameter optimization approach.
Another future research direction is the extension of the proposed algorithm to extremum seeking problems (cf. \cite{jfESNCM}) for the two-point algorithmic scheme. 
Finally, designing exploration sequences plays a key role in our algorithms. This corresponds to the problem of finding sequences such that the first and second iterated summations, i.e. the one- and two-dimensional projected areas,
have the values specified on the right hand sides of \eqref{eq:svd_prob}. 
To the best of our knowledge, a \emph{general} and \emph{algorithmic} characterization of solutions to this inverse problem (i.e. given signature values and find corresponding paths) is not known and in our opinion it is an interesting mathematical research question by its own \cite{lyons2018inverting}.
\begin{figure}[h]
	\vspace*{-1em}
	\centering
	\includegraphics[width=0.9\textwidth]{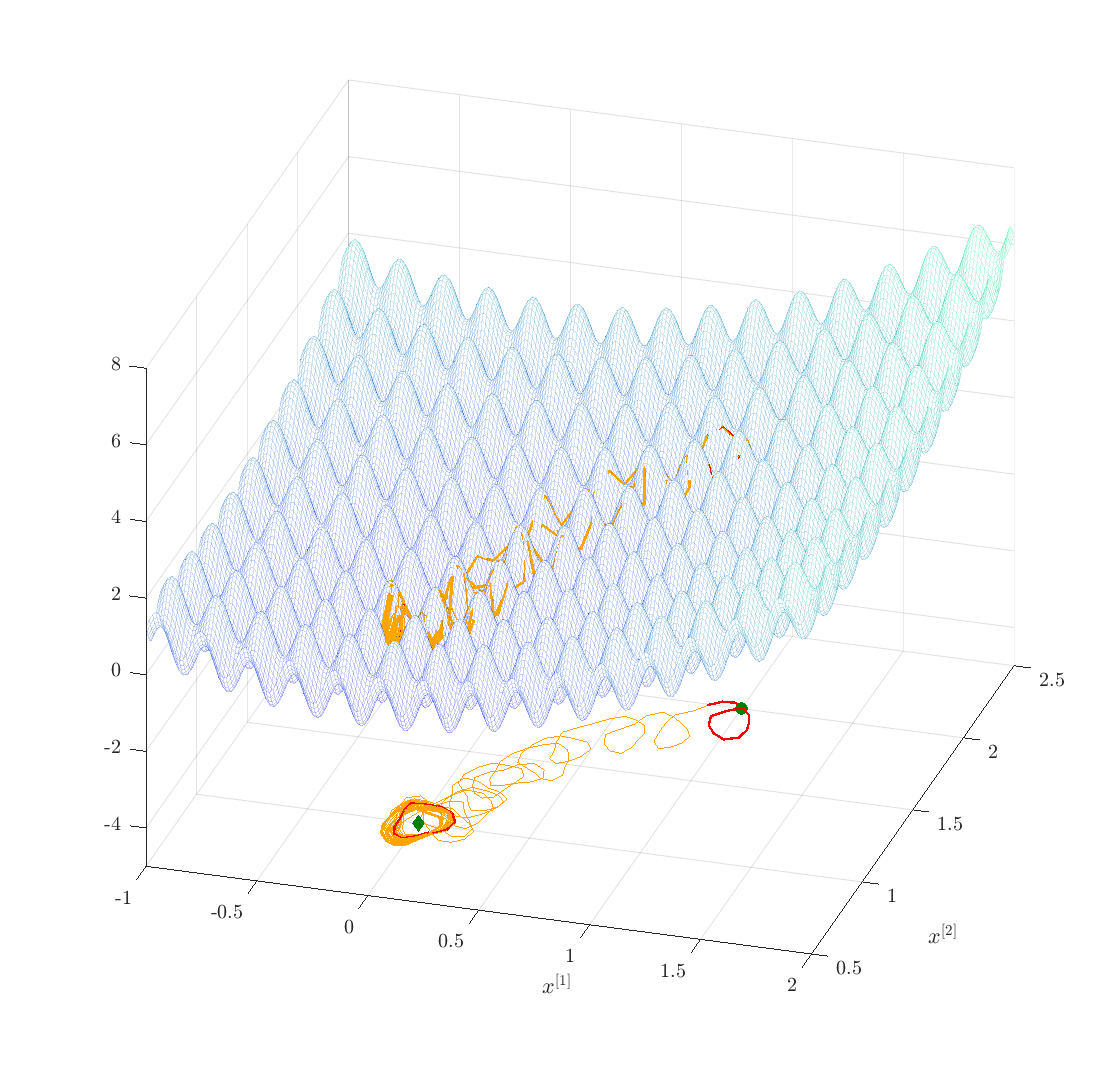} 
	\caption[]{\corrected{An illustration of the algorithms behavior (setup as in Simulation 4) for $T_d$ as in \eqref{eq:cases_H2} and singular values $[\sigma_1\ \sigma_2] = [1\ 1]$ of exploration matrix $W$. The trajectories $x_k$ ({\color{orange90}$\thinmin$}), the first and last $m$ steps ({\color{red}$\boldmin$}), the initial state $x_0$ ({\color{green60}$\bullet$}), and the optimizer $x^*$ ({\color{green60}$\blacklozenge$}) are depicted.}}
	\label{fig:sim4}
\end{figure}
\begin{figure}[h]
	\vspace*{-1em}
	\centering
	\includegraphics[width=0.9\textwidth]{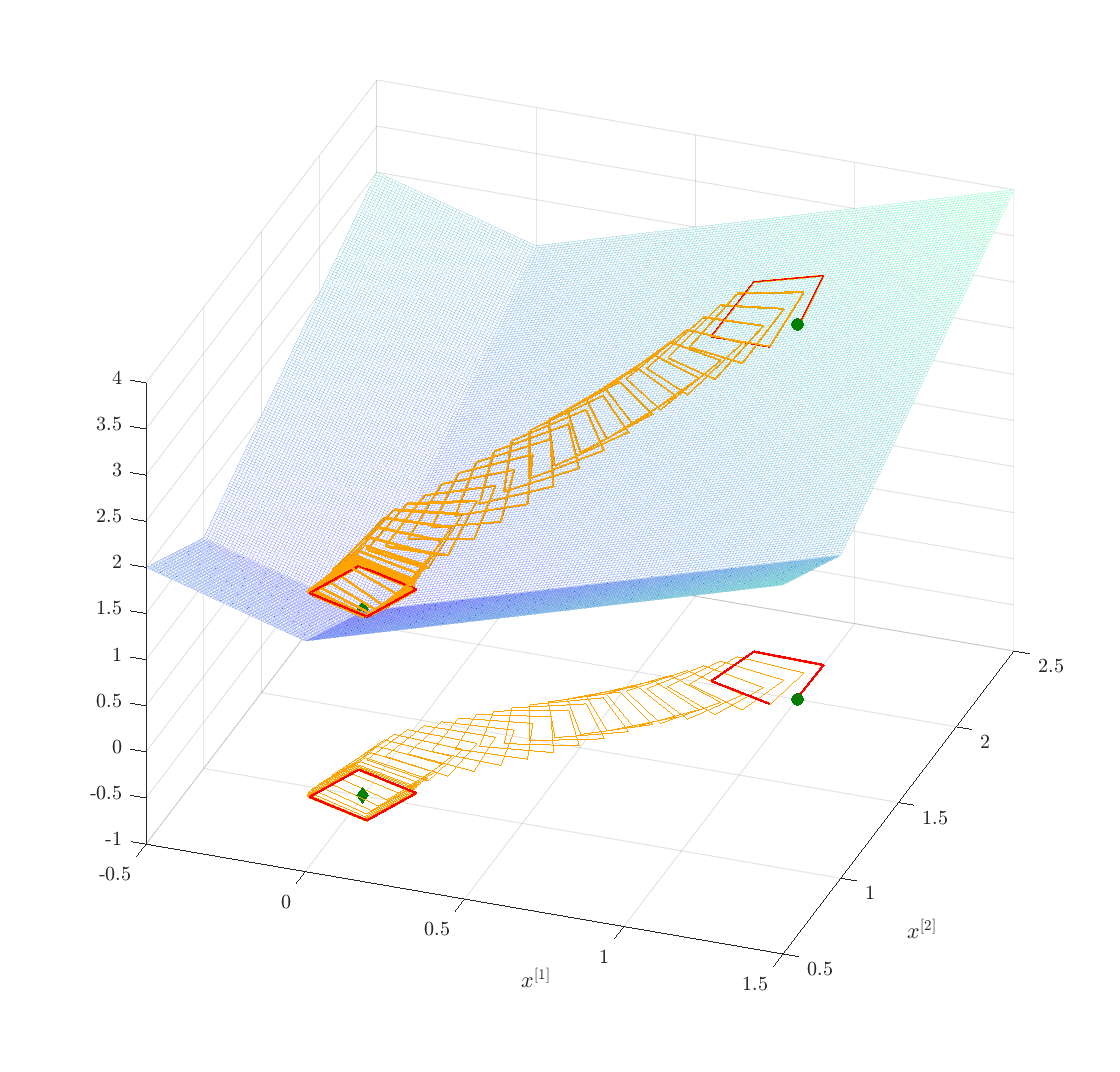} 
	\caption[]{\corrected{An illustration of the algorithm's behavior (setup as in Simulation 5) for $T_d$ as in \eqref{eq:cases_H2} and singular values $[\sigma_1\ \sigma_2] = [2\ 2]$ of exploration matrix $W$. The trajectories $x_k$ ({\color{orange90}$\thinmin$}), the first and last $m$ steps ({\color{red}$\boldmin$}), the initial state $x_0$ ({\color{green60}$\bullet$}), and the optimizer $x^*$ ({\color{green60}$\blacklozenge$}) are depicted.}}
	\label{fig:sim5}
\end{figure}

\bibliographystyle{IEEEtran}
\bibliography{IEEEabrv,bibfile}

\appendix
\section{PRELIMINARY LEMMAS}
\begin{lem}
	\label{thm:taylor}
	Let $a \in C^2(\mathbb{R}^n;\mathbb{R})$, $b\in C^0(\mathbb{R}^p; \mathbb{R}^n)$, and $h\in \mathbb{R}_{\ge 0}$.
	Then for any compact convex set $\mathcal{Z}\subseteq\mathbb{R}^n$ and any compact set $\mathcal{Y}\subseteq\mathbb{R}^p$ 
	there exist a $R \in C^0(\mathbb{R}^n \times \mathbb{R}^p \times \mathbb{R}_{\ge 0}; \mathbb{R})$ and a $M\in\mathbb{R}_{\ge 0}$ such that for all $z,\ z+hb(y)\in\mathcal{Z}$ and $y \in \mathcal{Y}$ we have
	\begin{align}
	a(z + hb(y)) &= a(z) + h \frac{\partial a}{\partial z}(z)^\top b(y) + R(z,y;h^2)
	\label{eq:taylor}
	\end{align}
	with $|R(z,y;h^2)| \le M h^2$, i.e., $\lim_{h \rightarrow 0} R(z,y;h^2)=0$.
\end{lem}
\begin{proof}
	Equation \eqref{eq:taylor} is obtained by applying Taylor's theorem \cite[Theorem 5.15, p. 110]{rudin1964principles} up to degree two, thus, there exits a $\theta \in [0,1]$ such that
	\begin{align}
	a(z + hb(y)) &= a(z) + h \frac{\partial a}{\partial z}(z)^\top b(y)  \nonumber \\ &+\frac{h^2}{2} b(y)^\top \frac{\partial^2 a}{\partial z}(\bar{x}) b(y) 
	\end{align}
	holds with $\bar{x} = z + \theta h b(y)$. The term 
	\begin{align}
	R(z,y;h^2) = h^2 b(y)^\top \frac{\partial^2 a}{\partial^2 z}(\bar{x}) b(y)
	\label{eq:tayler_R}
	\end{align}
	is the Lagrange remainder where $\nicefrac{\partial^2 a}{\partial^2 z}(\cdot)$ is the Hessian of $a(\cdot)$.
	Since $\mathcal{Z}$  and $\mathcal{Y}$ are compact and $b \in C^0(\mathbb{R}^p;\mathbb{R}^n)$ there exists a $M_b \in \mathbb{R}_{\ge  0}$ such that
	\begin{align}
	\|b(y)\|_2 \le M_b,\quad y\in \mathcal{Y}
	\end{align}
	holds. Furthermore, since $\mathcal{Z}$ is convex and compact and $a\in C^2(\mathbb{R}^n;\mathbb{R})$ there exists a $M_a \in \mathbb{R}_{>0}$ such that 
	\begin{align}
	\|\frac{\partial^2 a}{\partial z^2}(\bar{x})\|_2 \le M_a,\quad \bar{x} \in \mathcal{Z}.
	\end{align}
	Finally, $|R(z,y;h^2)| \le M h^2$ with $M=M_b^2M_a$.  
\end{proof}
\label{sec:pre_lem}
\begin{lem}
\label{thm:taylor_gen}
Let Assumption \ref{assump:smooth_fg} hold true.
Moreover, let $\mathcal{X} \subseteq \mathbb{R}^n$ and $\mathcal{J} \subseteq \mathbb{R}$
be compact convex sets, and $m\in\mathbb{N}_{\ge 1}$.
Then there exist a function $R_{k+m-1}:\mathbb{R}^n\times\mathbb{R}\times\mathbb{R}_{\ge 0} \rightarrow \mathbb{R}^n$ and a constant $M_{k+m-1} \in \mathbb{R}_{\ge 0}$ such that 
for any iterates $x_k,\ldots,x_{k+m}$ of the algorithm \eqref{eq:algo} with maps in \eqref{eq:TM}, $x_t,x_t+\sqrt{h} s_k(J(x_t)) \in \mathcal{X}$, and $J(x_t),J(x_t + \sqrt{h} s_k(J(x_t))) \in \mathcal{J}$ for $t=k,\ldots,k+m$, and
we have
\begin{align}
&x_{k+m} = x_k + \sqrt{h} (\alpha_1+\alpha_2) \sum_{i=k}^{k+m-1} s_i(J(x_k)) \nonumber \\
&+ h\alpha_2\sum_{i=k}^{k+m-1} \frac{\partial s_i}{\partial J}(J(x_k)) s_i(J(x_k))^\top \nabla J(x_k)
\nonumber \\
&+ h(\alpha_1+\alpha_2)^2\!\sum_{i=k}^{k+m-1} \sum_{j=k}^{i-1} \frac{\partial s_i}{\partial J}(J(x_k))   s_j(J(x_k))\!^\top\! \nabla J(x_k) \nonumber \\
&+R_{m-1}(x_k,J(x_k);h^{3/2}), 
\label{eq:taylor_gen}
\end{align}
with  $\|R_{k+m-1}(x_k,J(x_k);h^{3/2})\|_2 \le M_{k+m-1} h^{3/2}$, i.e., 
$R_{k+m-1}(x_k,J(x_k);h^{3/2}) = \mathcal{O}(h^{3/2})$. 
\end{lem}

\begin{proof}
W.l.o.g we set $k=0$, i.e., we show by induction that the $m$-step evolution of \eqref{eq:algo} with transition map \eqref{eq:TM} is give by \eqref{eq:taylor_gen} with $k=0$. 
Similarly to  $R(\cdot,\cdot;\cdot)$ in \eqref{eq:taylor},
we introduce the following notation of the
\textit{Taylor remainder (T.R.)} terms for  $k=0,\ldots,m-1$:
\begin{itemize}[leftmargin=*]
	\item $R_{J,k}(\cdot,\cdot;h)$ of \textit{T.R.} of $J(x_k)$ 
	\item $R_{s,k}(\cdot,\cdot;h)$ of \textit{T.R.} of $s_k(J(x_k))$
	\item $R^+_{J,k}(\cdot,\cdot;h)$ of \textit{T.R.} of $J(x_k+\sqrt{h}s_k(J(x_k)))$ 
	\item $R^+_{s,k}(\cdot,\cdot;h)$ of \textit{T.R.} of $s_k(J(x_k+\sqrt{h}s_k(J(x_k))))$
\end{itemize}
and  aggregated remainders with terms of order $h$ or $h^{3/2}$ and higher for $k=0,\ldots,m-1$:
\begin{itemize}[leftmargin=*]
	\item $R_k(\cdot,\cdot;h^{3/2})$ of $x_k$
	\item $\bar{R}_{J,k}(\cdot,\cdot;h)$ of $J(x_k)$ 
	\item $\bar{R}_{s,k}(\cdot,\cdot;h)$ of $s_k(J(x_k))$
	\item $\bar{R}^+_{J,k}(\cdot,\cdot;h)$ of $J(x_k+\sqrt{h}s_k(J(x_k)))$ 
	\item $\bar{R}^+_{s,k}(\cdot,\cdot;h)$ of $s_k(J(x_k+\sqrt{h}s_k(J(x_k))))$
\end{itemize}
\textit{Step 1: Basis.} 
Consider the first step of \eqref{eq:algo} with \eqref{eq:TM}, i.e,
\begin{align}
x_1 &= x_0 + \sqrt{h} \alpha_1 s_0 (J(x_0)) \nonumber \\
&+ \sqrt{h} \alpha_2 s_0 \big(J(x_0+\sqrt{h}s_0(J(x_0)))\big) 
\nonumber \\
&\!\!\!\!\!\!\!\!\!\overset{(\text{\Cref{thm:taylor}})}{=} x_0 + \sqrt{h}(\alpha_1+\alpha_2)s_0(J(x_0)) \nonumber \\
&+ h \alpha_2 \frac{\partial s_0}{\partial J}(J(x_0))s_0(J(x_0))^\top \nabla J(x_0) 
\nonumber \\
&+ R_0(x_0,J(x_0);h^{3/2})
\label{pf:taylor_1}
\end{align}
with $R_0:\mathbb{R}^n \times \mathbb{R} \times \mathbb{R}_{\ge 0} \rightarrow \mathbb{R}^n$.
In the above equation, \Cref{thm:taylor} is applied twice. First, for $J(x_0 + \sqrt{h}s_0(J(x_0))$ where $a(\cdot):=J(\cdot)$, $b(\cdot) := s_0(\cdot)$, $z:=x_0$, and $y:=J(x_0)$ in \Cref{thm:taylor} are chosen such that $R^+_{J,0}(\cdot,\cdot;h)$ as in $\eqref{eq:tayler_R}$ exists, i.e., 
\begin{align}
J\big (x_0+\sqrt{h}s_0(J(x_0))\big) &= J(x_0) + \sqrt{h} s_0(J(x_0))^\top \nabla J(x_0)  \nonumber \\ &+ R^+_{J,0}(x_0,J(x_0);h)
\label{pf:taylor_2}
\end{align}
with $R^+_{J,0}:\mathbb{R}^n \times \mathbb{R} \times \mathbb{R}_{\ge 0} \rightarrow \mathbb{R}$.
Second, for $s_0(J(x_0+\sqrt{h}s_0(J(x_0))))$ where $a(\cdot):=f(\cdot)$ and $a(\cdot):=g(\cdot)$, $b(x_0,J(x_0))=s_0(J(x_0))^\top \nabla J(x_0)+R^+_{J,0}(x_0,J(x_0);h)$, $z:=J(x_0)$, and $y = [x_0^\top\ J(x_0)]^\top$ in \Cref{thm:taylor} are chosen such that a $R^+_{s,0}(\cdot,\cdot;h)$ as in $\eqref{eq:tayler_R}$ exists with given bounds on $u_0$ and $v_0$, i.e., 
\begin{align}
&s_0\big(J(x_0+\sqrt{h}s_0(J(x_0)))\big) 
\nonumber 
\\
&\quad= s_0\Big(J(x_0) + \sqrt{h} s_0(J(x_0))^\top \nabla J(x_0) 
\nonumber \\ 
&\quad  + R^+_{J,0}(x_0,J(x_0);h)\Big)
\nonumber \\
&\quad= s_0(J(x_0)) + \frac{\partial s_0}{\partial J} (J(x_0))\Big(\sqrt{h}s_0(J(x_0))^\top \nabla J(x_0)  \nonumber \\
&\quad + R^+_{J,0}(x_0,J(x_0);h)\Big) + R^+_{s,0}(x_0,J(x_0);h)
\label{pf:taylor_3}
\end{align}
with $R^+_{s,0}:\mathbb{R}^n \times \mathbb{R} \times \mathbb{R}_{>0} \rightarrow \mathbb{R}^n$.
Note that $s_0(\cdot) = f(\cdot)u_0 + g(\cdot) v_0$ (see \eqref{eq:TM}), i.e., \eqref{pf:taylor_3} is obtained by applying \Cref{thm:taylor} to $f(\cdot)$ and $g(\cdot)$ separately. We neglected this intermediate step and directly stated the Taylor expansion for $s_0(\cdot)$. In the sequel we do not highlight this intermediate step and ''directly`` apply \Cref{thm:taylor} to $a(\cdot):=s_i(\cdot)$ which implies that \Cref{thm:taylor} is applied to $a(\cdot):=f_i(\cdot)$ and $a(\cdot):=g_i(\cdot)$ separately with given bounds on $u_i$ and $v_i$. 

Then, since $J\in C^2(\mathbb{R}^n;\mathbb{R})$ by Assumption \ref{assump:smooth_fg}, $x_0,x_0+\sqrt{h} s_0(J(x_0)) \in \mathcal{X}$, and $J(x_0) \in \mathcal{J}$, we conclude with \Cref{thm:taylor}, there exists a $M^+_{J,0}\in\mathbb{R}_{\ge 0}$ such that $\|R^+_{J,0}(x_0,J(x_0);h)\|_2 \le M^+_{J,0} h$. Additionally, by assumption $s_0 \in C^2(\mathbb{R}; \mathbb{R}^n)$ (see Assumption \ref{assump:smooth_fg}) and $J(x_0 + \sqrt{h} s_0(J(x_0))) \in \mathcal{J}$, thus, by \Cref{thm:taylor} there exists a $M
^+_{s,0} \in \mathbb{R}_{\ge 0}$ such that $\|R^+_{s,0}(x_0,J(x_0);h)\|_2 \le M^+_{s_0}h$. Putting these facts together, we obtain that 
\begin{align}
R_0(x_0,J(x_0);h^{3/2}) &= \sqrt{h}\frac{\partial s_0}{\partial J}(J(x_0) R^+_{J,0}(x_0,J(x_0);h) \nonumber \\
&+ \sqrt{h} R^+_{s,0}(x_0,J(x_0);h)
\label{pf:taylor_4}
\end{align}
in \eqref{pf:taylor_1} holds. Then by assumption $s_0 \in C^2(\mathbb{R}; \mathbb{R}^n$) (see Assumption \ref{assump:smooth_fg}) and $J(x_0) \in \mathcal{J}$, there exits a $L_{s,0} \in \mathbb{R}_{\ge 0}$ such that 
\begin{align}
\|\frac{\partial s_0}{\partial J}(J(x_0)\|_2 \le L_{s,0}.
\label{pf:taylor_5}
\end{align}
Then it follows that 
\begin{align}
\|R_0(x_0,J(x_0);h^{3/2})\|_2 &\le M_0 h^{3/2} \nonumber \\
&:=(M^+_{J,0}L_{s,0}+M^+_{s,0})h^{3/2}
\label{pf:taylor_6}
\end{align}
and we can obtain that $R_0(x_0,J(x_0);h^{3/2}) = \mathcal{O}(h^{3/2})$, thus, \eqref{pf:taylor_1} is \eqref{eq:taylor_gen} for $k=0$ and $m=1$. 
\textit{Step 2: Inductive Step.} Assume that \eqref{eq:taylor_gen} holds for $x_{m-1}$, i.e., that the evolution of $x_k$ for $k = 0,\ldots,m-1$ reads
\begin{align}
x_{m-1} &= x_0 + \sqrt{h}(\alpha_1+\alpha_2)\sum_{i=0}^{m-2}s_i(J(x_0)) \nonumber
\\
&+h\alpha_2 \sum_{i = 0}^{m-2} \frac{\partial s_i}{\partial J}(J(x_0))s_i(J(x_0))^\top \nabla J(x_0)  
\nonumber \\
&+h (\alpha_1+\alpha_2)^2 \sum_{i = 0}^{m-2} \sum_{j = 0}^{i-1} \frac{\partial s_i}{\partial J}(J(x_0)) s_j(J(x_0))^\top \nabla J(x_0) \nonumber
\\
&+ R_{m-2}(x_0,J(x_0);h^{3/2})
\label{pf:taylor_7}
\end{align}
and there exists a $M_{m-2}\in \mathbb{R}_{\ge 0}$ such that $\|R_{m-2}(x_0,J(x_0);h^{3/2})\|_2 \le M_{m-2} h^{3/2}$.
Next we consider the $m$-th step of \eqref{eq:algo} with \eqref{eq:TM}, i.e.,
\begin{align}
x_{m} &= x_{m-1}+\sqrt{h} \alpha_1 s_{m-1}(J(x_{m-1})) \nonumber \\
&+\sqrt{h} \alpha_2 s_{m-1}\Big(J\big(x_{m-1}+\sqrt{h}s_{m-1}(J(x_{m-1}))\big)\Big).
\label{pf:taylor_8}
\end{align}
Again, as in Step 1, we apply \Cref{thm:taylor} several times. First for $J(x_{m-1})$ where $a(\cdot):=J(\cdot)$, $b(x_0,J(x_0))= h^{-1/2}(rhs.\ of\ \eqref{pf:taylor_7}-x_0)$, $z:=x_0$, and $y = [x_0^\top\ J(x_0)]^\top$ in \Cref{thm:taylor} are chosen such that a $R_{J,m-1}(\cdot,\cdot;h)$ as in $\eqref{eq:tayler_R}$ exists, i.e., 
\begin{align}
&J(x_{m-1}) = J(r.h.s.\ of\ \eqref{pf:taylor_7}) 
\nonumber \\
&\!\!\!\!\!\!\!\!\!\overset{(\text{\Cref{thm:taylor}})}{=} J(x_0) + 
\sqrt{h}(\alpha_1+\alpha_2)\sum_{i=0}^{m-2}s_i(J(x_0))^\top \nabla J(x_0)  \nonumber
\\
&+\bar{R}_{J,m-1}(x_0,J(x_0);h)
\label{pf:taylor_9}
\end{align}
with $\bar{R}_{J,m-1}:\mathbb{R}^n\times\mathbb{R}\times\mathbb{R}_{>0} \rightarrow \mathbb{R}$ where 
\begin{align}
&\bar{R}_{J,m-1}(x_0,J(x_0);h) =
h \nabla J(x_0)^\top \nonumber \\
&\quad\times \Big(\alpha_2 \sum_{i=0}^{m-2} s_i(J(x_0)) \frac{\partial s_i}{\partial J}(J(x_0))^\top 
\nonumber \\
&\quad+(\alpha_1+\alpha_2)^2 \sum_{i = 0}^{m-2} \sum_{j = 0}^{i-1}  s_j(J(x_0)) \frac{\partial s_i}{\partial J}(J(x_0))^\top \Big) \nabla J(x_0)
\nonumber \\
&\quad+R_{m-2}(x_0,J(x_0);h^{3/2})+R_{J,m-1}(x_0,J(x_0);h)
\label{pf:taylor_10}
\end{align}
with $R_{J,m-1}:\mathbb{R}^n\times\mathbb{R}\times\mathbb{R}_{\ge 0}\rightarrow \mathbb{R}$.
Secondly, \Cref{thm:taylor} is applied on $s_{m-1}(J(x_{m-1}))$ where $a(\cdot):=s_{m-1}(\cdot)$, $b(x_0,J(x_0))=h^{-1/2}(rhs.\ of\ \eqref{pf:taylor_9}-J(x_0))$, $z:=J(x_0)$, and $y = [x_0^\top\ J(x_0)]^\top$ in \Cref{thm:taylor} are chosen such that a $R_{s,m-1}(\cdot,\cdot;h)$ as in $\eqref{eq:tayler_R}$ exists, i.e., 
\begin{align}
&s_{m-1}(J(x_{m-1})) = s_{m-1}(rhs.\ of\ \eqref{pf:taylor_9})
\nonumber \\
&\quad \!\!\!\!\!\!\!\!\!\overset{(\text{\Cref{thm:taylor}})}{=} s_{m-1}(J(x_0)) 
+ \sqrt{h} (\alpha_1+\alpha_2) \nonumber 
\\ 
&\quad \times \sum_{i=0}^{m-2}\frac{\partial s_{m-1}}{\partial J}(J(x_0)) s_i(J(x_0))^\top \nabla J(x_0)
\nonumber \\
&\quad + \bar{R}_{s,m-1}(x_0,J(x_0);h)
\label{pf:taylor_11}
\end{align}
with $\bar{R}_{s,m-1}:\mathbb{R}^n\times\mathbb{R}\times\mathbb{R}_{\ge 0} \rightarrow \mathbb{R}^n$ where
\begin{align}
\bar{R}_{s,m-1}(x_0,J(x_0);h) &= \frac{\partial s_{m-1}}{\partial J}(J(x_0) \bar{R}_{J,m-1}(x_0,J(x_0);h) \nonumber 
\\
&+ R_{s,m-1}(x_0,J(x_0);h)
\label{pf:taylor_12}
\end{align}
with $R_{s,m-1}:\mathbb{R}^n\times\mathbb{R}\times\mathbb{R}_{\ge 0} \rightarrow \mathbb{R}^n$. Thirdly, \Cref{thm:taylor} is applied on $J(x_{m-1}+\sqrt{h}s_{m-1}(J(x_{m-1})))$ where $a(\cdot):=J(\cdot)$, $b(x_0,J(x_0))= h^{-1/2}(rhs.\ of\ \eqref{pf:taylor_7}-x_0)+\sqrt{h}(rhs.\ of\ \eqref{pf:taylor_11})$, $z:=x_0$, and $y = [x_0^\top\ J(x_0)]^\top$ in \Cref{thm:taylor} are chosen such that a $R^+_{J,m-1}(\cdot,\cdot;h)$ as in $\eqref{eq:tayler_R}$ exists, i.e., 
\begin{align}
&J(x_{m-1}+\sqrt{h}s_{m-1}(J(x_{m-1})))  \nonumber 
\\
&\quad = J(rhs.\ of\ \eqref{pf:taylor_7}+\sqrt{h}(rhs.\ of\ \eqref{pf:taylor_11})) \nonumber 
\\
&\quad \!\!\!\!\!\!\!\!\!\overset{(\text{\Cref{thm:taylor}})}{=} J(x_0) + 
\sqrt{h}(\alpha_1+\alpha_2)\sum_{i=0}^{m-2}s_i(J(x_0))^\top \nabla J(x_0) \nonumber 
\\
&\quad +\sqrt{h}s_{m-1}(J(x_0))^\top \nabla J(x_0) \nonumber \\
&\quad + \bar{R}^+_{J,m-1}(x_0,J(x_0);h)
\label{pf:taylor_13}
\end{align}
with $\bar{R}^+_{J,m-1}:\mathbb{R}^n\times\mathbb{R}\times\mathbb{R}_{\ge 0} \rightarrow \mathbb{R}$ where 
\begin{align}
&\bar{R}^+_{J,m-1}(x_0,J(x_0);h)  \nonumber
\\
&\quad = h \nabla J(x_0)^\top \Big(\alpha_2 \sum_{i=0}^{m-2} s_i(J(x_0)) \frac{\partial s_i}{\partial J}(J(x_0))^\top 
\nonumber \\
&\quad+(\alpha_1+\alpha_2)^2 \sum_{i = 0}^{m-2} \sum_{j = 0}^{i-1}  s_j(J(x_0)) \frac{\partial s_i}{\partial J}(J(x_0))^\top \nonumber \\
&\quad \times (\alpha_1+\alpha_2)\sum_{i=0}^{m-2} s_i(J(x_0))\frac{\partial s_{m-1}}{\partial J}(J(x_0))^\top\Big) \nabla J(x_0)
\nonumber \\
&\quad+R_{m-2}(x_0,J(x_0);h^{3/2}) + \sqrt{h} \bar{R}_{s,m-1}(x_0,J(x_0);h) \nonumber 
\\
&\quad +R^+_{J,m-1}(x_0,J(x_0);h)
\label{pf:taylor_14}
\end{align}
with $R^+_{J,m-1}:\mathbb{R}^n\times\mathbb{R}\times\mathbb{R}_{\ge 0}\! \rightarrow\! \mathbb{R}$. Lastly, \Cref{thm:taylor} is applied on \linebreak[4] $s_{m-1}(J(x_{m-1}+\sqrt{h}s_{m-1}(J(x_{m-1}))))$ where $a(\cdot):=s_{m-1}(\cdot)$, $b(x_0,J(x_0))=h^{-1/2}(rhs.\ of\ \eqref{pf:taylor_13}-J(x_0))$, $z:=J(x_0)$, and $y = [x_0^\top\ J(x_0)]^\top$ in \Cref{thm:taylor} are chosen such that a $R^+_{s,m-1}(\cdot,\cdot;h)$ as in $\eqref{eq:tayler_R}$ exists, i.e., 
\begin{align}
&s_{m-1}(J(x_{m-1}+\sqrt{h}s_{m-1}(J(x_{m-1})))) \nonumber 
\\
&\quad = s_{m-1}(rhs.\ of\ \eqref{pf:taylor_13})
\nonumber \\
&\quad \!\!\!\!\!\!\!\!\! \overset{(\text{\Cref{thm:taylor}})}{=} s_{m-1}(J(x_0)) 
+ \sqrt{h} \nonumber 
\\ 
&\quad \times (\alpha_1+\alpha_2)\sum_{i=0}^{m-2}\frac{\partial s_{m-1}}{\partial J}(J(x_0)) s_i(J(x_0))^\top \nabla J(x_0)
\nonumber \\
&\quad+\sqrt{h}\frac{\partial s_{m-1}}{\partial J}(J(x_0))s_{m-1}(J(x_0))^\top \nabla J(x_0) \nonumber 
\\
&\quad + \bar{R}^+_{s,m-1}(x_0,J(x_0);h)
\label{pf:taylor_15}
\end{align}
with $\bar{R}^+_{s,m-1}:\mathbb{R}^n\times\mathbb{R}\times\mathbb{R}_{>0} \rightarrow \mathbb{R}^n$ where
\begin{align}
\bar{R}^+_{s,m-1}(x_0,J(x_0);h) &= \frac{\partial s_{m-1}}{\partial J}(J(x_0) \bar{R}^+_{J,m-1}(x_0,J(x_0);h) \nonumber \\
& + R^+_{s,m-1}(x_0,J(x_0);h)
\label{pf:taylor_16}
\end{align}
with $R^+_{s,m-1}:\mathbb{R}^n\times\mathbb{R}\times\mathbb{R}_{\ge 0} \rightarrow \mathbb{R}^n$. 
Then with the same arguments as in Step 1, namely by assumption $J\in C^2(\mathbb{R}^n;\mathbb{R})$ (see Assumption \ref{assump:smooth_fg}) for every $k=0,\ldots,m-1$, $s_k \in C^2(\mathbb{R}; \mathbb{R}^n)$ (see Assumption \ref{assump:smooth_fg}), $x_k,x_k+\sqrt{h} s_k(J(x_k)) \in \mathcal{X}$, and $J(x_k),J(x_k + \sqrt{h} s_k(J(x_k))) \in \mathcal{J}$, thus, there exist $\bar{M}_{J,m-1}, \bar{M}_{s,m-1},$ $\bar{M}^+_{J,m-1},$ $\bar{M}^+_{s,m-1} \in \mathbb{R}_{\ge 0}$ such that $\|\bar{R}_{J,m-1}(x_0,J(x_0);h)\|_2 \le \bar{M}_{J,m-1}h$ and so on. Note that $\bar{M}_{J,m-1}$ is derived by \eqref{pf:taylor_10} as
\begin{align}
\bar{M}_{J,m-1} &= L_J^2 \Big(\alpha_2 \sum_{i=0}^{m-2}K_{s,i}L_{s,i} + (\alpha_1+\alpha_2)^2 \sum_{i=0}^{i-1} K_{s,i}L_{s,i}\Big) \nonumber 
\\
& + \sqrt{h} M_{m-2} + M_{J,m-1}
\label{pf:taylor_17}
\end{align}
with $L_J,K_{s,i},L_{s,i},M_{m-2}\in \mathbb{R}_{\ge 0}$, where $\|\nabla J(x_0)\|_2 \le L_J$, $\|\nabla s_i(J(x_0))\|_2 \le K_{s,i}$, $\| \nicefrac{\partial s_i}{\partial J}J(x_0)\|_2 \le L_{s,i}$, and $\|R_{m-2}(x_0,J(x_0);h^{3/2})\|_2 \le \bar{M}_{m-2}h^{3/2}$. Additionally, by \Cref{thm:taylor} there exists a $M_{J,m-1}\in\mathbb{R}_{>0}$ such that  $\|R_{J,m-1}(x_0,J(x_0);h)\|_2 \le M_{J,m-1}h$. The constants $\bar{M}_{s,m-1}, \bar{M}^+_{J,m-1}$, $\bar{M}^+_{s,m-1}$ are derived in the same manner via \eqref{pf:taylor_12}, \eqref{pf:taylor_14}, and \eqref{pf:taylor_16}, respectively. 
Finally, plugging \eqref{pf:taylor_11} and \eqref{pf:taylor_15} in \eqref{pf:taylor_8} yields \eqref{eq:taylor_gen} with $k=0$ and 
\begin{align}
R_{m-1}(x_0,J(x_0);h^{3/2}) &= R_{m-2}(x_0,J(x_0);h^{3/2}) \nonumber 
\\
&+ \sqrt{h}\alpha_1\bar{R}_{s,m-1}(x_0,J(x_0);h)
\nonumber \\
&+\sqrt{h}\alpha_2\bar{R}^+_{s,m-1}(x_0,J(x_0);h),
\end{align}
where,
\begin{align}
&\|R_{m-1}(x_0,J(x_0);h^{3/2})\|_2 \le M_{m-1} h^{3/2} \nonumber \\
&\quad :=(M_{m-2} + \alpha_1 \bar{M}_{s,m-1} + \alpha_1 \bar{M}^+_{s,m-1})h^{3/2}.
\end{align}
\end{proof}
The following two lemmas state the sufficient part of Cauchy's interlacing inequalities \cite{Cauchy} for real skew-symmetric matrices. Hence, the imaginary part of the eigenvalues of the principal submatrix can be chosen w.r.t. certain inequalities depending on the eigenvalues of the given skew-symmetric matrix. 
\begin{lem}
	\label{thm:fan_pall}
	Let $C \in \mathbb{R}^{p\times p}$ be a skew-symmetric matrix with eigenvalues $\pm \eta_k i, \eta_k \in \mathbb{R}_{\ge 0}, k = 1,\ldots,\lceil p/2 \rceil$ and let $\omega_\ell \in \mathbb{R}_{\ge 0},\ \ell=1,\ldots,\lceil p/2 \rceil-1$ such that the inequality
	\begin{align}
		\eta_1 \ge \omega_1 \ge \eta_2 \ge \omega_2 \cdots \ge \eta_{\lceil p/2 \rceil-1} \nonumber \\
		\ge \omega_{\lceil p/2 \rceil-1} \ge \eta_{\lceil p/2 \rceil} \ge 0,
		\label{eq:fan_pall1}
	\end{align}
	is satisfied. Then there exists an unitary matrix $\Theta \in \mathbb{R}^{p\times p}$ such that $Q \in \mathbb{R}^{(p-1) \times (p-1)}$ is a principal submatrix of $\Theta^\top C\Theta$ with eigenvalues $\pm \omega_\ell i$.
\end{lem}
For a proof, we refer to \cite{THOMPSON1979249}.
\begin{lem}
	\label{thm:fan_pall_mult}
	Let $C \in \mathbb{R}^{p\times p}$ be a skew-symmetric matrix with eigenvalues $\pm \eta_\ell i,\eta_{\ell}\in\mathbb{R}_{\ge 0}, \ell = 1,\ldots,p$, arranged according to
	\begin{align}
		\eta_1\ge\eta_2\ge\ldots\ge\eta_{\lceil p/2 \rceil}\ge 0.
	\end{align}
	Then for $\omega_1 \ge \omega_2 \ge \ldots \ge \omega_r$ with $\omega_k \in \mathbb{R}_{\ge 0}$ such that
	\begin{align}
		&\eta_k \ge \omega_k \ge \eta_{\lceil p/2 \rceil -r + k},
		\label{eq:fan_pall_mult}
	\end{align}
	there exists an unitary matrix $\Theta \in \mathbb{R}^{p\times p}$ such that $Q \in \mathbb{R}^{(2r) \times (2r)}$ is a principal submatrix of $\Theta^\top C\Theta$ with eigenvalues $\pm \omega_\ell i,\ \ell =1,\ldots,r$.
\end{lem}
\begin{proof}
	Applying \cref{thm:fan_pall} $\lceil p/2 \rceil - r$ times, yields the result, similar to the proof of \cite[Theorem 1]{fan1957imbedding}.
\end{proof}
\subsubsection*{Numerical validation of \cref{property:interlacing}}
	\label{thm:interlacing_P}
	Due to the dependency of $C(m)$ on $m$ in \eqref{eq:pm}, the interlacing lemmas, see \Cref{thm:fan_pall} and \Cref{thm:fan_pall_mult}, are not applicable, since the entries of $C(m)$ change with dimension due to $\epsilon(m)$.  
	We verified numerically that the interlacing property \eqref{eq:interlacing} holds for all $m \le 10000$. The corresponding Matlab code is available in the ancillary file folder on Arxiv. Moreover, in \Cref{fig:interlacing_proof} the interlacing property for $4 \le m \le 10$ is visualized. 
	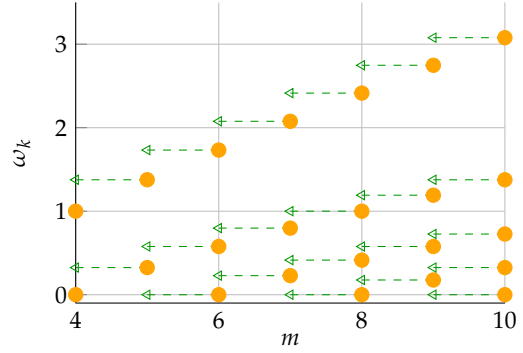
\begin{figure}[h]
		\centering
		\setlength\figureheight{4cm}
		\setlength\figurewidth{6cm}
		{\small 
%
\definecolor{mycolor1}{rgb}{0.00000,0.44700,0.74100}%
\definecolor{mycolor2}{rgb}{1.00000,0.64453,0.00000}%
\definecolor{mycolor3}{rgb}{0.49400,0.18400,0.55600}%
\definecolor{mycolor4}{rgb}{0.46600,0.67400,0.18800}%
\definecolor{mycolor5}{rgb}{0.92900,0.69400,0.12500}%
\definecolor{mycolor6}{rgb}{0.63500,0.07800,0.18400}%
\begin{tikzpicture}

\begin{axis}[%
width=0.951\figurewidth,
height=\figureheight,
at={(0\figurewidth,0\figureheight)},
scale only axis,
xmin=4,
xmax=10,
xlabel=$m$,
xlabel style={at={(0.5,0.06)}},
ymin=-0.1,
ymax=3.5,
ylabel=$\omega_k$,
ylabel style={at={(0.09,0.5)}},
axis background/.style={fill=white},
axis x line*=bottom,
axis y line*=left,
xmajorgrids,
ymajorgrids
]
\addplot [color=mycolor1, draw=none, mark size=2.7pt, mark=*, mark options={solid, fill=mycolor2, mycolor2}, forget plot]
  table[row sep=crcr]{%
4	1\\
4	0\\
};
\addplot [color=green60, dashed, mark=triangle, mark options={solid, rotate=90, green60}, forget plot]
  table[row sep=crcr]{%
4	1.37638192047117\\
5	1.37638192047117\\
};
\addplot [color=green60, dashed, mark=triangle, mark options={solid, rotate=90, green60}, forget plot]
  table[row sep=crcr]{%
4	0.324919696232906\\
5	0.324919696232906\\
};
\addplot [color=mycolor3, draw=none, mark size=2.7pt, mark=*, mark options={solid, fill=mycolor2, mycolor2}, forget plot]
  table[row sep=crcr]{%
5	1.37638192047117\\
5	0.324919696232906\\
};
\addplot [color=green60, dashed, mark=triangle, mark options={solid, rotate=90, green60}, forget plot]
  table[row sep=crcr]{%
5	1.73205080756888\\
6	1.73205080756888\\
};
\addplot [color=green60, dashed, mark=triangle, mark options={solid, rotate=90, green60}, forget plot]
  table[row sep=crcr]{%
5	0.577350269189626\\
6	0.577350269189626\\
};
\addplot [color=green60, dashed, mark=triangle, mark options={solid, rotate=90, green60}, forget plot]
  table[row sep=crcr]{%
5	0\\
6	0\\
};
\addplot [color=mycolor1, draw=none, mark size=2.7pt, mark=*, mark options={solid, fill=mycolor2, mycolor2}, forget plot]
  table[row sep=crcr]{%
6	1.73205080756888\\
6	0.577350269189626\\
6	0\\
};
\addplot [color=green60, dashed, mark=triangle, mark options={solid, rotate=90, green60}, forget plot]
  table[row sep=crcr]{%
6	2.07652139657234\\
7	2.07652139657234\\
};
\addplot [color=green60, dashed, mark=triangle, mark options={solid, rotate=90, green60}, forget plot]
  table[row sep=crcr]{%
6	0.797473388882404\\
7	0.797473388882404\\
};
\addplot [color=green60, dashed, mark=triangle, mark options={solid, rotate=90, green60}, forget plot]
  table[row sep=crcr]{%
6	0.22824347439015\\
7	0.22824347439015\\
};
\addplot [color=mycolor4, draw=none, mark size=2.7pt, mark=*, mark options={solid, fill=mycolor2, mycolor2}, forget plot]
  table[row sep=crcr]{%
7	2.07652139657234\\
7	0.797473388882404\\
7	0.22824347439015\\
};
\addplot [color=green60, dashed, mark=triangle, mark options={solid, rotate=90, green60}, forget plot]
  table[row sep=crcr]{%
7	2.4142135623731\\
8	2.4142135623731\\
};
\addplot [color=green60, dashed, mark=triangle, mark options={solid, rotate=90, green60}, forget plot]
  table[row sep=crcr]{%
7	1\\
8	1\\
};
\addplot [color=green60, dashed, mark=triangle, mark options={solid, rotate=90, green60}, forget plot]
  table[row sep=crcr]{%
7	0.414213562373095\\
8	0.414213562373095\\
};
\addplot [color=green60, dashed, mark=triangle, mark options={solid, rotate=90, green60}, forget plot]
  table[row sep=crcr]{%
7	0\\
8	0\\
};
\addplot [color=mycolor5, draw=none, mark size=2.7pt, mark=*, mark options={solid, fill=mycolor2, mycolor2}, forget plot]
  table[row sep=crcr]{%
8	2.4142135623731\\
8	1\\
8	0.414213562373095\\
8	0\\
};
\addplot [color=green60, dashed, mark=triangle, mark options={solid, rotate=90, green60}, forget plot]
  table[row sep=crcr]{%
8	2.74747741945462\\
9	2.74747741945462\\
};
\addplot [color=green60, dashed, mark=triangle, mark options={solid, rotate=90, green60}, forget plot]
  table[row sep=crcr]{%
8	1.19175359259421\\
9	1.19175359259421\\
};
\addplot [color=green60, dashed, mark=triangle, mark options={solid, rotate=90, green60}, forget plot]
  table[row sep=crcr]{%
8	0.577350269189626\\
9	0.577350269189626\\
};
\addplot [color=green60, dashed, mark=triangle, mark options={solid, rotate=90, green60}, forget plot]
  table[row sep=crcr]{%
8	0.176326980708465\\
9	0.176326980708465\\
};
\addplot [color=mycolor1, draw=none, mark size=2.7pt, mark=*, mark options={solid, fill=mycolor2, mycolor2}, forget plot]
  table[row sep=crcr]{%
9	2.74747741945462\\
9	1.19175359259421\\
9	0.577350269189626\\
9	0.176326980708465\\
};
\addplot [color=green60, dashed, mark=triangle, mark options={solid, rotate=90, green60}, forget plot]
  table[row sep=crcr]{%
9	3.07768353717526\\
10	3.07768353717526\\
};
\addplot [color=green60, dashed, mark=triangle, mark options={solid, rotate=90, green60}, forget plot]
  table[row sep=crcr]{%
9	1.37638192047117\\
10	1.37638192047117\\
};
\addplot [color=green60, dashed, mark=triangle, mark options={solid, rotate=90, green60}, forget plot]
  table[row sep=crcr]{%
9	0.726542528005362\\
10	0.726542528005362\\
};
\addplot [color=green60, dashed, mark=triangle, mark options={solid, rotate=90, green60}, forget plot]
  table[row sep=crcr]{%
9	0.324919696232906\\
10	0.324919696232906\\
};
\addplot [color=green60, dashed, mark=triangle, mark options={solid, rotate=90, green60}, forget plot]
  table[row sep=crcr]{%
9	0\\
10	0\\
};
\addplot [color=mycolor6, draw=none, mark size=2.7pt, mark=*, mark options={solid, fill=mycolor2, mycolor2}, forget plot]
  table[row sep=crcr]{%
10	3.07768353717526\\
10	1.37638192047117\\
10	0.726542528005362\\
10	0.324919696232906\\
10	0\\
};
\end{axis}
\end{tikzpicture}%
}
		\caption{Illustration of the absolute value of the complex conjugated eigenvalues $\pm \omega_k i$ ({\color{orange90}$\bullet$}) of $\tilde{P}$ in \eqref{eq:P_tilde} w.r.t. $m$ for $k = 1,\ldots,\lceil m/2 \rceil$. The interlacing property \eqref{eq:interlacing} is apparent and visualized by the arrows, mapping the eigenvalues to the next lower dimension.}
		\label{fig:interlacing_proof}
	\end{figure}
	Note that $m$ acts as the exploration sequences' lengths as introduced in \eqref{eq:composition} and \eqref{eq:TM}, where in the worst case $m=4n$ (cf. \eqref{eq:input_coord}). This implies the property holds for sure for $n \le 2500$ dimensional problems, which is a very high-dimensional problem for derivative-free algorithms.  
\begin{lem}
	\label{thm:gen_area_skew}
	Let $W$ such that $W\mathbb{1}=0$ be given and et $T(W)$ be as in \eqref{eq:taylor_matrix} with $[\alpha_1\ \alpha_2] = [\nicefrac{1}{2}\ \nicefrac{1}{2}]$. Then the value in the $p$-th row and $q$-th column of $T(W)$ is equivalent to the net area (cf. Gauss area formula in \cite{spivak2018calculus}) of the $n$-sided polygon in the
	$e_p-e_q$ plane with corner points 
	\begin{align}
	\correct{x_{p,i} = \sum_{k=0}^{i-1} e_p^\top w_k,\quad y_{q,i} = \sum_{k=0}^{i-1} e_q^\top w_k}
	\label{eq:lem_gen_area_points}
	\end{align}
	for $i=0,\ldots,m-1$ and $p,q=1,\ldots,2n$, where $x_{p,0}= 0$ and $y_{q,0} = 0$.
\end{lem}
\begin{proof}
	The entry in the $p$-th row and $q$-th column of the matrix $T(W)$ in \eqref{eq:taylor_matrix} with $[\alpha_1\ \alpha_2] = [\nicefrac{1}{2}\ \nicefrac{1}{2}]$ is given by 
	\begin{align}
	T(W)_{pq} = \sum_{i=0}^{m-1}\frac{1}{2}e_p^\top w_iw_i^\top e_q + \sum_{j=0}^{i-1}e_p^\top w_iw_j^\top e_q. 
	\end{align}
	Condition $W\mathds{1}=0$ implies $w_{m-1} = -\sum_{i=0}^{m-1} w_i$ such that
	we have
	\begin{align}
	T(W) &= \sum_{i=0}^{m-2}\Big( \frac{1}{2} w_iw_i^\top + \sum_{j=0}^{i-1} w_iw_j^\top \Big) 
	+ \frac{1}{2}w_{m-1}w_{m-1}^\top \nonumber \\ 
	&+ \sum_{j=0}^{m-2} w_{m-1}w_j^\top 
	\nonumber \\
	&= \sum_{i=0}^{m-2}\Big( \frac{1}{2} w_iw_i^\top + \sum_{j=0}^{i-1} w_iw_j^\top 
	+\frac{1}{2} \sum_{j=0}^{m-2} w_iw_j^\top \nonumber \\
	&- \sum_{j=0}^{m-2}w_iw_j^\top \Big)
	\nonumber \\
	&= \frac{1}{2} \sum_{i=0}^{m-2}\Big(w_iw_i^\top + \sum_{j=0}^{i-1} w_iw_j^\top-\sum_{j=i}^{m-2} w_iw_j^\top \Big)
	\nonumber \\
	&= \frac{1}{2}\sum_{i=0}^{m-2}\Big( \sum_{j=0}^{i-1} w_iw_j^\top -\frac{1}{2} \sum_{j=i+1}^{m-2} w_iw_j^\top \Big)
	\nonumber \\
	&= \frac{1}{2}\sum_{i=1}^{m-2}\sum_{j=0}^{i-1} (w_iw_j^\top - w_jw_i^\top),
	\label{eq:lem_area_1}
	\end{align}
	where the last line of \eqref{eq:lem_area_1} is obtained by the reindexing
	\begin{align}
	\sum_{i=0}^{m-2}\sum_{j=i+1}^{m-2}w_iw_j^\top &= w_0 (w_1+w_2+\cdots+w_{m-2})^\top 
	\nonumber \\
	&+ w_1 (w_2+w_3+\cdots+w_{m-2})^\top 
	\nonumber \\
	&+ \ldots 
	\nonumber \\
	&+ w_{m-3} w_{m-2}^\top
	=\sum_{i=1}^{m-2}\sum_{j=0}^{i-1}w_jw_i^\top.
	\end{align}
	Hence, it follows for the entry of $T(W)$ in the $p$-th row and $q$-th column
	\begin{align}
	T(W)_{pq} = \frac{1}{2}\sum_{i=1}^{m-2}\sum_{j=0}^{i-1} e_p^\top w_i w_j^\top e_q - e_p^\top w_j w_i^\top e_q.
	\end{align}
	On the other hand, the net area formula of a $n$-sided polygon with corner points $\{x_{p,i},y_{q,i}\}_{i=1}^{n}$ is given as 
	\begin{align}
	A_{pq} = \frac{1}{2} \sum_{i=0}^{n-1}x_{p,i+1}y_{q,i} - x_{p,i} y_{q,i+1}.
	\label{eq:lem_gen_area_area}
	\end{align}
	Plugging \eqref{eq:lem_gen_area_points} with $n=m$ in \eqref{eq:lem_gen_area_area} 
	and using  $\sum_{i=0}^{m-1} w_i=0$
	yields 
	\begin{align}
	A_{pq} &= \frac{1}{2} \sum_{i=0}^{m-1} \Big( \sum_{j=0}^{i} e_p^\top w_j \Big) \Big( \sum_{j=0}^{i-1}  w_j^\top e_q \Big) 
	\nonumber
	\\
	&- \Big( \sum_{j=0}^{i-1} e_p^\top w_j \Big) \Big( \sum_{j=0}^{i}  w_j^\top e_q \Big) 
	\nonumber \\ 
	&= \frac{1}{2} \sum_{i=1}^{m-2}  e_p^\top w_i \Big( \sum_{j=0}^{i-1}  w_j^\top e_q \Big) \nonumber \\
	&- \Big( \sum_{j=0}^{i-1} e_p^\top w_j \Big)  w_i^\top e_q.
	\label{eq:lem_gen_area_area_calc}
	\end{align}
	Consequently, $A_{pq} = T(W)_{pq}$ holds. 
\end{proof}
\section{PROOFS}
\subsection{Proof \cref{thm:cond_equ}}
\label{sec:proof_thm_cond_equ}
The proof utilizes the result of \cref{thm:taylor_gen}. Consider the $m$-th step of the evolution of \eqref{eq:algo} represented by \eqref{eq:taylor_gen} with transition map \eqref{eq:TM}. Let $w_\ell = [u_\ell^\top\ v_\ell^\top]^\top$ for $\ell = k,\ldots,k+m-1$ and 
\begin{align}
Y(f(z),g(z)) &= \begin{bmatrix}f(z)I & g(z)I \end{bmatrix}, \\
\tilde{Y}(f(z),g(z)) &= \begin{bmatrix}\frac{\partial f}{\partial z}(z)I& \frac{\partial g}{\partial z}(z)I \end{bmatrix}.
\end{align}
Then plugging $\{w_\ell\}_{\ell = k}^{k+m-1},\ Y(f(J(x_k)),g(J(x_k)))$ and $\tilde{Y}(f(J(x_k)),g(J(x_k)))$ into \eqref{eq:taylor_gen} yields 
\begin{align}
&x_{k+m} = x_k +\sqrt{h} (\alpha_1+\alpha_2) Y(f(J(x_k)),g(J(x_k))) \sum_{i=k}^{k+m-1}w_i
\nonumber \\
&+ h \tilde{Y}(f(J(x_k)),g(J(x_k))) \left\lbrace \sum_{i=k}^{k+m-1}  \left( \alpha_2 w_iw_i^\top \right.\right. \nonumber \\
&\left.\left.+(\alpha_1+\alpha_2)^2 \sum_{j=k}^{i-1} w_iw_j^\top \right)\right\rbrace Y(f(J(x_k)),g(J(x_k)))^\top \nonumber\\ 
&\times \nabla J(x_k)  + \mathcal{O}(h^{3/2}).
\label{eq_p1:1}
\end{align}
The term in the curly brackets in \eqref{eq_p1:1} yields $T(W)$ in \eqref{eq:cond_equ_gen} and therefore \eqref{eq:taylor_matrix} is recovered.
$\hfill \small{\square}$
\subsection{Proof \cref{thm:matrices}}
\label{sec:proof_matrices}
Condition \eqref{eq:cond_equ_gen_zero} implies 
\begin{align}
w_{m-1} = - \sum_{i=0}^{m-2}w_i.
\label{eq:1}
\end{align}
Plugging \eqref{eq:1} into \eqref{eq:cond_equ_gen} yields 
\begin{align}
T(W) &= \sum_{i=0}^{m-2}\Big( \alpha_2 w_iw_i^\top + (\alpha_1 + \alpha_2)^2 \sum_{j=0}^{i-1} w_iw_j^\top  \Big)
\nonumber \\
&+\alpha_2 w_{m-1}w_{m-1}^\top + (\alpha_1+\alpha_2)^2\sum_{i = 0}^{m-2} w_{m-1}w_i^\top
\nonumber  \\
&= \alpha_2 \sum_{i=0}^{m-2}\Big(w_iw_i^\top + \sum_{j=0}^{m-2} w_iw_j^\top \Big)
\nonumber  \\
&+(\alpha_1+\alpha_2)^2\sum_{i=0}^{m-2}\Big(\sum_{j=0}^{i-1}w_iw_j^\top -\sum_{j=0}^{m-2}w_iw_j^\top \Big)
\nonumber \\
&= \alpha_2 \sum_{i=0}^{m-2}\Big(w_iw_i^\top + \sum_{j=0}^{m-2} w_iw_j^\top \Big)
\nonumber  \\
\end{align}
\begin{align}
&-(\alpha_1+\alpha_2)^2\sum_{i=0}^{m-2}\sum_{j=i}^{m-2}w_iw_j^\top 
\nonumber \\
&= (\alpha_2 - (\alpha_1 + \alpha_2)^2 )\sum_{i=0}^{m-2}\sum_{j=i}^{m-2} w_iw_j^\top
\nonumber \\
&+\alpha_2\sum_{i=0}^{m-2}\sum_{j=0}^{i} w_iw_j^\top.
\end{align}
Hence, $P$ in \eqref{eq:def_mat} is recovered.
$\hfill \small{\square}$
\subsection{Proof \cref{thm:existence_svd}}
\label{sec:proof_existence_svd}
Consider \eqref{eq:svd_prob} and the singular value decomposition of the exploration sequence matrix
\begin{align}
	W = U\Sigma V^\top
	\label{eq:svd}
\end{align} 
with
\begin{align}
	\bullet\ \ &U = [a_1\ b_1\ \cdots a_n\ b_n],
	\label{pf:construct_U}
	\\	
	\bullet\ \ &\Sigma = 
	\begin{bmatrix}	\Sigma_0 & 0 \\	0 & 0 \end{bmatrix},
	\ 
	\Sigma_0 = \text{diag}([\sigma_1\ \cdots \sigma_{r}])
	\label{pf:construct_sigma} 
	\\
	\bullet\ \ &V = 
	\begin{bmatrix}
		\Theta - \epsilon\mathds{1}_{}\mathds{1}_{}^\top \Theta &\!\!\! m^{-1/2} \mathds{1}_{} 
			\\ -\mathds{1}_{}^\top\Theta  + \epsilon (m-1)\mathds{1}_{}^\top \Theta &\!\!\! m^{-1/2}
	\end{bmatrix}\ \text{with}
	\label{pf:construct_V}
	\\
	&\epsilon = (m-1)^{-1}(1-m^{-1/2}) \text{ and some }  \nonumber
	\\
	&\Theta \in \mathbb{R}^{(m-1)\times (m-1)} \text{ s.t. } \Theta^\top \Theta = \Theta \Theta^\top = I  \nonumber
\end{align}
Hereby, $a_\ell \pm b_\ell i$ with $a_\ell,b_\ell \in \mathbb{R}^{2n}$ for $\ell = 1,\ldots,n$ are the eigenvectors of $T_d$. Since $U$, as defined in \eqref{pf:construct_U}, is constructed by the real and imaginary parts of the eigenvectors of the matrix $T_d$, $U$ is orthogonal \cite{connes1998noncommutative}. Moreover, 
\begin{align}
	\begin{split}
	&X:=U^\top T_d U = \text{diag}([C_1\ \cdots\ C_n]),
	\\
	&\text{with }C_\ell = 
	\begin{bmatrix}
		\gamma_\ell & -\delta_{\ell} \\ \delta_{\ell} & ~~~\gamma_\ell
	\end{bmatrix},\ \ell = 1,\ldots,n,
	\label{pf_svd:X}
	\end{split}
\end{align}
where $\gamma_{\ell} \pm \delta_{\ell} i$ are the eigenvalues of $T_d$ with $(2\alpha_2-(\alpha_1+\alpha_2)^2)\gamma_\ell \in \mathbb{R}_{\ge 0}$ and $\delta_\ell \in \mathbb{R}_{\ge 0}$. Note that $\gamma_\ell = 0$ and $\delta_\ell = 0$ for $\ell > \text{rk}(T_d)$.
The orthogonality of $V$, given in \eqref{pf:construct_V}, is shown by direct evaluation:
\begin{align}
	V^\top V &= 
	\begin{bmatrix}
		\hat{V}_{11} & 0 \\
		0 & 1
	\end{bmatrix}\quad \text{with}
	\label{pf:2_7}
	\\
	\hat{V}_{11}&=I + (1-2\epsilon m + \epsilon^2m(m-1))\Theta^\top\allone{}\allone{}^\top \Theta, \nonumber
	\\
	VV^\top &= 
	\begin{bmatrix}
		\tilde{V}_{11}  & \tilde{V}_{12} \\
		\tilde{V}_{12}^\top & \tilde{V}_{22}
	\end{bmatrix} \quad \text{with} 	
	\label{pf:2_8}
	\\
	\tilde{V}_{11} &= I + (\epsilon^2(m-1) + m^{-1} - 2\epsilon)\allone{}\allone{}^\top 
	\nonumber \\
	\tilde{V}_{12}&=-(\epsilon^2(m-1)^2-2\epsilon(m-1) -m^{-1} + 1 ) \allone{}
	\nonumber \\
	 \tilde{V}_{22}&= \epsilon^2(m-1)^3 - 2\epsilon (m-1)^2 + m^{-1} + m - 1
	 \nonumber
\end{align}
where we used the fact that $\Theta$ in \eqref{pf:construct_V} is orthogonal. By plugging $\epsilon = (m-1)^{-1}(1-m^{-1/2})$ into \eqref{pf:2_7} and \eqref{pf:2_8}, the orthogonality of $V$, i.e., $VV^\top = V^\top V = I$ is recovered. 
Now plugging \eqref{eq:svd} into \eqref{eq:svd_prob} associated with \eqref{pf_svd:X} reveals 
\begin{align}
	\begin{bmatrix}
		\Sigma_0 & 0 \\
		0 & 0
	\end{bmatrix}
	V^\top P V 
	\begin{bmatrix}
		\Sigma_0 & 0 \\
		0 & 0
	\end{bmatrix} 
	= X,
	\label{pf:2_9}
\end{align}
where 
\begin{align}
	&\!\!\!Q := V^\top P V = 
	\begin{bmatrix}
		\tilde{Q} & * \\ * & *
	\end{bmatrix}
	\ \text{with}\ \tilde{Q} = \Theta^\top \corrected{\tilde{P}} \Theta\ \text{and} 
	\label{pf:2_Y}
	\\
	&\!\!\!\tilde{P} = \big(P-\epsilon(\mathds{1}\mathds{1}^\top P + P\mathds{1}\mathds{1}^\top)+\epsilon^2 \mathds{1}\mathds{1}^\top P \mathds{1}\mathds{1}^\top\big)_{1:m-1}
	\label{pf:2_P_}
\end{align}
with $\tilde{P} \in \mathbb{R}^{(m-1) \times (m-1)}$, which can be written as
\begin{align}
\tilde{P} &= (\alpha_2-\frac{1}{2}(\alpha_1+\alpha_2)^2)
\nonumber \\
&\times 
\big[I+(\underbrace{m(m-1)\epsilon^2-\corrected{2m\epsilon}+1}_{=0})\mathds{1}\mathds{1}^\top  \big] 
\nonumber	\\
&\!\!+
\frac{1}{2}(\alpha_1+\alpha_2)^2
\begin{bmatrix}
0\! & \! d_1 \! &\! d_2 \!& \!\ldots\! &\! d_{m-2} \\
-d_1\! & \!\ddots\! & \!\ddots\! & \!\ddots\! & \!\vdots \\
-d_2\! & \!\ddots\! & \!\ddots & \!\ddots\! & \!d_2  \\
\vdots\! & \! \ddots \! & \!\ddots\! & \! \ddots\! & \!d_1 \\
-d_{m-2}\! & \!\cdots\! & \!-d_2\! & \!-d_1\! & \!0
\end{bmatrix}	
\label{pf:2_P_1}
\end{align}
with $d_i = 2i\epsilon-1$ for $i=1,\ldots,m-2$ and $\epsilon$ defined in \eqref{pf:construct_V}, where $\tilde{P}\tilde{P}^\top - \tilde{P}^\top \tilde{P} = 0$, hence $\tilde{P}$ normal. More precise, $\tilde{P}$ has complex conjugated eigenvalues $\mu_{\ell} \pm \omega_{\ell} i$ with $\ell = 1,\ldots,\lceil (m-1)/2\rceil$  where $\mu_\ell = \mu = \alpha_2-\nicefrac{1}{2}(\alpha_1+\alpha_2)^2$ and the skew-symmetric part is a Toeplitz matrix (third line of \eqref{pf:2_P_1}).
Eventually, \eqref{pf:2_9} impose the conditions
\begin{align}
	X_{1:r} &= \Sigma_0 \tilde{Q}_{1:r} \Sigma_0,
	\label{pf:2_cond}
	\\
	X_{r+1:n} &= 0.
	\label{pf:2_cond_1}
\end{align}
with $r=\text{rk}(T_d)$. Then, \eqref{pf:2_cond_1} holds, since there exist $n-r-1$ eigenvalues identical to zero and $U$ can be ordered, accordingly. 
Additionally, let $\Theta$ in \eqref{pf:construct_V} be of the form such that 
\begin{align}
\begin{split}
&\tilde{Q}_{1:r} = \text{diag}([D_1\ \cdots\ D_{\lceil r/2 \rceil} ])  \\
&\text{with } D_\ell = 
\begin{bmatrix}
\mu  & -\hat{\omega}_\ell \\ \hat{\omega}_\ell & \mu
\end{bmatrix},
\ell = 1,\ldots,\lceil r/2 \rceil
\end{split}
\label{pf:2_Y_1r}
\end{align}
holds, where the imaginary part of the eigenvalues of the principal submatrix $\tilde{Q}_{1:r}$ of $\tilde{Q}$ in \eqref{pf:2_Y} is denoted by $\pm \hat{\omega}_k i$ for $k=1,\ldots,\lceil r/2 \rceil$.
Then \eqref{pf:2_cond} implies for $k = 1,\ldots,\lceil r/2 \rceil$
\begin{align}
&\mu \sigma_{2k -1}^2 = \mu \sigma_{2k}^2 = \gamma_k  \quad \text{and} 
\label{pf:2_cond1}
\\
&\sigma_{2k-1}\sigma_{2k} \hat{\omega}_k = \delta_k  .
\label{pf:2_cond2}
\end{align}
In the case of $2\alpha_2 - (\alpha_1+\alpha_2)^2 = 0$, $T_d$ and $\tilde{P}$ are skew-symmetric due to the assumption in \cref{thm:existence_svd} and \eqref{pf:2_P_1}, respectively. Hence, $\gamma_{\ell} = 0$ for $\ell = 1,\ldots,n$, and $\mu = 0$, which implies that \eqref{pf:2_cond1} is satisfied. Equation \eqref{pf:2_cond2} is satisfied for $k=1,\ldots,\lceil r/2 \rceil$ with 
\begin{align}
 &\sigma_{2k} = \delta_k\omega_k^{-1}\sigma_{2k-1}^{-1} \quad \text{and} \quad  \sigma_{2k -1} \in\mathbb{R}_{>0}
\label{pf:2_cond_skew}
\end{align}
where $m=r+1$ and therefore $\hat{\omega}_k = \omega_k$. Then $\Theta$ is constructed as orthogonal transformation similar to $U$. 
In the case  $2\alpha_2 - (\alpha_1+\alpha_2)^2 \neq 0$, \eqref{pf:2_cond1} and \eqref{pf:2_cond2} together yield
\begin{align}
\sigma_{2k-1}^2 = \sigma_{2k}^2 = \frac{\delta_k}{\hat{\omega}_k} = \frac{\gamma_k}{\mu},
\label{pf:2_cond_normal}
\end{align}
for $k=1,\ldots,\lceil r/2 \rceil$, hence 
\begin{align}
\hat{\omega}_k = \frac{\delta_k}{\gamma_k}\mu
\label{pf:2_w_}
\end{align}
has to be satisfied. Note that $\delta_k,\ \gamma_k,$ and $\mu$ are specified by $T_d$ and $\alpha_1,\alpha_2$ and $\mu/\gamma_k \ge 0$ due to the positive definiteness condition in \Cref{thm:existence_svd}. Applying \cref{thm:fan_pall_mult} to $\tilde{P}$ in \eqref{pf:2_Y} implies that there exists a $\Theta$ such that \eqref{pf:2_Y_1r} and the interlacing property
\begin{align}
\omega_{k} \ge \hat{\omega}_k \ge \omega_{\lceil (m-1)/2 \rceil - \lceil r/2 \rceil +k},\ k=1,\ldots,\lceil r/2 \rceil
\label{pf:2_interlacing1}
\end{align}
holds, where $\hat{\omega}_k$ can be chosen in the given intervals. Note that \cref{thm:fan_pall_mult} can be applied to the normal matrix $\tilde{P}$, due to the decomposition of a scaled unit and skew-symmetric matrix \cite{connes1998noncommutative}.
W.l.o.g., $\delta_k/\gamma_k$ and $\hat{\omega}_k$ are in decreasing order. Then by applying \Cref{property:interlacing} (cf. \Cref{rem:pm}) to $\tilde{P}$ in \eqref{pf:2_P_} successively, there exists a $m \ge r+1$ such that \eqref{pf:2_interlacing1} hold with \eqref{pf:2_w_} for all $k = 1,\ldots,\lceil r/2 \rceil$.
${\color{white}oooooooooooo}\hfill \small{\square}$

\subsection{Proof \cref{thm:structure_W_H}}
\label{sec:p_sructure_W_H}
Let $2\alpha_2-(\alpha_1+\alpha_2)^2 = 0$ and $T(W)=T_d$ is partitioned as in \eqref{eq:structure_T}. Since $T_d$ has to be skew-symmetric (see also \Cref{thm:existence_svd}), the eigenvalues of $T_d$ are purely imaginary. Then condition \eqref{eq:problem} with $T_{12} = -T_{21}^\top = -R\in\mathbb{R}^{n\times n}$, where $R$ is arbitrary, reads
\begin{align}
	 - I &= g'(z)f(z)R^\top-f'(z)g(z)R \nonumber \\
	&+ f'(z)f(z) T_{11} + g'(z)g(z) T_{22}.
	\label{eq_p2:1}
\end{align}
Since, $T_{11} = - T_{11}^\top$, $T_{22} = -T_{22}^\top$, w.l.o.g. we can express $R$ as  $R = I + \tilde{R}$, where $\text{diag}(\tilde{R}) = 0$. Hence, it has to hold 
\begin{align}
	g'(z)f(z)-f'(z)g(z) = -1,
	\label{eq_p2:2}
\end{align} 
which is satisfied by the generating functions \cite[Theorem 1]{GRUSHKOVSKAYA2018151}
\begin{align}
	g(z) = -f(z)\int \frac{1}{f^2(z)} dz.
	\label{eq_p2:3}
\end{align}
Assume that generating functions $f$ and $g$ as in \eqref{eq_p2:3} satisfy \eqref{eq_p2:2}, then condition \eqref{eq_p2:1} translates into 
\begin{align}
    f'(z)f(z) T_{11}\! +\! g'(z)g(z) T_{22} \nonumber \\
	= (f'(z)g(z)+g'(z)f(z))\tilde{R},
	\label{eq_p2:3_1}
\end{align} 
implying that $\tilde{R} = -\tilde{R}^\top$ holds. 
Next we consider three cases:
\\\\
\textbf{Case 1: $\tilde{R} = 0$}. Hence,
\begin{align}
	f'(z)f(z) T_{11} + g'(z)g(z) T_{22} = 0.
	\label{eq_p2:4}
\end{align}
Clearly, \eqref{eq_p2:4} is satisfied by $T_{11} = T_{22} = 0$ with $f$ arbitrary while satisfying $g$ in \eqref{eq_p2:3}, i.e., \eqref{eq:cases_H1} results. 
For the sub-case $a^{-1}T_{11} = b^{-1} T_{22} =: Q$ arbitrary skew-symmetric with $a,b\in \mathbb{R}_{>0}$, $f$ and $g$ have to satisfy $af'(z)f(z)+b g'(z)g(z) = 0$, i.e., w.l.o.g. $af^2(z)+b g^2(z) = 1$. Accordingly, with \eqref{eq_p2:3} and $y'(z) = f^{-2}(z)$ it yields
\begin{align}
	y'(z) = a + b y^2(z).
	\label{eq_p2:5}
\end{align}
The unique solution of \eqref{eq_p2:5} with $\phi \in \mathbb{R}$ is
\begin{align}
	y(z) = \sqrt{\frac{a}{b}}\tan\Big(\sqrt{a b} z + \phi \Big)
	\label{eq_p2:5_2}
\end{align}
and therefore with the definition of $y(z)$ and \eqref{eq_p2:3} one reveals \eqref{eq:cases_H2}. 
Repeating the above calculations for $a^{-1} T_{11} = - b^{-1} T_{22} =: Q$ arbitrary skew-symmetric yields \begin{align}
	y(z) = \sqrt{\frac{a}{b}}\tanh\Big(\sqrt{ab} z + \phi \Big)
\end{align}
and therefore with the definition of $y(z)$ and \eqref{eq_p2:3} one reveals \eqref{eq:cases_H3}.  
The remaining sub-case $T_{11} \neq \pm a T_{22}$, implies that 
\begin{align}
	f'(z)f(z) T_{11} = 0\ \ \text{and}\ \ g'(z)g(z) T_{22} = 0,
	\label{eq_p2:5_1}
\end{align}
since $	f'(z)f(z) T_{11} + g'(z)g(z) T_{22} = 0$ must hold. 
If $f'(z)f(z) = 0$ and $T_{11}=Q$ arbitrary skew-symmetric, i.e., $f^2(z) = a$ with $a\in\mathbb{R}_{>0}$, it implies that $f(z) = \sqrt{a}$. Hence, with \eqref{eq_p2:3}, $g'(z)g(z) \neq 0$ for all $z\in\mathbb{R}$ yields $T_{22} = 0$, i.e. \eqref{eq:cases_H4} results.
The same argumentation holds for $g'(z)g(z) = 0$ with arbitrary skew-symmetric $T_{22}$ such that \eqref{eq:cases_H5} is recovered. 
The circumstance that  $T_{11} \neq T_{22}$ with $T_{11} \neq 0$ and $T_{22} \neq 0$ is not valid due to \eqref{eq_p2:3} and \eqref{eq_p2:5_1}. Specifically, $f'(z)f(z) = 0$ and $g'(z)g(z) = 0$ has to hold; obviously, based on the above cases, $f(z)=\sqrt{a}$ and $g(z)=\sqrt{a}$ are in conflict with \eqref{eq_p2:2}.
\\\\
%
%
\textbf{Case 2:} $f'(z)g(z)+g'(z)f(z) = 0$ for all $z\in \mathbb{R}$. Hence, 
\begin{align}
    &f'(z)f(z) T_{11}\! +\! g'(z)g(z) T_{22} = 0 \text{ and } 
    \label{eq_p2:6_00}
    \\
    &f'(z)g(z)+g'(z)f(z) = 0
    \label{eq_p2:6_01}
\end{align} 
has to be satisfied. Clearly, \eqref{eq_p2:6_00} is satisfied by $T_{11} = T_{22} = 0$, where \eqref{eq_p2:6_01} implies $-af(z)g(z)=1$ with $a \in \mathbb{R}_{>0}$. Accordingly, with \eqref{eq_p2:3} and $y'(z) = f^{-2}(z)$ it yields
\begin{align}
     y'(z) = ay(z) .
     \label{eq_p2:6_1}
\end{align} 
 The unique solution of \eqref{eq_p2:6_1} with $c\in\mathbb{R}$ is 
\begin{align}
	y(z) = e^{az}+c
	\label{eq_p2:7}
\end{align}
and therefore with the definition of $y(z)$ and \eqref{eq_p2:3} one gets \eqref{eq:cases_H6}.
The sub-cases where $a^{-1}T_{11} = b^{-1} T_{22} =: Q$ arbitrary skew-symmetric or  $T_{11} \neq \pm a T_{22}$ as discussed for Case 1 are not valid. With the same approach as above, i.e., $y'(z) = f^{-2}$ it yields to $y'(z) = 0$ and therefore no solution for $f(z)$ (and $g(z)$) can be found or \eqref{eq_p2:6_00} and \eqref{eq_p2:6_01} are not satisfied as discussed in the last paragraph of Case 1, respectively. 
\\\\
\textbf{Case 3: $\tilde{R} = -\tilde{R}^\top \neq 0$}. Hence, 
\begin{align}
    f'(z)f(z) T_{11}\! +\! g'(z)g(z) T_{22} \nonumber \\
	= (f'(z)g(z)+g'(z)f(z))\tilde{R}.
	\label{eq_p2:case3}
\end{align} 
Clearly, $T_{11} = T_{22} = 0$ is not valid, since $(f'(z)g(z)+g'(z)f(z))\tilde{R} \neq 0$ in this last case.
For the sub-case \corrected{$a^{-1}T_{11} = b^{-1} T_{22} =: Q = -Q^\top \neq 0$ with $a,b \in \mathbb{R}_{>0}$, it has to hold that $a f'(z)f(z) + b g'(z)g(z) = c(f'(z)g(z)+g'(z)f(z))$ and $\tilde{R} = cQ$ with $c \in \mathbb{R}\backslash{\{0\}}$, i.e., w.l.o.g. $\nicefrac{a}{2}f^2(z) + \nicefrac{b}{2} g^2(z) -c f(z)g(z)= 1$. Accordingly, with \eqref{eq_p2:3}} and $y'(z) = f^{-2}(z)$ it yields
\begin{align}
     y'(z) = \frac{a}{2} + cy(z) + \frac{b}{2} y^2(z).
	\label{eq_p2:6}
\end{align} 
The unique solution of \eqref{eq_p2:6} with $\phi \in \mathbb{R}$ is 
\begin{align}
	y(z) = \frac{\sqrt{ab-c^2}}{b} \tan\Big(\sqrt{ab-c^2} z + \phi \Big)
	\label{eq_p2:7}
\end{align}
and therefore with the definition of $y(z)$ and \eqref{eq_p2:3} one reveals \eqref{eq:cases_H7}.
The remaining sub-case $T_{11} \neq a T_{22}$, implies that
\begin{align}
&f'(z)f(z) T_{11} - (f'(z)g(z)+g'(z)f(z)) \tilde{R} = 0 \quad \text{and} 
\nonumber \\
&g'(z)g(z) T_{22} = 0, \quad \text{or} 
\label{eq_p2:8}
\\
&g'(z)g(z) T_{22} - (f'(z)g(z)+g'(z)f(z)) \tilde{R} = 0 \quad \text{and} 
\nonumber \\
&f'(z)f(z) T_{11} = 0, 
\label{eq_p2:9}
\end{align}
since \eqref{eq_p2:case3} must hold. However, we show in the sequel that \eqref{eq_p2:8} and \eqref{eq_p2:9} lead to no new solution or is not valid, respectively. 
For \eqref{eq_p2:8}, \corrected{$a^{-1}T_{11} = c^{-1} \tilde{R} =: Q = -Q^\top \neq 0$} with $a \in \mathbb{R}_{>0}$ and $c \in \mathbb{R}\backslash{\{0\}}$, it has to hold that \corrected{$a f'(z)f(z) = c(f'(z)g(z)+g'(z)f(z))$}, i.e., w.l.o.g. 
\begin{align}
    \frac{a}{2} f^2(z) -cf(z)g(z)= 1.
    \label{eq_p2:9_1}
\end{align}
Accordingly, with \eqref{eq_p2:3} and $y'(z) = f^{-2}(z)$ it yields
\begin{align}
     y'(z) = \frac{a}{2} + cy(z).
	\label{eq_p2:9_2}
\end{align} 
The unique solution of \eqref{eq_p2:9} with $d \in \mathbb{R}$ is 
\begin{align}
	y(z) = de^{cz}-\frac{a}{2c}
	\label{eq_p2:10}
\end{align}
and therefore with the definition of $y(z)$ and \eqref{eq_p2:3} one gets $f(z) = (cd)^{-\nicefrac{1}{2}}\text{exp}(-\nicefrac{c}{2}\, z)$ and $g(z) = -(d/c)^{-\nicefrac{1}{2}}\text{exp}(\nicefrac{c}{2}\, z)$ such that $T_{22} = 0$. However, $f,g$ satisfying \eqref{eq_p2:9_1} for all $z\in \mathbb{R}$ only for $a=0$, hence $T_{11}=0$, i.e, the same result as \eqref{eq:cases_H6}.
For \eqref{eq_p2:9}, \corrected{$b^{-1}T_{11} = c^{-1} \tilde{R} =: Q = -Q^\top \neq 0$} with $a \in \mathbb{R}_{>0}$ and $c \in \mathbb{R}\backslash{\{0\}}$, it has to hold that \corrected{$b g'(z)g(z) = c(f'(z)g(z)+g'(z)f(z))$}, i.e., w.l.o.g. 
\begin{align}
    \frac{b}{2} g^2(z) -cf(z)g(z)= 1.
    \label{eq_p2:10_1}
\end{align}
Following the procedure above yields $f(z) = (2)^{-\nicefrac{1}{2}} c^{-1} (b \text{exp}(\nicefrac{c}{2}\, z + \nicefrac{d}{2})-\text{exp}(-\nicefrac{c}{2}\, z - \nicefrac{d}{2}))$ and $g(z) = 2^{\nicefrac{1}{2}}b^{-1}\text{exp}(-\nicefrac{c}{2}z - \nicefrac{d}{2})$ with $d\in\mathbb{R}$. However, \eqref{eq_p2:10_1} is only satisfied for $d=-cz-\text{ln}(b)$ which leads to $f(z) = 0$ and therefore no valid solution. 
Notice that every feasible structure of the skew-symmetric matrix $T_d$ is discussed above case by case, and the differential equations arising in the analysis are solved uniquely. Hence, we believe that the list of triples $(T_d,f,g)$ in \cref{thm:structure_W_H} for $2\alpha_2-(\alpha_1+\alpha_2)^2=0$ and $T_d$ skew-symmetric is essentially exhaustive, save for some scaled version of the presented cases.
$\hfill \small{\square}$
\subsection{Proof \cref{thm:structure_W_E}}
\label{sec:p_sructure_W_E}
Let $T_d$ be normal and $(2\alpha_2-(\alpha_1+\alpha_2)^2)(T_d+T_d^\top)$ be positive definite and let $T(W)=T_d$ be partitioned as in \eqref{eq:structure_T}. Then equation \eqref{eq:problem} reads
\begin{align}
	- I &=f'(z)f(z) T_{11} + f'(z)g(z)T_{12} \nonumber \\
	&+g'(z)f(z)T_{21}  + g'(z)g(z) T_{22}.
	\label{eq_p3:1}
\end{align}
Choosing $T_d$ ($T(W)=T_d$) as in \eqref{eq:cases_E1}, it holds that  $(2\alpha_2-(\alpha_1+\alpha_2)^2)(T_d+T_d^\top)$ is positive definite, yielding 
\begin{align}
	- 1 &= a \Big(f'(z)f(z) + g(z)'g(z)\Big)
	\nonumber 
	\\
	&+  g'(z)f(z)-f'(z)g(z).
	\label{eq_p3:2}
\end{align}
This equation has been considered in \cite{jfESNCM}. We refer to the proof of \cite[Theorem 1]{jfESNCM} for the derivation of $f$ and $g$ as specified in \eqref{eq:cases_E1}.
Case \eqref{eq:cases_E2} is analogous to \eqref{eq:cases_H2}. First, note that $(2\alpha_2 - (\alpha_1+\alpha_2))(T_d+T_d^\top)$ is positive definite and normal with the given $T_d$ in \eqref{eq:cases_E2} and $(2\alpha_2 - (\alpha_1+\alpha_2))(Q+Q^\top)$ is positive definite and normal, since 
\begin{align}
    T_dT_d^\top - T_d^\top T_d = \begin{bmatrix} QQ^\top - Q^\top Q & 0 \\ 0 & QQ^\top - Q^\top Q  \end{bmatrix} 
\end{align}
and 
\begin{align}
    T_d + T_d^\top = \begin{bmatrix} Q + Q^\top  & 0 \\ 0 & Q + Q^\top  \end{bmatrix}, 
\end{align}
such that the real part of the eigenvalues of $Q$ is identical to that of $T_d$ and therefore the definiteness property is conserved.
Hence, the derivation of the generating functions $f$ and $g$ in this case goes along the lines of arguments as used in the proof of \cref{thm:structure_W_H} in \cref{sec:p_sructure_W_H}, specifically for the case \eqref{eq:cases_H2}, i.e. \eqref{eq_p2:5} and \eqref{eq_p2:5_2}.
$\hfill \small{\square}$
The structure of normal matrices $T_d$ brings more degrees of freedom compared to $T_d$ skew-symmetric as in \Cref{thm:structure_W_H} and thus the two cases listed in \cref{thm:structure_W_E} are not exhaustive.
\section{CONSTRUCTION OF EXPLORATION SEQUENCE MATRIX}
In this section we summarize the main steps to construct an exploration sequence matrix $W$ based on the constructive proof of \Cref{thm:existence_svd} in \Cref{sec:proof_existence_svd}. A \texttt{MATLAB} toolbox of the described construction procedure in the
sequel can be found in the ancillary file folder on Arxiv.

\subsection{Step-by-step Construction of $W$}
\label{sec:construction_svd}
The exploration sequence matrix $W\in\mathbb{R}^{2n\times m}$ with dimension of the optimization variable $n$ and exploration sequence length $m$ is constructed by a singular-value decomposition 
\begin{align}
	W = U\Sigma V^\top
	\label{algo_ES:svd}
\end{align}
with $U\in\mathbb{R}^{2n\times 2n}$, $\Sigma\in\mathbb{R}^{2n\times m}$, and $V \in \mathbb{R}^{m\times m}$, where $m$ has to be determined. In the following $U,\Sigma,V$ as stated in \eqref{pf:construct_U} to \eqref{pf:construct_V} are constructed.
%

\textbf{Step 1}: \textit{Choose design parameters.} Select the map parameters $\alpha_1,\alpha_2\in\mathbb{R}$ with $\alpha_1+\alpha_2 \neq 0$, and $T_d \in \mathbb{R}^{2n\times 2n}$ according to  \cref{thm:structure_W_H} and \cref{thm:structure_W_E} with $r = \text{rk}(T_d)$. 
%

%
\textbf{Step 2}: \textit{Calculate eigenvalues and eigenvectors of $T_d$}.
Calculate eigenvalues of $T_d$ with $\gamma_{\ell}\pm\delta_{\ell}i$ where $\gamma_\ell \in \mathbb{R},\ \delta_{\ell}\in \mathbb{R}_{\ge 0}$ and eigenvectors of $T_d$ with $a_{\ell}\pm b_{\ell}i$ where $a_\ell,b_\ell \in \mathbb{R}^{2n}$ for $\ell = 1,\ldots,n$. The eigenvalues
are sorted according to $\delta_1 \ge \delta_2 \ge \cdots \ge \delta_n \ge 0$. Note that $\gamma_{\ell} = 0$ if $T_d$ skew-symmetric and $\gamma_{p} = 0$, $\delta_{p} = 0$ for $p = \lceil r/2\rceil+1,\ldots,n$.
%

%
\textbf{Step 3}: \textit{Calculate $U$}.
Construct as specified in \eqref{pf:construct_U}. 
%

%
\textbf{Step 4}: \textit{Design $\Sigma$.} 
The principal submatrix of $\Sigma$ is constructed as $\Sigma_0 = \text{diag}(\sigma_1,\ldots,\sigma_r) \in \mathbb{R}^{r\times r}$ with the singular values $\sigma_k$ of $W$. This step has various degrees of freedom to influence the sequence length $m$; distinguished in the following: 
\begin{enumerate}[leftmargin=*, label=\Roman*.]
    \item $T_d$ skew-symmetric
    \begin{enumerate}[label=\roman*), leftmargin=15pt]
        \item $m=r+1$
        \item $m\ge r+1$ ($m$ determined in Step 5)
    \end{enumerate}
    \item $T_d$ as in \eqref{eq:Td_E} 
    \begin{enumerate}[label=\roman*), leftmargin=15pt]
        \item $m=2n+1$ 
        \item $m \ge 2n+1$ ($m$ determined in Step 5)
    \end{enumerate}
    \item $T_d$ normal and $(2\alpha_2-(\alpha_1+\alpha_2)^2)(T_d+T_d^\top)$ positive definite ($m$ determined in Step 5)
\end{enumerate}
For I.i) and II.i), calculate the eigenvalues of $\tilde{P}(m)$ in \eqref{eq:P_tilde} with $\mu\pm \omega_k i$ for $k=1,\ldots,\lceil m/2 \rceil$, $\mu \in \mathbb{R},\omega_k \in \mathbb{R}_{\ge 0}$ and $\omega_1\ge \omega_2 \ge \cdots \ge \omega_{\lceil m/2 \rceil}\ge 0$. 
The choices of $\sigma_{2\ell-1},\sigma_{2\ell}$ for $\ell = 1,\ldots,\lceil r/2 \rceil$ for each case above is presented in the following:
\begin{enumerate}[leftmargin=24pt]
    \item[I.i)] $\sigma_{2\ell} = \delta_\ell\omega_\ell^{-1}\sigma_{2\ell-1}^{-1} \quad \text{and} \quad  \sigma_{2\ell -1} \in\mathbb{R}_+$
    \item[I.ii)] $\sigma_{q} \in \mathbb{R}_{+},\ q=1,\ldots,r $
    \item[II.i)] $\sigma_{2\ell -1} = \sigma_{2\ell} = \omega_\ell^{-1/2}$
    \item[II.ii)] $\sigma_{2\ell -1} = \sigma_{2\ell} \in \mathbb{R}_{+} $
    \item[III.] $\sigma_{2\ell -1} = \sigma_{2\ell} = \gamma_\ell^{1/2}(\alpha_2-(\alpha_1+\alpha_2)^2)^{-1/2}$
\end{enumerate}
Note that $\Sigma$ depends on $m$, i.e, for I.ii), II.ii), and III. $m$ has to be determined (see Step 5) first. Then, and for cases I.i) and II.i)  $\Sigma=\text{diag}([\Sigma_0 \ 0_{m-r}])$ can be constructed.
%

%
\textbf{Step 5}: \textit{Determine sequence length $m$.}
If in Step 4, $\Sigma$ was constructed based on I.i) or II.i), $m=r+1$ or $m=2n+1$, respectively (proceed directly with Step 6). Otherwise, calculate
\begin{align}
	\hat{\omega}_{\pi(\ell)} = \delta_{\ell}\big(\sigma_{2\ell-1}\sigma_{2\ell}\big)^{-1}
	\label{algo_ES:w_skew}
\end{align} 
with $\sigma_{2\ell-1},\sigma_{2\ell}$ as designed in Step 4 for $\ell = 1,\ldots,\lfloor r/2 \rfloor$ and  permutation $\pi:\{1,\ldots,\lfloor r/2 \rfloor\}\rightarrow\{1,\ldots,\lfloor r/2 \rfloor\}$ such that $\hat{\omega}_{1} \ge \hat{\omega}_{2} \ge \cdots \hat{\omega}_{\lfloor r/2 \rfloor} \ge 0$ hold. Construct the  permutation matrix $\hat{R}(\pi)\in\mathbb{R}^{\lfloor r/2 \rfloor \times \lfloor r/2 \rfloor}$
as
\begin{align}
	\hat{R}(\pi) = \begin{bmatrix}
	    e_{2\pi(1)-1}\!\! &\!\! e_{2\pi(1)}\!\!&\!\!\cdots\!\!&\!\!e_{2\pi(r)-1}\!\! &\!\! e_{2\pi(r)}
	\end{bmatrix}.
\end{align}
Then, set $\hat{m}=r+1$ and apply the following steps:
\begin{enumerate}
	\item[(a)] Calculate $\tilde{P}(\hat{m})$ as defined in \eqref{eq:P_tilde}.
	\item[(b)] Calculate eigenvalues of $\tilde{P}(\hat{m})$ with $\mu \pm \omega_{k}i$ where $\mu \in \mathbb{R}, \omega_k \in \mathbb{R}_{\ge 0}$ for $k=1,\ldots,\lceil \hat{m}/2\rceil$ and $\omega_{1}\ge \omega_{2} \ge \cdots \ge \omega_{\lceil \hat{m}/2 \rceil}$.
	\item[(c)]   Check if the interlacing property 
	\begin{align}
		\omega_k \ge \hat{\omega}_{k} \ge \omega_{\lceil \hat{m}/2 \rceil - \lfloor r/2\rfloor+k}
		\label{algo_ES:interlacing}
	\end{align}
	for every $k=1,\ldots,\lfloor r/2 \rfloor$ is satisfied with $\hat{\omega}_{k}$ calculated in \eqref{algo_ES:w_skew}.
	\item[(d)] If (c) is true, $m = \hat{m}$. Otherwise, $\hat{m} \leftarrow \hat{m}+1$ and goto (a).
\end{enumerate}
%

%
\textbf{Step 7}: \textit{Calculate $V$. } 
Construct $V$ as specified in \eqref{pf:construct_V}. Therein, the required orthogonal matrix $\Theta\in\mathbb{R}^{(m-1)\times (m-1)}$ is calculated as
\begin{align}
    \Theta = \begin{bmatrix}\hat{R}(\pi) & 0 \\ 0 & I \end{bmatrix} \tilde{\Theta}
\end{align}
with $\tilde{\Theta}\in\mathbb{R}^{(m-1)\times (m-1)}$ based on the construction procedure for principle submatrices below (cf. \Cref{sec:fan_pall}), where in there 
\begin{align}
    &C = \tilde{P}(m) - I (\alpha_2-(\alpha_1 + \alpha_2)^2)
\end{align}
with $\tilde{P}(m)$ in \eqref{eq:P_tilde} and $[\hat{\omega}_{k}]_{k=1,\ldots,\lfloor r/2 \rfloor}$ calculated in \eqref{algo_ES:w_skew} has to be applied.
%

%
\textbf{Step 8}: \textit{Determine $W$. } 
Finally, construct $W$ according to \eqref{algo_ES:svd}.

\subsection{Construction Procedure of Principal Submatrix}
\label{sec:fan_pall}

In this section we present a procedure to construct an orthogonal matrix $\Theta \in \mathbb{R}^{p \times p}$, such that for a given skew-symmetric matrix $C \in \mathbb{R}^{p \times p}$ with eigenvalues $\pm \eta_\ell i,\ \eta_\ell \in \mathbb{R}_{\ge 0}$ for $\ell=1,\ldots,\lfloor p/2 \rfloor$, and for given values $[\hat{\omega}_{k}]_{k=1,\ldots,q}$ ($p\ge 2q$), which are satisfying the interlacing inequalities 
\begin{align}
\eta_k \ge \hat{\omega}_k \ge \eta_{\lceil p/2 \rceil - q +k},\quad k = 1,\ldots,q;
\label{alg:interlacing1}
\end{align}
it holds that
\begin{align*}
&\Theta^\top C \Theta = \begin{bmatrix} Q & * \\ * & * \end{bmatrix} \quad \text{with}  \\
&Q = \text{diag}([Q_1\ Q_2\ \ldots Q_q]),\quad Q_k = \begin{bmatrix} 0 & -\hat{\omega}_k \\ \hat{\omega}_k & 0 \end{bmatrix}.
\end{align*}
Thus, $Q$ is the principal submatrix of $C$, where $q$ and $[\hat{\omega}_{k}]_{k=1,\ldots,q}$ can be chosen w.r.t. \eqref{alg:interlacing1}.
The main procedure to construct $\Theta$ is given in \Cref{alg:theta_main} (p. \pageref{alg:theta_main}) as an iterative algorithm. In each iteration $j=1,\ldots,\lceil p/2 \rceil -q$, a $\Theta_j \in \mathbb{R}^{p \times p}$ is constructed, described in the sub-routine \Cref{alg:theta_sub} (p. \pageref{alg:theta_sub}), such that 
\begin{align}
(\Theta_1\Theta_2\cdots\Theta_j)^\top C \Theta_1\Theta_2\cdots\Theta_j = \begin{bmatrix}
D_j & * \\ * & *
\end{bmatrix}
\label{alg:org_trans}
\end{align}
is satisfied, where $D_j \in \mathbb{R}^{2\lfloor (p-j)/2\rfloor \times 2\lfloor (p-j)/2\rfloor}$ is a block diagonal skew-symmetric matrix with eigenvalues $\pm \nu_{k}i,\ \nu_k\in\mathbb{R}_{\ge 0}$, which are satisfying the interlacing inequalities
\begin{align}
\eta_k \ge \nu_k \ge \eta_{j +k},
\end{align}
for $k=1,\ldots,\lfloor p/2 \rfloor - j$, where $D_j$ is determined in the sub-routine \Cref{alg:PS} (p. \pageref{alg:PS}).

In particular, the sub-routine given in \Cref{alg:theta_sub} (p. \pageref{alg:theta_sub}), constructs a $\hat{\Theta} \in \mathbb{R}^{r\times r}$ such that for two given skew-symmetric matrices $D_1\in\mathbb{R}^{r \times r}$ and $D_2\in\mathbb{R}^{s \times s}$ in block diagonal form, where  $s = 2\lfloor (r-1)/2\rfloor$ (hence $s$ always even), it holds that
\begin{align}
\hat{\Theta} D_1 \hat{\Theta}^\top = \begin{bmatrix}D_2 & * \\ * & *\end{bmatrix},
\end{align}
and the eigenvalues $\pm \delta_k i, \delta_k\in\mathbb{R}_{\ge 0}$ for $k=1,\ldots,\lceil r/2 \rceil$ 
of $D_1$
and the eigenvalues  $\pm \zeta_k i, \zeta_k\in\mathbb{R}_{\ge 0}$ for $k=1,\ldots, s/2 $ 
of $D_2$ 
satisfy the interlacing inequality 
\begin{align}
\delta_1 \ge \zeta_1 \ge \delta_2 \ge \zeta_2 \ge \ldots \ge \zeta_{s/2} \ge \delta_{\lceil r/2 \rceil} \ge 0.
\end{align}
The sub-routine given in \Cref{alg:PS} (p. \pageref{alg:PS}) constructs a block diagonal skew-symmetric matrix $\hat{D} \in \mathbb{R}^{2(\lceil t/2 \rceil - 1) \times 2(\lceil t/2 \rceil -1)}$ with eigenvalues $\pm \nu_j i,\ \nu_j\in\mathbb{R}_{\ge 0}$, $j=1,\ldots,\lceil t/2 \rceil -1$, such that for a given skew-symmetric matrix $D\in \mathbb{R}^{t \times t}$ with eigenvalues $\pm \gamma_j i,\ \gamma_j\in\mathbb{R}_{\ge 0}$, $j=1,\ldots,\lceil t/2 \rceil$,  the interlacing inequality
\begin{align}
\gamma_1 \ge \nu_1 \ge \ldots \ge \nu_{\lceil t/2\rceil -1} \ge \gamma_{\lceil t/2 \rceil} \ge 0.
\end{align}
is satisfied. Additionally, the interlacing inequalities 
\begin{align}
\gamma_k \ge \hat{\omega}_k \ge \gamma_{\lceil t/2 \rceil-q+k},\quad k = 1,\ldots,q
\label{alg:interlacing2} 
\end{align}
hold.
Summarizing, \Cref{alg:PS} (p. \pageref{alg:PS}) computes a principle submatrix of a dimension that is
two (or one in the first iteration) less than in the previous iteration in \Cref{alg:theta_main} (p. \pageref{alg:theta_main}), while the interlacing property \eqref{alg:interlacing1} is preserved by \eqref{alg:interlacing2}. Then, \Cref{alg:theta_sub} (p. \pageref{alg:theta_sub}) constructs a $\Theta_j$ (iteration $j=1,\ldots,\lceil p/2 \rceil - q$) based on $D_j$, calculated in \Cref{alg:PS} (p. \pageref{alg:PS}), such that the computed principle submatrix is obtained by an orthogonal transformation as written in \eqref{alg:org_trans}. This is repeated until the principal submatrix is the block diagonal matrix with eigenvalues $[\hat{\omega}_{k}]_{k=1,\ldots,q}$.
Note that the function $\text{eigVal}(\cdot)$ in \Cref{alg:PS} and \Cref{alg:theta_sub} (p.\pageref{alg:PS}f) computes the eigenvalue of a matrix in decreasing order w.r.t to the imaginary part (in \Cref{alg:theta_sub} (p. \pageref{alg:theta_sub}) only skew-symmetric matrices are present). 
%
%

%
\begin{algorithm}[h]
	\caption{Calculate $\Theta$}
	\begin{algorithmic}[1]
		\State \textbf{Given:} $C \in \mathbb{R}^{p \times p}, [\hat{\omega}_{k}]_{k=1,\ldots,q}$
		\If{$p = 2q$}
		\State $[a_\ell \pm b_\ell i]_{\ell=1,\ldots,p/2} = \text{eigVec(C)}$
		\State $\Theta_1 = [a_1\ b_1\ a_2\ b_2\ \cdots \ a_{p/2}\ b_{p/2}]$
		\Else{\ $(p > 2q)$}
		\State $D_0 = C$
		\For{$j = 1,\ldots,\lceil p/2 \rceil-q$}
		\State $ D_j = \texttt{calcPSMatrix} (D_{j-1}, [\hat{\omega}_{k}]_{k=1,\ldots,q})$
		\State $\hat{\Theta}_j = \texttt{calcThetaSub} (D_{j-1},D_{j})$
		\State $\Theta_j = \begin{bmatrix} \hat{\Theta}_j & 0 \\ 0 & I \end{bmatrix} \in \mathbb{R}^{p \times p}$
		\EndFor
		\EndIf
		\State $\Theta \leftarrow \Theta_1\Theta_2\cdots\Theta_{\lceil p/2\rceil -q}$
	\end{algorithmic}
	\label{alg:theta_main}
\end{algorithm}
\begin{algorithm}[h]
	\caption{Sub-routine: calculate principal submatrix}
	\begin{algorithmic}[1]
		\Function{$\hat{D} = \texttt{calcPSMatrix}$}{$D,[\hat{\omega}_{k}]_{k=1,\ldots,q}$}
		\State $[\pm \gamma_k i]_{k=1,\ldots,\lceil t/2 \rceil} = \text{eigVal}(D)$)
		\For{$j = 1,\ldots,\lceil t/2 \rceil-2$} 
		\State $\rho = \{\hat{\omega} \in [\hat{\omega}_k]_{k=1,\ldots,q}:\gamma_j \le \hat{\omega} \le \gamma_{j+1}\}$
		\State $\nu_j =\max\{ \gamma_{j+1}, \rho \} $
		\EndFor
		\State $\nu_{\lceil t/2 \rceil-1} = \max \big(\hat{\omega}_{q},\gamma_{\lceil t/2 \rceil-2} \big)$
		\For{$j = 1,\ldots,\lceil t/2 \rceil-1$} 
		\State $N_j = \begin{bmatrix} 0 & -\nu_j \\ \nu_j & 0 \end{bmatrix}$
		\EndFor
		\State $\hat{D} = \text{diag}([N_1\ N_2\ \cdots \ N_{\lceil t/2 \rceil-1}]) $ 
		\EndFunction
	\end{algorithmic}
	\label{alg:PS}
\end{algorithm}
\begin{algorithm}[h]
	\caption{Sub-routine: calculate $\hat{\Theta}$}
	\begin{algorithmic}[1]
		\Function{$\hat{\Theta}$ = $\texttt{calcThetaSub}$}{$D_1,D_2$}
		\State $[\pm \delta_k i]_{k=1,\ldots,\lceil r/2 \rceil} = \text{eigVal}(D_1)$)
		\State $[\pm \zeta_k i]_{k=1,\ldots, s/2 } = \text{eigVal}(D_2)$)
		\For{$j = 1, \ldots, \lceil r/2 \rceil -1$} 
		\State $x_{j} = \Big(\displaystyle\prod_{i=1}^{\lfloor r/2 \rfloor}(\zeta_j^2-\delta_i^2)\Big)$,\quad $y_j = \Big(2\displaystyle\prod_{\stackrel{i=1}{i \neq j}}^{\lceil r/2 \rceil -1}(\zeta_j^2-\zeta_i^2)\Big)$
		\EndFor
		\If{$r$ even}
		\For{$j = 1, \ldots, r/2 -1$} 
		\If{$y_j == 0$}
		\State $z_{2j-1} = 0$
		\State $z_{2j} = 0$
		\Else
		\State $z_{2j-1} = (-x_j y_j^{-1} \nu_j^{-1})^{1/2}$
		\State $z_{2j} = z_{2j-1}$
		\EndIf
		\EndFor
		\State $z_{r-1} = \Big(\displaystyle \prod_{i=1}^{r/2} \delta_i\Big)\Big(\displaystyle \prod_{i=1}^{r/2-1} \zeta_i^{-1} \Big)$
		\State $z = [z_1\ z_2\ \cdots\ z_{r-2}]^\top$
		\State $\bar{D}_1 = 
		\begin{bmatrix}
		D_2 & 0 & z \\
		0 & 0 & z_{r-1} \\
		-z^\top & -z_{r-1} & 0 
		\\
		\end{bmatrix}$
		\State $[a_\ell \pm b_\ell i]_{\ell = 1,\ldots,r/2} = \text{eigVec}(D_1)$
		\State $[c_\ell \pm d_\ell i]_{\ell = 1,\ldots,r/2} = \text{eigVec}(\bar{D}_1)$
		\State $\Theta = [a_1\ b_1\ a_2\ b_2 \cdots \ a_{r/2} \ b_{r/2}]$
		\State $\bar{\Theta} = [c_1\ d_1\ c_2\ d_2 \cdots \ c_{r/2} \ d_{r/2}]$ 
		\Else {\ ($r$ odd)}
		\For{$j = 1, \ldots, \lfloor r/2 \rfloor $} 
		\If{$y_j == 0$}
		\State $z_{2j-1} = 0$
		\State $z_{2j} = 0$
		\Else
		\State $z_{2j-1} = (-x_j y_j^{-1})^{1/2}$
		\State $z_{2j} = z_{2j-1}$
		\EndIf
		\EndFor
		\State $z = [z_1\ z_2\ \cdots\ z_{ r-1}]^\top$
		\State $\bar{D}_1 = 
		\begin{bmatrix}
		D_2 & z \\
		-z^\top & 0 
		\\
		\end{bmatrix}$
		\State $[a_\ell \pm b_\ell i]_{\ell = 1,\ldots,\lceil r/2\rceil } = \text{eigVec}(D_1)$
		\State $[c_\ell \pm d_\ell i]_{\ell = 1,\ldots,\lceil r/2\rceil} = \text{eigVec}(\bar{D}_1)$
		\State $\Theta = [a_1\ b_1\ a_2\ b_2 \cdots \ a_{\lfloor r/2\rfloor} \ b_{\lfloor r/2\rfloor}\ a_{\lceil r/2\rceil}]$
		\State $\bar{\Theta} = [c_1\ d_1\ c_2\ d_2 \cdots \ c_{\lfloor r/2\rfloor} \ d_{\lfloor r/2\rfloor}\ c_{\lceil r/2\rceil}]$
		\EndIf
		\State $\hat{\Theta} = \Theta \bar{\Theta}$
		\EndFunction
	\end{algorithmic}
	\label{alg:theta_sub}
\end{algorithm}

\begin{rem}
	Note that \Cref{alg:theta_sub} (p. \pageref{alg:theta_sub}) is the construction procedure of the sufficient interlacing eigenvalue statement in \Cref{thm:fan_pall}, whereas \Cref{alg:theta_main} (p. \pageref{alg:theta_main}) is the construction procedure of \Cref{thm:fan_pall_mult}. 
	To verify this constructive approach we follow the lines of the proof of \cite[Lemma B.3.]{marshall1979inequalities}, in which symmetric matrices that satisfy the interlacing property \eqref{eq:fan_pall1} are considered. Because of minor changes in the proof, we present only the case for $p$ odd; thus, we construct $\bar{D}_1 \in \mathbb{R}^{p\times p}$ with $D_2\in\mathbb{R}^{(p-1)\times(p-1)}$ and $z\in\mathbb{R}^{p-1}$ as specified in \Cref{alg:theta_sub} (p. \pageref{alg:theta_sub}). Then
	\begin{align}
	g(\lambda) :=& \det(\lambda I - \bar{D}_1) = \prod_{i=1}^{\lfloor p/2\rfloor}(\lambda^2+\delta_i^2),
	\\
	f(\lambda) :=& \det(\lambda I - D_2) = \prod_{i=1}^{\lfloor p/2\rfloor}(\lambda^2+\nu_i^2)
	\nonumber \\ 
	=& g(\lambda)(\lambda + z^\top (\lambda I - \bar{D}_1)^{-1}z) 
	\nonumber \\
	=& \lambda g(\lambda)\Big(1 + \sum_{i=1}^{\lfloor p/2\rfloor}\frac{z_{2i-1}^2+z_{2i}^2}{\lambda^2+\delta_i^2}\Big)
	\end{align}
	hold. Let for $j=1,\ldots,\lfloor p/2 \rfloor$
	\begin{align}
	\!z_{2j-1}^2 = z_{2j}^2 = - \Big(\displaystyle\prod_{i=1}^{\lfloor p/2\rfloor}(\delta_j^2-\nu_i^2)\Big)\Big(2\displaystyle\prod_{\stackrel{i=1}{i \neq j}}^{\lceil p/2\rceil-1}(\delta_j^2-\delta_i^2)\Big)^{-1}
	\end{align}
	then with $\lambda = \pm \nu_{\ell}i$, $\ell = 1,\ldots,\lfloor p/2\rfloor$
	\begin{align}
	\sum_{i=1}^{\lfloor p/2\rfloor}\frac{z_{2i-1}^2+z_{2i}^2}{\lambda^2+\delta_i^2} = - \sum_{i=1}^{\lfloor p/2\rfloor}  \frac{\displaystyle \prod_{\stackrel{k=1}{k \neq \ell}}^{\lfloor p/2\rfloor}(\delta_i^2-\nu_k^2) }{\displaystyle \prod_{\stackrel{k=1}{k \neq i}}^{\lfloor p/2\rfloor}(\delta_i^2-\delta_k^2)} = -1.
	\end{align}
	Thus, $f(\pm \nu_{\ell}i)=0$, $\ell = 1,\ldots,\lfloor p/2\rfloor$ and therefore $\pm \nu_{\ell}i$ are the eigenvalues of $\bar{D}_1$, while the eigenvalues of the principal submatrix $D_2$ are $\pm\delta_{\ell}i$, $\ell = 1,\ldots,\lfloor p/2\rfloor$. 
\end{rem}

\end{document}